\documentclass[12pt]{amsart}

\usepackage{amssymb}
\usepackage{amsmath}
\usepackage{bm}
\usepackage{bbm}
\usepackage{enumitem}
\usepackage{float}
\usepackage{makecell}
\usepackage{multirow}
\usepackage{stmaryrd}
\usepackage{mathrsfs}
\usepackage[cmtip,all]{xy}
\usepackage[colorlinks=true,linkcolor=blue,citecolor=blue,urlcolor=blue, pagebackref=true]{hyperref}
\ifpdf
\usepackage[pdftex,final]{graphicx}
\usepackage{subfigure}
\usepackage[pdftex,lmargin=1in,rmargin=1in,tmargin=1in,bmargin=1in]{geometry}

\usepackage{tikz}
\usetikzlibrary{matrix,arrows,arrows.meta}
\usepackage{tikz-cd}

\title[Handle Decompositions and the 1D inputs Skein Lasagna Module]{Handle Decompositions and the 1-Dimensional Inputs Skein Lasagna Module}
\author{Imogen Montague}
\address{The University of Texas at Austin, Austin, TX 78712}
\email{imogen.montague@austin.utexas.edu}
\author{Ian Sullivan}
\address{The University of California, Davis, Davis, CA 95616}
\email{iasullivan@ucdavis.edu}

\usepackage{hyperref}
\hypersetup{
    colorlinks=true,
    linkcolor=magenta,
    citecolor=magenta,      
    urlcolor=magenta,
}
\usepackage{amsmath,amssymb,amsthm}
\usepackage{tikz}
    \usetikzlibrary{cd} 
\usepackage{tikz-cd}
\allowdisplaybreaks
\usepackage{comment}
\usepackage{todonotes}
\usepackage{stmaryrd} 
\SetSymbolFont{stmry}{bold}{U}{stmry}{m}{n}
\usepackage{MnSymbol} 
\usepackage{textcomp} 
\usepackage{bbm}
\usepackage{enumitem}

\theoremstyle{definition}
\newtheorem{theorem}{{Theorem}}[section]
\newtheorem{lemma}[theorem]{{Lemma}}
\newtheorem{proposition}[theorem]{{Proposition}}

\newtheorem{corollary}[theorem]{{Corollary}}
\newtheorem{example}[theorem]{{Example}}
\newtheorem{conjecture}[theorem]{{Conjecture}}

\newtheorem{question}{{Question}}

\theoremstyle{definition}
\newtheorem{remark}[theorem]{Remark}
\newtheorem{definition}[theorem]{{Definition}}

\newcommand{\doublearrow}[2]{%
  \mathrel{\begin{tikzpicture}[baseline=-3pt]
    \draw[-{Stealth[angle=30:5pt]}] (0,0.6ex) -- (2em,0.6ex) node[midway,above] {$\scriptstyle #1$};
    \draw[-{Stealth[angle=30:5pt]}] (0,-0.6ex) -- (2em,-0.6ex) node[midway,below] {$\scriptstyle #2$};
  \end{tikzpicture}}%
}
\newcommand{\belt}[1]{\widetilde{\mathbbm{1}}({#1})}

\DeclareMathOperator{\im}{im}
\DeclareMathOperator{\id}{id}

\DeclareMathOperator{\std}{std}

\newcommand{\del}{\partial}
\newcommand{\wt}[1]{\widetilde{#1}}

\newcommand{\<}{\langle}
\renewcommand{\>}{\rangle}
\newcommand{\ul}[1]{\underline{#1}}
\newcommand{\ol}[1]{\overline{#1}}

\renewcommand{\id} { \mathrm{id}}

\newcommand{\threepartdef}[6]
{
	\left\{
		\begin{array}{ll}
			#1 & \mbox{if } #2 \\
			#3 & \mbox{if } #4 \\
            #5 & \mbox{if } #6
		\end{array}
	\right.
}

\newcommand{\bb}[1]{\mathbb{#1}}
\newcommand{\bbm}[1]{\mathbbm{#1}}
\newcommand{\mbf}[1]{\mathbf{#1}}
\newcommand{\cal}[1]{\mathcal{#1}}

\renewcommand{\sf}[1]{\mathsf{#1}}
\renewcommand{\frak}[1]{\mathfrak{#1}}

\DeclareMathOperator{\op}{op}
\DeclareMathOperator{\Tr}{Tr}

\newtheorem{thm}{Theorem}[section]

\newtheorem{prop}[thm]{Proposition}

\theoremstyle{definition}

\newtheorem{clm*}{Claim}

\theoremstyle{remark}

\makeatletter
\let\c@equation\c@thm
\makeatother
\numberwithin{equation}{section}

\newcommand{\idmap}{\mathrm{id}}  

\newcommand{\Kh}{\mathrm {Kh}}

\newcommand{\KhR}{\mathrm{KhR}} 
\newcommand{\Links}{\textbf{Links}} 

\newcommand{\fdVect}{fd\textbf{Vect}} 
\newcommand{\fVect}{f\textbf{Vect}} 
\newcommand{\skein}{\mathcal{S}}

\newcommand{\FT}{\mathrm{FT}} 

\usepackage{todonotes}

\begin{document}

\begin{abstract}
We establish handle attachment formulas for the Khovanov skein lasagna module with 1-dimensional inputs over $\mathbb{Q}$, defined recently by Ren, Wedrich, Willis, Zhang, and the second author. For a $4$-manifold built out of $1$- and $2$-handles, the invariant can be computed in terms of a cabled colimit of Rozansky-Willis homologies, modulo a new relation which we call the \emph{lasso} relation. We then present some explicit calculations for disk bundles over $S^{2}$, as well as a partial vanishing result for $4$-manifolds of the form $\Sigma_{g}\times D^{2}$, $g\geq 1$.
\end{abstract}

\maketitle

\section{Introduction}

In \cite{MWW-lasagna}, Morrison--Walker--Wedrich introduced the (Khovanov) skein lasagna module, a $4$-manifold invariant constructed using a version of $\frak{gl}_{2}$ Khovanov-Rozansky homology for links in $S^{3}$. More precisely, given a compact oriented $4$-manifold $X$ and a (possibly empty) framed oriented link $L\subset\del X$, the authors of \cite{MWW-lasagna} associate to the pair $(X,L)$ a triply-graded $\bb{Q}$-module denoted by $\skein_{0}^{2}(X;L)$. In recent years this invariant and its various extensions have been shown to detect exotic phenomena in dimension $4$ in a purely algebraic manner \cite{RW24,Nahm25,Sul25}.

The invariant also enjoys robust gluing properties \cite{MN22,MWWhandles,BKL25,MSW26} for facilitating computations. In particular, Manolescu--Walker--Wedrich \cite{MWWhandles} provide formulas for computing the Khovanov skein lasagna module of a 4-manifold after attaching an index $k$ handle, $k\in\{1,\dots,4\}$. For $k=2$, the handle attaching formula expresses the invariant as a colimit of lasagna modules governed by a \emph{cabling pattern} of the attaching link of the $2$-handle. In the case where a $2$-handle is attached to a 4-ball $B^{4}$ along a framed link $K$, the handle attaching formula provides an isomorphism between the skein lasagna module of the resulting $2$-handlebody and an invariant called the \emph{cabled Khovanov homology} of $K$ \cite{MN22} (see also \cite{HRW25}). 

In principle, the handle attachment formulas in \cite{MWWhandles} allow for the computation of the skein lasagna module for 4-manifolds built out of $1$- and $2$-handles. However, these formulas are computationally complex even in the simplest cases. For example, at present it is unknown how to compute the terms in the colimit computing $\skein_{0}^{2}(B^{4}):=\skein_{0}^{2}(B^{4};\emptyset)$ via the Kirby diagram of $B^{4}$ consisting of a canceling 1- and 2-handle (Figure \ref{fig:B4andSigma_gxD^2_kirby_diagrams}). 

Motivated by the desire for an invariant more well-suited for $4$-manifolds built out of $1$- and $2$-handles, Ren, Wedrich, Willis, Zhang and the second author \cite{RSWWZ25} defined a new invariant called the \emph{1-dimensional inputs (Khovanov) skein lasagna module}, an invariant of smooth 4-manifolds constructed in the same spirit as the original \cite{MWW-lasagna}, but with \emph{Rozansky--Willis} homology as the input link homology theory. Introduced by Rozansky \cite{ROZ-Cat} and further developed by Willis \cite{WillisS1xS2}, Rozansky--Willis homology (denoted by $\Kh_{RW}(L)$) is an invariant for links in connected sums of $S^{1}\times S^{2}$. In \cite{RSWWZ25} the authors constructed a $\frak{gl}_{2}$-refinement of Rozansky-Willis homology and showed it satisfied the functoriality properties necessary to construct a skein lasagna module. One can extract from their construction two link homology theories which we denote in this paper by $\KhR_{2,O}^{-}(L)$ and $\KhR_{2,T}^{-}(L)$; these are defined for (framed, oriented) null-homologous and two-divisible links $L\subset\sqcup_{i=1}^{k}\#^{m_{i}}S^{1}\times S^{2}$, respectively. These give rise to 1-dimensional inputs Khovanov skein lasagna modules $\overline{\skein}_{0}^{2,O}(X;L)$ and $\overline{\skein}_{0}^{2,T}(X;L)$.

In this work, we present handle attachment formulas for $k$-handles for $k\in\{1,\dots,4\}$ for the invariants $\overline{\skein}_{0}^{2,\diamond}(X;L)$, $\diamond\in\{O,T\}$. The index $k$ handle attaching formulas for $\overline{\skein}_{0}^{2,\diamond}(X;L)$ and the original skein lasagna module $\skein_{0}^{2}(X;L)$ are most noticeably distinct when $k=2$:

Consider a pair $(X,L)$ where $X$ is built from attaching $2$-handles along an $m$-component link $K=K_{1}\cup\cdots\cup K_{m}\subset\del\natural^{n}S^{1}\times B^{3}$. Analogous to the construction in \cite{MN22}, for $\diamond\in\{O,T\}$ one can define versions of \emph{cabled Rozansky-Willis homology} for a pair of links $K,L\subset\sqcup_{i=1}^{k}\#^{m_{i}}S^{1}\times S^{2}$, denoted by $\ul{\KhR}^{-}_{2,\diamond}(K,L)$. In view of Manolescu-Neithalath's $2$-handlebody formula, one would expect an isomorphism between $\ol{\skein}_{0}^{2,\diamond}(X,L)$ and $\ul{\KhR}^{-}_{2,\diamond}(K,L)$.

However, there are lasagna fillings in the 1-dimensional inputs setting which admit isotopies (called \emph{lasso moves}) that change the number of core-parallel sheets intersecting 2-handles (see Figure \ref{fig:lasso}). This gives rise to an additional relation (called the \emph{lasso relation}) on $\ul{\KhR}^{-}_{2,\diamond}(K,L)$ that one must quotient by to obtain the correct isomorphism type of $\ol{\skein}_{0}^{2,\diamond}(X,L)$.

We briefly describe the lasso relation. Let $\belt{1,1}\subset S^{1}\times{S^{2}}$ be the 2-component standard belt link as in Figure \ref{fig:beltlink}, and let $i\in\{1,\dots,m\}$. By taking the connected sum of the pairs $(\#^{n}S^{1}\times{S^{2}},K)$ and $(S^{1}\times{S^{2}},\belt{1,1})$ along $K_{i}\subset K$ and $\belt{1,0}\subset\belt{1,1}$, we obtain an $(m+1)$-component link $K(i)\subset{\#^{n+1}S^{1}\times{S^{2}}}$ as in Figure \ref{fig:new_K}. The lasso relation is then given by identifying the images of two $2$-handle cobordism maps
\[\Phi_{\diamond},\Psi_{\diamond}:\bigoplus_{i=1}^{m}\ul{\KhR}^{-}_{2,\diamond}(K(i),L)\to\ul{\KhR}^{-}_{2,\diamond}(K,L)\]
defined in Section \ref{subsec:lasso}.

\begin{figure}
    \centering
    \includegraphics[width=0.85\linewidth]{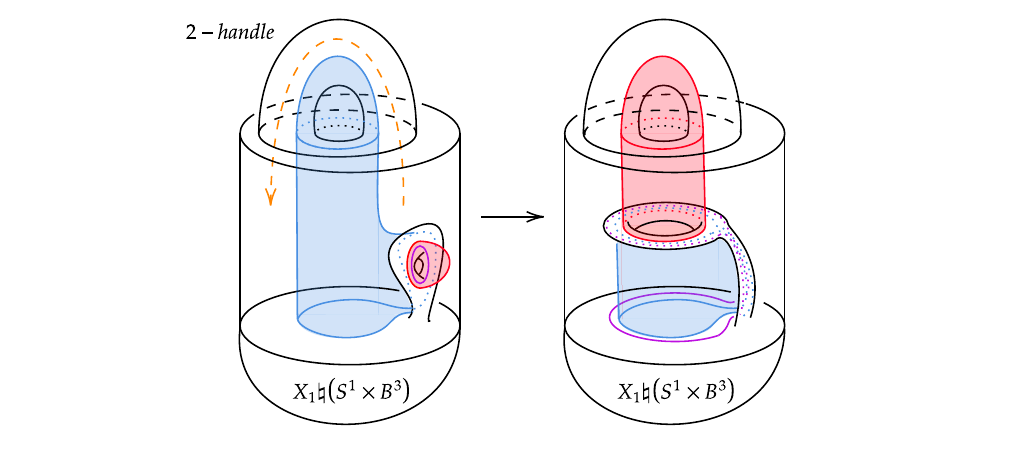}
    \caption{A cartoon of the lasso move isotopy.}
    \label{fig:lasso}
\end{figure}

We now state the main theorem of this work:

\begin{theorem}\label{theorem:main}
    Let $X$ be a compact oriented 4-manifold with handle decomposition $X_{1}\subset{X_{2}}\subset{X_{3}}\subset{X_{4}}=X$, where:
    \begin{itemize}
        \item $X_{i+1}$ is obtained from $X_{i}$ by attaching the index $(i+1)$ handles.
        \item $X_{1}$ is a 4-dimensional 1-handlebody.
        \item $K\subset\del X_{1}$ is framed attaching link for the $2$-handles of $X$.
    \end{itemize}
    Furthermore, let $L\subset\del X$ be a link, which we can interpret as a link in $\del X_{1}$. For $\diamond\in\{O,T\}$ the following statements are true:
    \begin{enumerate}
        \item\label{item:1} There is a natural isomorphism $\overline{\skein}_{0}^{2,\diamond}(X_{1};L)\cong{\KhR_{2,\diamond}^{-}(\partial X_{1},L)}$ (\cite{RSWWZ25}).
        \item\label{item:2} The lasagna module $\overline{\skein}_{0}^{2,\diamond}(X_{2};L)$ is isomorphic to a quotient of the cabled Rozansky-Willis homology of the triple $(\partial X_{1},K,L)$ (See Definition \ref{def:cabled_RW}) given by coequalizer of the maps $\Phi_{\diamond}$ and $\Psi_{\diamond}$ above:
        \[
        \overline{\skein}_{0}^{2,\diamond}(X_{2};L)\cong\mathrm{Coeq}\Big(\bigoplus_{i=1}^{m}\ul{\KhR}^{-}_{2,\diamond}(K(i),L)\doublearrow{\Phi_{\diamond}}{\Psi_{\diamond}}\ul{\KhR}^{-}_{2,\diamond}(K,L)\Big),
        \]
        where the skein grading is omitted (see Corollary \ref{cor:tortellini_1_2_handlebody}).
        \item\label{item:3} The lasagna module $\overline{\skein}_{0}^{2,\diamond}(X_{3};L)$ is isomorphic to a quotient of $\overline{\skein}_{0}^{2,\diamond}(X_{2};L)$ given by a coequalizing relation as in Proposition \ref{prop:3-handles}, similar to the 3-handle attachment formula in \cite{MWWhandles}.
        \item\label{item:4} There is a natural isomorphism $\overline{\skein}_{0}^{2,\diamond}(X_{3};L)\cong{\overline{\skein}_{0}^{2,\diamond}(X_{4};L)}$.
    \end{enumerate}
\end{theorem}

\subsection{Calculations}
\label{subsec:intro_calculations}

In this work, we use items \ref{item:1} and \ref{item:2} to perform partial computations of $\overline{\skein}_{0}^{2,\diamond}$ for disk bundles over $S^{2}$, and a partial vanishing result for $\Sigma_{g}\times{D^{2}}$ where $\Sigma_{g}$ is a genus $g$ surface. We also recalculate $\overline{\skein}_{0}^{2,\diamond}(B^{4})$ using a Kirby diagram of $B^{4}$ consisting of a canceling 1- and 2-handle pair.

Recall that in \cite{MN22}, the authors compute the ordinary skein lasagna module of $S^{2}\times D^{2}$ and exhibit an isomorphism
\[\cal{S}^{2}_{0}(S^{2}\times D^{2})\cong\bb{Q}[A_{0},A_{0}^{-1},A_{1}]\]
for some formal variables $A_{0}$ and $A_{1}$. We show that $\ol{\cal{S}}^{2,O}_{0}(S^{2}\times D^{2})$ and $\ol{\cal{S}}^{2,T}_{0}(S^{2}\times D^{2})$ are proper, non-trivial quotients of $\cal{S}^{2}_{0}(S^{2}\times D^{2})$:

\begin{proposition}[Proposition \ref{prop:tortellini_S^2xD^2}]
\label{prop:intro_tortellini_S^2xD^2}
We have algebra isomorphisms
\begin{align*}
    &\ol{\cal{S}}^{2,O}_{0}(S^{2}\times D^{2})\cong\bb{Q}[A_{0},A_{0}^{-1}] & &\ol{\cal{S}}^{2,T}_{0}(S^{2}\times D^{2})\cong\bb{Q}[A_{0}]/(A_{0}^{2}-1)
\end{align*}
where $A_{0}^{\pm 1}$ is concentrated in tri-degree $(0,0,\pm 1)$. In particular,
\[\ol{\cal{S}}^{2,\diamond}_{0}(S^{2}\times D^{2};\alpha)\cong\bb{Q}\]
for $\diamond\in\{O,T\}$ and each skein grading $\alpha$, concentrated in bi-degree $(h,q)=(0,0)$. Section \ref{sec:calculations} also provides a computation for disjoint unions of connected sums of $S^{2}\times{D^{2}}$ and addresses the case with a link in the boundary.
\end{proposition}

Let $D(p)$ denote the $D^{2}$-bundle over $S^{2}$ with Euler number $p$. As in the case of the ordinary skein lasagna modules, the invariants $\ol{\cal{S}}^{2,O}_{0}(D(p))$ and $\ol{\cal{S}}^{2,T}_{0}(D(p))$ can distinguish between the cases $p>0$ and $p<0$:

\begin{proposition}[Proposition \ref{prop:tortellini_D(p)}]
\label{prop:intro_tortellini_D(p)}
For $\diamond\in\{O,T\}$ there are isomorphisms
\begin{align*}
    &\ol{\skein}_{0}^{2,\diamond}(D(p))\cong\skein_{0}^{2}(D(p))=0\qquad\text{for }p>0, &
    &\ol{\skein}_{0,0,*}^{2,\diamond}(D(p);0)\cong\skein_{0,0,*}^{2}(D(p))\cong\bb{Q}\qquad\text{for }p<0,
\end{align*}
where the latter is concentrated in $q$-degree $0$.
\end{proposition}

Finally, consider the 4-manifold $\Sigma_{g}\times{D^{2}}$ where $\Sigma_{g}$ denotes the orientable surface of genus $g$.

\begin{proposition}(Corollary \ref{cor:surfacebundle})
In homological degrees $h\geq{1}$ and for $\diamond\in\{O,T\}$, we have that
\[\overline{\skein}_{0,h,*}^{2,\diamond}(\Sigma_{g}\times{D^{2})}=0.\]
\end{proposition}

\subsection{Future Directions}
\label{subsec:intro_future_directions}

The computational techniques developed and used in this work suggest the following questions and further directions.

Without evidence to the contrary (and in view of Proposition \ref{prop:intro_tortellini_S^2xD^2}), we pose the following conjecture:

\begin{conjecture}
Let $X$ be a 2-handlebody. Then for $\diamond\in\{O,T\}$, $\ol{\skein}_{0}^{2,\diamond}(X;\alpha)$ is finite-dimensional in each skein grading $\alpha$.
\end{conjecture}

Note that the above property does not hold for the ordinary skein lasagna invariants, as for example $\skein_{0}^{2}(S^{2}\times D^{2},\alpha)$ is infinitely generated (as a $\bb{Q}$-vector space) in each skein grading $\alpha\in\bb{Z}$.

\begin{question}\label{q:surfacexdh=0}
For $g\geq 1$ and $\diamond\in\{O,T\}$, is $\overline{\skein}_{0}^{2,\diamond}(\Sigma_{g}\times{D^{2})}$ non-vanishing (in non-positive homological degrees)? What about for $D^{2}$ bundles over $\Sigma_{g}$ with non-zero Euler number?
\end{question}

In analogy with the $g=0$ case, the authors conjecture a negative answer for positive-degree disk bundles and an affirmative answer for disk bundles of non-positive degree. It seems plausible that one could generalize the methods from \cite{Kho04} to compute these groups directly in homological degree zero, although this is beyond the scope of this paper. 

Note that this conjecture slightly contrasts what one would expect in view of the adjunction inequality from Seiberg--Witten theory (\cite{MSzT96}, \cite{OSz00}). From this point view, an affirmative answer to the conjecture would indicate that the skein lasagna package could potentially detect distinct exotic phenomena from the Seiberg--Witten invariants.

Next, given a link $L\subset\#^{n}S^{1}\times S^{2}$ and $k\in\bb{N}$, let $L(k)\subset S^{3}$ denote the link obtained from $L$ by replacing each zero-surgery curve with $k$ full-twists. It was shown in \cite{WillisS1xS2} that (the original version of) Rozansky-Willis homology can be approximated by the Khovanov homology of $L(k)$: more precisely, there exists an isomorphism $\Kh^{h,*}_{RW}(L)\cong\Kh^{h,*}(L(\mbf{k}))$ for all homological degrees $h$ larger than some constant $c(L,\mbf{k})$ depending on $L$ and $k$ such that $c(L,k)\to-\infty$ as $k\to\infty$. 

Analogous statements hold for the $\frak{gl}_{2}$-analogues $\KhR_{2,O}^{-}$ and $\KhR_{2,O}^{-}$ on the level of groups, but it does not imply that these approximation schema are sufficiently functorial for, e.g., facilitating computations of the relevant skein lasagna invariants. It would therefore be desirable to lift this full-twist approximation scheme to the $\frak{gl}_{2}$-setting:

\begin{question}\label{Q:gl2roz}
Does the $\mathfrak{gl}_{2}$ Rozansky projector defined in \cite[Appendix A]{RSWWZ25} admit an approximation by full-twist complexes? Furthermore, can a statement analogous to Proposition 6.1 in \cite{WillisS1xS2} be proven for Rozansky-Willis homology in the $\mathfrak{gl}_2$ webs and foams setting? 
\end{question}

In \cite{RW24}, Ren--Willis used Lee's deformation of Khovanov homology (\cite{lee-endo}, \cite{rasinv-genus}) in tandem with Manolescu-Neithalath's $2$-handle formula (\cite{MN22}) in order to develop vanishing/non-vanishing criteria for the ordinary skein lasagna modules of $2$-handlebodies, allowing them to detect families of exotic pairs of such manifolds.

\begin{question}
Can one use the Lee deformation of Rozansky--Willis homology (in the sense of \cite{MMSW22}) along with Theorem \ref{theorem:main} to develop vanishing/non-vanishing criteria for the $1$-dimensional inputs skein lasagna modules of $4$-manifolds built out of $1$- and $2$-handles, and potentially detect exotic pairs of this form?
\end{question}

The handle formulas proven in this paper should port over directly to analogues of $1$-dimensional inputs skein lasagna modules defined with respect to link homology theories other than Khovanov homology. For example, Chen \cite{CHE-Floer} defined an analogue of the $0$-dimensional inputs skein lasagna module using link Floer homology $\widehat{HFL}$ for links in $S^{3}$ as the input TQFT. However, the invariant $\widehat{HFL}$ is also functorial for links in closed 3-manifolds including connected sums of $S^{1}\times{S^{2}}$.

\begin{question}
Can a Floer theoretic invariant similar in spirit to the 1-dimensional input skein lasagna be defined and computed for $\widehat{HFL}$ using \cite{RSWWZ25} and techniques in this work?
\end{question}

Alternatively, one could define a new invariant of links in connected sums of $S^{1}\times S^{2}$ by constructing an analogue of the Rozansky projector in knot Floer homology via infinite full twists (see \cite{AGL25} for a related construction) and attempt to construct a 1-dimensional inputs skein lasagna modules from this putative homology theory. It would be interesting to see how this potential theory would compare to one constructed directly from $\widehat{HFL}$.


\subsection{Acknowledgments} 
The authors would like to thank Remy Bohm, Eugene Gorsky, Maggie Miller, Trevor Oliveira-Smith, and Melissa Zhang for helpful discussions, as well as Lisa Piccirillo and Qiuyu Ren for comments on an earlier draft. The first author was partially supported by the Simons collaboration ``New structures in low-dimensional topology'', and the second author by NSF grant no. DMS-2302305 during the preparation of this article. The authors would also like to thank the 2025 Trisectors Workshop for facilitating the working group which helped inspire this project.

\section{Background}
\label{sec:background}

\subsection{Notation and Conventions}
\label{subsec:conventions}

The conventions for Khovanov homology used in this paper follow \cite{RSWWZ25}, with only minor aesthetic differences in notation as recorded below. We refer to the link invariant defined by Rozansky and extended by Willis as \emph{Rozansky-Willis} homology and refer the reader to \cite{ROZ-Cat,WillisS1xS2} for further details and background about this invariant.

Throughout, we work over $\mathbb{Q}$, and we let ``$\otimes$" denote vertical stacking of tangles. 

\begin{itemize}
    \item For a framed oriented link $L\subset S^{3}$ we denote by $\KhR_{2}(L)$ the $\frak{gl}_{2}$-Khovanov-Rozansky homology of $L$.
    \begin{itemize}
        \item For the $0$-framed unknot $U$, we denote by $\mbf{1},\mbf{X}\in\KhR_{2}(U)$ the standard generators which lie in bi-degrees $(0,-1)$ and $(0,1)$, respectively.
        \item It was shown in \cite{MWW-lasagna} that the invariant $\KhR_{2}$ extends to a symmetric monoidal functor
        \[\KhR_{2}:\Links_{0}\to\fdVect_{\bb{Q}}^{\bb{Z}\times\bb{Z}}\]
        where:
        \begin{itemize}
            \item $\Links_{0}$ is the category of framed oriented links in abstract disjoint unions of $S^{3}$, with morphisms given by pairs $(W,\Sigma)$ where $W$ is a 4-manifold diffeomorphic to a disjoint union of 4-balls, with a disjoint union of 4-balls removed from each of them, and $\Sigma\subset W$ is a properly embedded framed oriented surface.
            \item $\fdVect_{\bb{Q}}^{\bb{Z}\times\bb{Z}}$ denotes the category of bigraded finite-dimensional $\bb{Q}$-vector spaces and homogeneous maps between them.
        \end{itemize}
        \item Given $L\subset S^{3}$, we have an isomorphism
        \[\KhR_{2}^{h,q}(L)\cong\Kh^{h,-q-w(L)}(m(L))\]
        where $w(L)$ denotes the writhe of of a diagram of $L$ in which the given framing of $L$ agrees with the blackboard framing, and $m(L)$ denotes the mirror of $L$.
    \end{itemize}
    \item For $\diamond\in\{O,T\}$, we have symmetric monoidal functors
    \begin{align*}
    &\KhR_{2,\diamond}^{+}:\Links_{1,\diamond}^{\op}\to\fVect_{\bb{Q}}^{\bb{Z}\times\bb{Z}} & &\KhR_{2,\diamond}^{-}:\Links_{1,\diamond}\to\fVect_{\bb{Q}}^{\bb{Z}\times\bb{Z}} 
    \end{align*}
    Here:
    \begin{itemize}
        \item $\Links_{1,T}$ is the category of (framed, oriented) links in abstract disjoint unions of connected sums of $S^{1}\times S^{2}$ (denoted by $\Links_{1}$ in \cite{RSWWZ25}) with $2$-divisible homology class, and $\Links_{1,O}\subset\Links_{1,T}$ denotes the full subcategory of null-homologous links. Cobordisms in these categories can be decomposed into elementary cobordisms; see \cite[Proposition 6.3]{RSWWZ25} for a comprehensive list of such cobordisms.
        \item $\fVect_{\bb{Q}}^{\bb{Z}\times\bb{Z}}$ denotes the category of bigraded $\bb{Q}$-vector spaces, finite-dimensional in every bi-degree, and homogeneous maps between them. These invariants are described below.
    \end{itemize}
    \item The \emph{type $O$} (or \emph{unrenormalized}) Rozansky-Willis homology $\KhR_{2,O}^{+}(L)$ (denoted by $\KhR_{2}^{+}(L)$ in \cite{RSWWZ25}) for a null-homologous link $L\subset\#^{n}S^1\times S^2$, is related to the usual Rozansky-Willis homology $\Kh_{RW}(L)$ of $L$ as defined in \cite{ROZ-Cat,WillisS1xS2} via an isomorphism
    \[\KhR_{2,O}^{+,h,q}(L)\cong(\Kh_{RW}^{-h,q+w(L)}(L))^{\vee}\]
    where $(\cdot)^{\vee}$ denotes the dual. In this paper we will use the convention that if $L$ is not null-homologous, then we formally set $\KhR_{2,O}^{+}(L)=0$.
    \item The \emph{type $T$} (or \emph{renormalized}) Rozansky-Willis homology, which we denote by $\KhR_{2,T}^{+}(L)$ (denoted by $\wt{\KhR}_{2}^{+}(L)$ in \cite{RSWWZ25}) for a 2-divisible link $L\subset\#^{n}S^{1}\times S^{2}$, is defined by
    \begin{equation}
    \label{eq:KhR_+_T_from_O}
        \KhR_{2,T}^{+,h,q}(L):=\KhR_{2,O}^{+,h+\frac{1}{2}w(L),q-\frac{1}{2}w(L)}(L).
    \end{equation}
    In this paper we will use the convention that if $L$ is not 2-divisible, then we formally set $\KhR_{2,T}^{+}(L)=0$. Note that the $h$- and $q$- gradings both take values in half-integers.
    \item There are dual versions of $\KhR_{2,O}^{+}(L)$ and $\KhR_{2,T}^{+}(L)$ defined by
    \[\KhR_{2,\diamond}^{-,h,q}(L):=(\KhR_{2,\diamond}^{+,-h,-q}(m(L)))^{\vee},\qquad\diamond\in\{O,T\}.\]
    By tracing through the definitions and noting that $w(m(L))=-w(L)$, we see that
    \begin{equation}
    \label{eq:KhR_-_T_from_O}
        \KhR_{2,T}^{-,h,q}(L)=\KhR_{2,O}^{-,h+\frac{1}{2}w(L),q-\frac{1}{2}w(L)}(L).
    \end{equation}
    In \cite{RSWWZ25} the invariants $\KhR_{2,O}^{-}(L)$ and $\KhR_{2,T}^{-}(L)$ are denoted by $\KhR_{2}^{-}(L)$ and $\wt{\KhR}_{2}^{-}(L)$, respectively.
    \item If $L$ is a link in $S^{3}$, we will sometimes simply write $\KhR_{2,\diamond}(L)$ to denote $\KhR_{2,\diamond}^{-}(S^{3},L)\cong\KhR_{2,\diamond}^{+}(S^{3},L)$, with the observation that $\KhR_{2,O}(L)=\KhR_{2}(L)$.
    \item When changing orientations of components of a link $L$, we obtain a degree shift in the type $O$-theory -- if $(L,\frak{o})$ denotes the link $L$ with a possibly different orientation $\frak{o}$ than the given one, we have a non-canonical isomorphism
    \[\KhR_{2,O}^{\pm,h,q}(L,\frak{o})\cong\KhR_{2,O}^{\pm,h+\frac{1}{2}(w(L)-w(L,\frak{o})),q-\frac{1}{2}(w(L)-w(L,\frak{o}))}(L).\]
    (See \cite{RW24}, Equation 10.) As a consequence, we have (non-canonical) isomorphisms
    \begin{equation}
    \label{eq:KhR_T_invariant_under_orientations}
        \KhR_{2,T}^{\pm,h,q}(L,\frak{o})\cong\KhR_{2,T}^{\pm,h,q}(L,\frak{o}')
    \end{equation}
    for any choices of orientations $\frak{o}$,$\frak{o}'$ on the components of a link $L$.
    \item For completeness, we note that
    \begin{align}
    \begin{split}
    \label{eq:conventions}
        &\KhR^{-,h,q}_{2,O}(L)=(\KhR_{2,O}^{+,-h,-q}(m(L)))^{\vee} \\
        &\qquad\qquad\qquad\qquad\cong((\Kh_{RW}^{h,-q+w(L)}(m(L)))^{\vee})^{\vee}
        =\Kh_{RW}^{h,-q+w(L)}(m(L)).
    \end{split}
    \end{align}
    \end{itemize}

\begin{remark}\label{rmk:signs}
We note here (once and for all) that we elect in this paper to work within the ($\frak{sl}_{2}$-)Bar-Natan formalism --- this has the advantage of being much more computable than the $\frak{gl}_{2}$-formalism in practice, but with the trade-off of introducing inherent sign ambiguities. We remark that fixing signs is possible by passing to the $\mathfrak{gl}_{2}$-webs and foams formalism and applying the main result of \cite{BHPW23}.
\end{remark}

The invariants $\KhR_{2,\diamond}^{-}$ for $\diamond\in\{O,T\}$ are governed by certain categorified idempotents called \emph{Rozansky projectors}, which are denoted by $P_{n,0}$. For these objects we follow the conventions found in \cite{SZ2024,RSWWZ25}. In \cite{SZ2024}, Zhang and the second author show that the original skein lasagna module of the pair $(S^{2}\times{D^{2}},L)$ for $L\subset\del(S^{2}\times D^{2})$ recovers $\KhR_{2,O}^{+}(L)$ up to an extra tensoral factor. However, the invariants $\ol{\skein}_{0}^{2,O}$ and $\ol{\skein}_{0}^{2,T}$ (defined later in Section \ref{subsec:skein_lasagna_recap}) take the minus versions $\KhR_{2,O}^{-}$ and $\KhR_{2,T}^{-}$ as their respective input TQFTs.

We recall the cobar construction of the Rozansky projector as discussed in \cite{hogancamp2020constructing}:

\begin{definition}\label{def:Roz-cobar}
Let $\mathcal{TL}_{n}$ denote the Temperley-Lieb category of $(n,n)$-planar tangles, let $\cal{B}_{n}$ denote the set of crossingless matchings on $n$ points, and let $C$ be the object in $\mathcal{TL}_{n}$ given by
\[C:=\bigoplus_{\delta\in{\cal{B}_{n}}}\delta\otimes{\delta^{t}}\]
where $\delta^{t}$ is the mirror of $\delta$ reflected across a horizontal axis. The object $C$ is a coalgebra with counit map $\epsilon:C\rightarrow{\bbm{1}}$ given by a composition of saddle maps $\nu_{\delta}:q^{n}(\delta\otimes{\delta}^{t})\rightarrow{\bbm{1}}$, where $\bbm{1}\in\cal{TL}_{n}$ denotes the $n$-strand identity braid. The Rozansky projector $P_{n,0}$ is homotopy equivalent to the infinite complex given by a cobar construction:
\[\dots\xrightarrow{f_{5}}{C^{\otimes 4}}\xrightarrow{f_{3}}C^{\otimes 3}\xrightarrow{f_{2}}C^{\otimes 2}\xrightarrow{f_{1}}\ul{C}\]
where $f_{k}:=\sum_{i=0}^{k}(-1)^{i}\id^{\otimes{i}}\otimes\epsilon\otimes\id^{\otimes{(k-i)}}$ for each $k\geq 1$, and the underlined term $\ul{C}$ is in homological degree zero. The counit on $C$ induces a counit $\epsilon:P_{n,0}\rightarrow{\bbm{1}}$ given by $\epsilon$ in homological degree $0$. We represent $P_{n,0}$ diagrammatically with a box on $2n$ strands:
\begin{center}
\includegraphics[width=2cm]{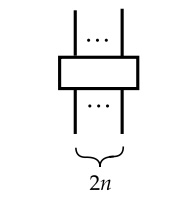}
\end{center}
\end{definition}

\begin{remark}\label{rmk:dual_projector}
The construction in Definition \ref{def:Roz-cobar} is dual to the construction in \cite{RSWWZ25}. In particular, the dual projector $P_{n,0}^{\vee}$ (which governs $\KhR_{2,\diamond}^{+}$) is obtained by the bar construction on the same coalgebra $C$, and is unital with unit map $\iota:\bbm{1}\to P_{n,0}^{\vee}$.
\end{remark}

\begin{remark}\label{rmk:Roz_gl_2}
The Rozansky projector $P_{n,0}$ is constructed in the $\frak{gl}_{2}$-formalism by lifting $C$ to an analogous counital coalgebra $\cal{C}$ comprised of $\frak{gl}_{2}$-webs, and applying Hogancamp's cobar construction to $\cal{C}$. We refer the reader to (\cite{RSWWZ25}, Appendix A.3) for additional details regarding the construction of $P_{n,0}$ in this manner.
\end{remark}

The Rozansky projector $P_{n,0}$ enjoys the following properties.

\begin{lemma}\cite[Proposition 2.6 - dually]{RSWWZ25}\label{lem:rozprops}
The Rozansky projector $P_{n,0}$ satisfies the following properties:
\begin{enumerate}
    \item As a complex, each term of $P_{n,0}$ consists entirely of $(n,n)$-Temperley-Lieb diagrams with through-degree 0.
    \item $P_{n,0}$ is counital, with counit map $\epsilon:P_{n,0}\to\bbm{1}$.
    \item The projector on $0$ strands is the empty projector $P_{0,0}$, and the counit map from $P_{0,0}$ to the empty braid is the identity.
    \item There is a chain homotopy equivalence $P_{n,0}\xrightarrow{\simeq}P_{n,0}\otimes{P_{n,0}}$, and furthermore the chain maps $\epsilon\otimes{\idmap}$ and $\idmap\otimes{\epsilon}$ are chain homotopic.
    \item For any $(k,l)$-tangle $T$, there are homotopy equivalences: $P_{l,0}\otimes{T}\otimes{P_{k,0}}\xrightarrow{\id\otimes\id\otimes{\epsilon}}P_{l,0}\otimes{T}$, $P_{l,0}\otimes{T}\otimes{P_{k,0}}\xrightarrow{\epsilon\otimes{\id}\otimes{\id}}T\otimes{P_{k,0}}$.
    \item The counit map yields chain homotopy equivalences:
        \begin{align}
            &\includegraphics[width=10cm]{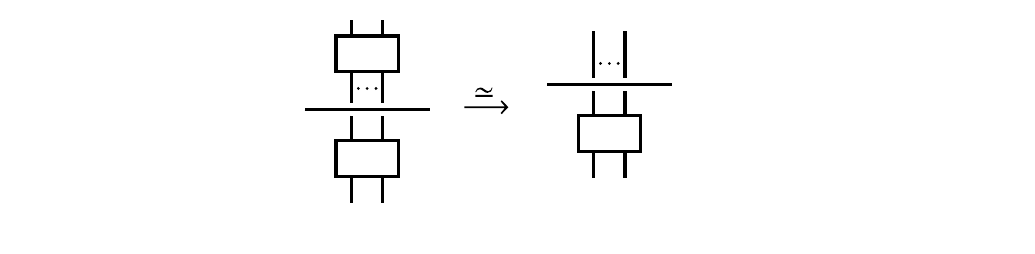} \\
            &\includegraphics[width=10cm]{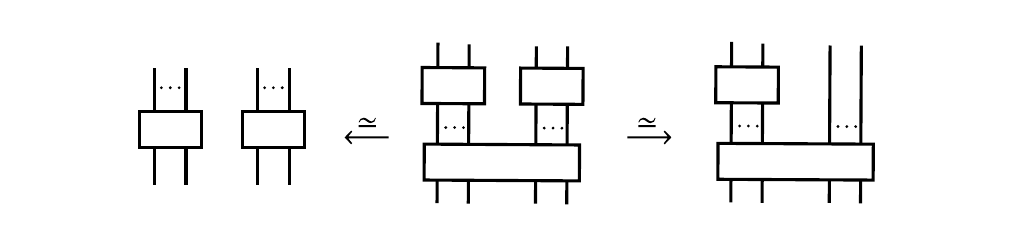}
        \end{align}
    \item The projector $P_{n,0}$ is homotopy equivalent to $\pi(P_{n,0})$, the complex obtained by a 180-degree planar rotation.
    \end{enumerate}
\end{lemma}

\begin{remark}(Gluck Twist Maps)
\label{rmk:gluck_twist_maps}
For $k,l\in{\mathbb{N}}$, let $\belt{k,\ell}$ denote a \emph{standard belt link} in $S^{1}\times{S^{2}}$ (See Figure \ref{fig:beltlink}) consisting of $k+\ell$ zero-framed standard longitudes with $k$ positively oriented components and $\ell$ negatively oriented components. For $p\in\bb{Z}$ let $\belt{k,\ell}_{p}\subset S^{1}\times S^{2}$ denote the \emph{$p$-twisted belt link} obtained from $\belt{k,\ell}$ by adding $p$ full twists and framing each of the components by $p$ (see Figure \ref{fig:beltlink}). Note that $\belt{k,\ell}_{p}$ and $\belt{k,\ell}_{q}$ are isotopic as framed links in $S^{1}\times S^{2}$ if $p\equiv q\pmod{2}$, and in general $(S^{1}\times S^{2},\belt{k,\ell}_{p})$ and $(S^{1}\times S^{2},\belt{k,\ell}_{q})$ are diffeomorphic as pairs for all $p,q\in\bb{Z}$. For all $p,q\in\bb{Z}$ and $\diamond\in\{O,T\}$ we may consider the natural isomorphisms
\begin{equation}\label{eq:gluckmap}
\tau_{\diamond}^{p}:\KhR_{2,\diamond}^{-}(S^{1}\times{S^{2}},\belt{k,\ell}_{q})\xrightarrow{\cong}{\KhR^{-}_{2,\diamond}(S^{1}\times{S^{2},\belt{k,\ell}_{p+q}}})
\end{equation}
induced by Reidemeister moves that assemble to the homotopy equivalence $P_{n,0}\simeq{P_{n,0}\otimes{\mathrm{FT}_{n}^{p}}}$, which satisfy the identity $\tau_{\diamond}^{q}\circ\tau_{\diamond}^{p}=\tau_{\diamond}^{p+q}$. We refer to these maps as $p$-fold \emph{Gluck twist} maps for $\KhR_{2,\diamond}^{-}$, given their relation to the Gluck twist isomorphisms on $\skein_{0}^{2}$ discussed in \cite[Section 7]{RSWWZ25}.
\end{remark}

\section{1-dimensional inputs skein lasagna modules}
\label{sec:1d_inputs_skein_lasagna}

\subsection{Skein Lasagna Modules with 0- and 1- Dimensional Inputs}
\label{subsec:skein_lasagna_recap}

For the duration of this paper, given a compact oriented 4-manifold $X$ and a framed oriented link $L\subset\del X$ we will refer to the pair $(X,L)$ as a 4-manifold/boundary link pair.

In \cite{MWW-lasagna} the authors construct an invariant $\skein_{0}^{2}(X;L)$ of 4-manifold/boundary link pairs $(X,L)$. These take the form of $\bb{Q}$-vector spaces generated by \emph{lasagna fillings}, whose underlying \emph{skeins} each have an associated set of \emph{input manifolds} given by a collection of neighborhoods of a finite disjoint set of points in the interior of the 4-manifold. As the input manifolds are points, we refer to this construction as the \emph{skein lasagna module with 0-dimensional inputs}. For more background on skein lasagna modules with 0-dimensional inputs, we refer the reader to \cite{MWW-lasagna,MWWhandles,MN22}. 

We recall the definition of the \emph{skein lasagna module with 1-dimensional inputs}\footnote{skein tortellini module, for short.} as constructed in \cite[Definition 3.4]{RSWWZ25}.

\begin{definition}\label{def:1dimfilling}
    Given a 4-manifold/boundary link pair $(X,L)$, a finite indexing set $I$, and $\diamond\in\{O,T\}$, a \emph{type $\diamond$ $1$-dimensional input lasagna filling} is a tuple $(\Sigma,\{B_{i}\}_{i\in{I}},\{L_{i}\}_{i\in{I}},\{v_{i}\}_{i\in{I}})$ where:
    \begin{itemize}
        \item For each $i\in I$:
        \begin{itemize}
            \item $B_{i}\subset\text{int}(X)$ is an embedded codimension-zero submanifold homeomorphic to $\natural^{n_{i}}S^{1}\times B^{3}$ for some $n_{i}\geq 0$ (called an \emph{input manifold}), such that $B_{i}\cap B_{j}=\emptyset$ for $i\neq j$.
            \item $L_{i}$ is an \emph{input link} in $\del B_{i}$.
            \item $v_{i}\in\KhR_{2,\diamond}^{-}(\del B_{i},L_{i})$ is a homogenous element, which we refer to as the \emph{decoration} on $L_{i}$.
        \end{itemize}
        \item $\Sigma$ is a properly embedded framed oriented surface (called the \emph{skein}) in $X\setminus{\bigsqcup_{i\in{I}}\mathrm{int}(B_{i})}$ such that $\partial{\Sigma}\cap\partial{X}=L$ and $(-\partial\Sigma)\cap\partial{B_{i}}=L_{i}$.
    \end{itemize}
    The \emph{type $\diamond$ 1-dimensional input skein lasagna module over $\mathbb{Q}$} is defined to be
    \[\ol{\skein}_{0}^{2,\diamond}(X;L):=\mathbb{Q}\langle{\text{$\diamond$-$1$-dim input lasagna fillings}}\rangle/\sim\]
    where $\sim$ is the transitive and linear closure of the following:
    \begin{itemize}
        \item isotopy of skeins rel boundary,
        \item multilinearity in the $\KhR_{2,\diamond}^{-}(L_{i})$ labels,
        \item the \emph{enclosement relation}: For a 1-dimensional input filling $F$ with label $v=\bigotimes_{i}v_{i}$, skein $\Sigma$, and input manifolds $B=\bigsqcup_{i}B_{i}$, if $B^{\prime}$ is a 4-dimensional 1-handlebody in $\mathrm{int}(X)$ such that $B\subseteq{\mathrm{int}(B^{\prime}})$, then the filling $F$ becomes identified with a filling $F^{\prime}$ obtained by replacing $B$ with $B^{\prime}$, $\Sigma$ with $\Sigma\setminus{\mathrm{int}(B^{\prime})}$, and $v$ with $\KhR_{2,\diamond}^{-}(\Sigma\cap(B^{\prime}\setminus{\mathrm{int}(B))(v)}$
    \end{itemize}
\end{definition}

If $L$ is the empty link, then we use the notation $\skein_{0}^{2}(X):=\skein_{0}^{2}(X;\emptyset)$ (resp. $\ol{\skein}_{0}^{2,\diamond}(X):=\overline{\skein}_{0}^{2}(X;\emptyset)$).

\begin{remark}
\label{rem:core_1d_input}
We can think of one-dimensional input manifolds as neighborhoods of 1-complexes in $\text{int}(X)$. More precisely, any fixed identification of a one-dimensional input manifold $B$ with $\natural^{n}S^{1}\times B^{3}$ induces a canonical deformation retraction of $B$ onto a subset $\Gamma(B)\cong\vee^{n}S^{1}$ called its \emph{core}. Although $\Gamma(B)$ a priori depends on the identification $B\cong\natural^{n}S^{1}\times B^{3}$, it is still well-defined up to isotopy. 
\end{remark}

\begin{remark}
Our invariant $\ol{\skein}_{0}^{2,T}(X;L)$ corresponds to the invariant simply denoted by $\ol{\skein}_{0}^{2}(X;L)$ in \cite{RSWWZ25}, whereas $\ol{\skein}_{0}^{2,O}(X;L)$ corresponds to the invariant mentioned in \cite[Remark 3.5]{RSWWZ25}.
\end{remark}

The invariants $\skein_{0}^{2}(X;L)$ and $\ol{\skein}_{0}^{2,O}(X;L)$ come equipped with a tri-grading $(h,q,\alpha)\in\bb{Z}\times\bb{Z}\times H_{2}^{L}(X;\bb{Z})$, whereas the tri-grading on $\ol{\skein}_{0}^{2,T}(X;L)$ takes values in $\frac{1}{2}\bb{Z}\times\frac{1}{2}\bb{Z}\times H_{2}^{L}(X;\bb{Z}/2)$. Here, for $R=\bb{Z}$ or $\bb{Z}/2$, $H_{2}^{L}(X;R):=\del^{-1}([L])\in H_{2}(X,L;R)$ where $\partial$ is the connecting map in the long exact sequence for the pair $(X,L)$. In all of these theories, the $h$- and $q$-gradings descend from the corresponding input homology theories, whereas the third grading (called the \emph{skein grading}) for an element represented by a lasagna filling $F$ corresponds to the homology class of the underlying skein surface $\Sigma$ of $F$.

\subsection{Properties}
\label{subsec:properties_1d_inputs_skein_lasagna}

We now discuss several properties of 1-dimensional inputs skein lasagna modules. First, note that there always exist canonical maps
\begin{align}
\label{eq:maps_between_skein_lasagnas}
    &f_{0,\diamond}:\skein_{0}^{2}(X;L)\to\ol{\skein}_{0}^{2,\diamond}(X;L),\;\diamond\in\{O,T\}, &f_{O,T}:\ol{\skein}_{0}^{2,O}(X;L)\to\ol{\skein}_{0}^{2,T}(X;L)
\end{align}
defined as follows. The map $f_{0,O}$ associates $0$-dimensional input lasagna fillings of $(X,L)$ to their corresponding type $O$ $1$-dimensional input lasagna fillings under the canonical isomorphism of functors
\[\KhR_{2,O}^{-}|_{\Links_{0}}\cong\KhR_{2}:\Links_{0}\to \fVect_{\bb{Q}^{\bb{Z}\times\bb{Z}}}.\]
This map is well-defined, as the relations defining $\skein_{0}^{2}(X;L)$ are a subset of those used to define $\ol{\skein}_{0}^{2,O}(X;L)$.

Next, given the canonical inclusion of categories $\Links_{1,O}\hookrightarrow\Links_{1,T}$, consider the natural transformation of functors $\eta:\KhR_{2,O}^{-}\Rightarrow\KhR_{2,T}^{-}|_{\Links_{1,O}}$ induced by the grading shift $(\cdot)\mapsto(tq^{-1})^{w(L)/2}(\cdot)$. We then define $f_{O,T}$ to be the map which sends each lasagna filling $F=(\Sigma,\{B_{i}\}_{i\in{I}},\{L_{i}\}_{i\in{I}},\{v_{i}\}_{i\in{I}})$ to its corresponding grading-shifted lasagna filling $F'=(\Sigma,\{B_{i}\}_{i\in{I}},\{L_{i}\}_{i\in{I}},\{\eta(v_{i})\}_{i\in{I}})$. Again, this map is well-defined as the relations defining $\ol{\skein}_{0}^{2,O}(X;L)$ are (modulo grading shifts) a subset of those used to define $\ol{\skein}_{0}^{2,T}(X;L)$. The map $f_{0,T}$ is then defined to be the composite $f_{O,T}\circ f_{0,O}$.

These maps decompose along tri-gradings, in the sense that they split as direct sums of maps
\begin{align*}
    &f_{0,O}^{(h,q,\alpha)}:\skein_{0,h,q}^{2}(X;L;\alpha)\to\ol{\skein}_{0,h,q}^{2,O}(X;L;\alpha), \\ &f_{0,T}^{(h,q,\alpha)}:\skein_{0,h,q}^{2}(X;L;\alpha)\to\ol{\skein}_{0,h+\frac{1}{2}\alpha^{2},q-\frac{1}{2}\alpha^{2}}^{2,T}(X;L;\ol{\alpha}), \\ &f_{O,T}^{(h,q,\alpha)}:\ol{\skein}_{0,h,q}^{2,O}(X;L;\alpha)\to\ol{\skein}_{0,h+\frac{1}{2}\alpha^{2},q-\frac{1}{2}\alpha^{2}}^{2,T}(X;L;\ol{\alpha}),
\end{align*}
where $\ol{\alpha}$ denotes the mod $2$ reduction of $\alpha$.

\begin{proposition}
\label{prop:surjections_lasagna_simply_connected}
Let $(X,L)$ be a 4-manifold/boundary link pair, and suppose $X$ is simply connected. Then the the maps $f_{0,\diamond}^{(h,q,\alpha)}$ for $\diamond\in\{O,T\}$ and $f_{O,T}^{(h,q,\alpha)}$ are surjective for each tri-grading $(h,q,\alpha)$.
\end{proposition}

\begin{proof}
To show that $f_{0,O}$ is surjective, it suffices to show that any type $O$ $1$-dimensional input lasagna filling of $(X,L)$ is equivalent to a $0$-dimensional lasagna filling via the enclosement relation. Let $F=(\Sigma,\{B_{i}\}_{i\in{I}},\{L_{i}\}_{i\in{I}},\{v_{i}\}_{i\in{I}})$ be such a filling, and for each $i\in I$ fix an identification $B_{i}=\natural^{n_{i}}S^{1}\times B^{3}$ and fix a deformation retraction of $\natural^{n_{i}}S^{1}\times B^{3}$ onto a standard bouquet of circles $\vee^{n_{i}}S^{1}$. Since $X$ is simply connected, each circle in $\vee^{n_{i}}S^{1}$ bounds an immersed disk in $X$ via a standard transversality argument, and hence can be smoothly isotoped to the basepoint of $\vee^{n_{i}}S^{1}$. This implies that each $B_{i}$ can be isotoped so that it is contained in a 4-ball contained in $\text{int}(X)$. Hence $F$ is equivalent to a lasagna filling containing only $0$-dimensional inputs, as desired. 

A similar argument as above implies that $f_{0,T}=f_{O,T}\circ f_{0,O}$ is surjective, and therefore $f_{O,T}$ is surjective as well. The proposition then follows from the observation that the enclosement relation respects tri-gradings.
\end{proof}

\begin{proposition}[See (\cite{MN22}, Theorem 1.4 and Corollary 7.3)]
\label{prop:connected_sum}
Let $(X_{i},L_{i})$, $i=1,2$ be 4-manifold/boundary link pairs.
Then there exists a natural isomorphism
\[\ol{\skein}_{0}^{2,\diamond}(X_{1}\sqcup X_{2};L_{1}\sqcup L_{2})\cong\ol{\skein}_{0}^{2,\diamond}(X_{1};L_{1})\otimes\ol{\skein}_{0}^{2,\diamond}(X_{2};L_{2}).\]
Furthermore, suppose at least one of $X_{1}$ and $X_{2}$ are simply connected. Then there exist natural isomorphisms
\begin{align*}
    &\ol{\skein}_{0}^{2,\diamond}(X_{1}\natural X_{2};L_{1}\sqcup L_{2})\cong\ol{\skein}_{0}^{2,\diamond}(X_{1};L_{1})\otimes\ol{\skein}_{0}^{2,\diamond}(X_{2};L_{2}), \\
    &\ol{\skein}_{0}^{2,\diamond}(X_{1}\# X_{2};L_{1}\sqcup L_{2})\cong\ol{\skein}_{0}^{2,\diamond}(X_{1};L_{1})\otimes\ol{\skein}_{0}^{2,\diamond}(X_{2};L_{2}).
\end{align*}
\end{proposition}

\begin{proof}
The isomorphism for disjoint unions follows from the monoidality of $\KhR_{2,\diamond}^{-}$ over $\mathbb{Q}$. In the case of connect sums, suppose without loss of generality that $X_{1}$ is the simply-connected connect summand, and for any 1-dimensional input lasagna filling $F$ of $X_{1}\#X_{2}$, let $B$ denote its input manifold. Recall from Remark \ref{rem:core_1d_input} that $B$ admits a deformation retraction onto its core $\Gamma(B)\approx\sqcup_{i=1}^{k}\vee^{m_{i}}S^{1}$. If $B$ has a disjoint component contained in $\mathrm{int}(X_{1})$, then it is contained in an interior $B^{4}$ and may be isotoped into $\mathrm{int}(X_{2})$. Alternatively, if $\Gamma(B)$ has a component which has non-trivial intersection with the connect sum 3-sphere, we may choose arcs in $S^{3}$ connecting intersection points and isotope $B$ along neighborhoods of disks bounded by the $S^{1}$ wedge summands of $\Gamma(B)$ and push the input manifold into $\mathrm{int}(X_{2})$.

We may then perform an isotopy so that the skein surface of $F$ intersects the connect-sum region in a trivial cobordism $L\times{I}$. The claim then follows from applying the same neck-cutting Lemma (\cite{MN22}, Lemma 7.2) as in the proof of (\cite{MN22}, Theorem 1.4). More specifically, the isotoped filling is equivalent under enclosement to the usual sum of ``neck-cut" fillings obtained by replacing the $L\times{I}$ region with two small disjoint input 4-balls with input links $L$ and $\overline{L}$ respectively with labels given by the homomorphism induced by $L\times{I}$ taken as a cobordism $\emptyset\rightarrow{L\sqcup{\overline{L}}}$. As $\KhR_{2,\diamond}^{-}$ is monoidal under disjoint unions, the result follows. The case of boundary connect sums follows by a nearly identical argument.
\end{proof}

The invariant $\ol{\skein}_{0}^{2,\diamond}(X,L)$, like the $0$-dimensional inputs version, satisfies more general gluing properties. Let $(X,L)$ be a 4-manifold/boundary link pair, and let $Y\subset X$ be a properly embedding separating 3-manifold such that $\del Y$ intersects $L\subset\del X$ in some set of points $P$. Let $(X_{1},T_{1})$ and $(X_{2},T_{2})$ denote the pairs obtained by cutting $(X,L)$ along $Y$, where $T_{i}\subset\del X_{i}$ are framed oriented tangles with endpoints $P$. Given any framed oriented tangle $T_{0}\subset Y$ with $\del T_{0}=P$, we have an induced map
\[\ol{\skein}_{0}^{2,\diamond}(X_{1};T_{1}\cup T_{0})\otimes \ol{\skein}_{0}^{2,\diamond}(X_{2};-T_{0}\cup T_{2})\to\ol{\skein}_{0}^{2,\diamond}(X;L)\]
obtained by gluing together lasagna fillings on each side (see \cite{MWWhandles} and \cite{RW24}, Section 2.2). We have the following analogue of (\cite{RW24}, Proposition 2.5):

\begin{proposition}
\label{prop:gluing_surjection}
Suppose at least one of $X_{1}$ and $X_{2}$ is simply connected. Then the gluing map
\[\sf{gl}:\bigoplus_{\substack{T_{0}\subset Y \\ \del T_{0}= P}}\ol{\skein}_{0}^{2,\diamond}(X_{1};T_{1}\cup T_{0})\otimes \ol{\skein}_{0}^{2,\diamond}(X_{2};-T_{0}\cup T_{2})\to\ol{\skein}_{0}^{2,\diamond}(X;L)\]
is a surjection.
\end{proposition}

\begin{proof}
Every lasagna filling of $(X,L)$ can be isotoped so that it intersects $Y$ transversely in some tangle $T_{0}$, and that (assuming at least one $X_{i}$ is simply-connected) the corresponding one-dimensional input manifold $B$ is disjoint from $Y$ by a similar argument as in the proof of Proposition \ref{prop:connected_sum}.
\end{proof}

From Proposition \ref{prop:gluing_surjection} we obtain the following corollary:

\begin{corollary}
\label{cor:embeddings}
Let $X\subset\text{int}(X')$ be an inclusion of $4$-manifolds, and suppose that
\begin{enumerate}
    \item At least one of $X$ and $X'\setminus X$ is simply connected.
    \item $\ol{\skein}_{0}^{2,\diamond}(X;L)=0$ for any link $L\subset\del X$.
\end{enumerate}
Then $\ol{\skein}_{0}^{2,\diamond}(X';L')$ vanishes for any link $L'\subset\del X'$.
\end{corollary}

It will also be helpful to consider the above gluing map when one of the inputs is fixed. In particular: suppose $(X,L)$ is a 4-manifold/boundary link pair with $\del X=Y$, let $W$ be a 4-dimensional cobordism from $Y$ to some other 3-manifold $Y'$, and $\Sigma\subset W$ a surface cobordism from $L\subset Y$ to some link $L'\subset Y'$. Then we have an induced map
\begin{align*}
    \ol{\skein}_{0}^{2,\diamond}(W;\Sigma):\ol{\skein}_{0}^{2,\diamond}(X;L)&\to\ol{\skein}_{0}^{2,\diamond}(X\cup_{Y} W;L') \\
    F\mapsto \sf{gl}(F\otimes\Sigma),
\end{align*}
where $\Sigma$ is treated as a lasagna filling of the 4-manifold/boundary link pair $(W,-L\sqcup L')$. One special case of this is where $(W,\Sigma)$ is the mapping cylinder of a diffeomorphism of pairs $\varphi:(Y,L)\to(Y',L')$. In this case we will denote the induced map simply by $\ol{\skein}_{0}^{2,\diamond}(\varphi)$.
\section{2-handles}
\label{sec:2-handles}

\subsection{Braid Group Actions}
\label{subsec:braid_group_actions}

Before discussing the 2-handle attachment formula for $\overline{\skein}_{0}^{2,\diamond}$, we remark on braid group actions on Rozansky-Willis homology. Let $K_{n}$ denote the $n$-cable of the framed link $K$ in the boundary of a 4-manifold $X$, where all parallel strands are oriented in the same direction, then there is a natural action of the braid group $\mathfrak{B}_{n}$ on $\skein_{0}^{2}(X;K_{n})$ that permutes the parallel strands of $K_{n}$. More specifically, for each generator $\sigma_{i}\in{\mathfrak{B}_{n}}$, there is an associated endomorphism $\beta_{i}\in\mathrm{End}({\skein_{0}^{2}(X;K_{n})})$ defined by pulling back the action of $\sigma_{i}$ on an $n$-punctured $D^{2}\times{S^{1}}$ to a tubular neighborhood of $K\subset{\partial{X}}$. When $X=B^{4}$, we remark that this action factors through the symmetric group $\mathfrak{S}_{n}$ by Grigsby--Licata--Wehrli \cite{GLW-schur-weyl}.

\begin{figure}
    \centering
    \includegraphics[width=0.85\linewidth]{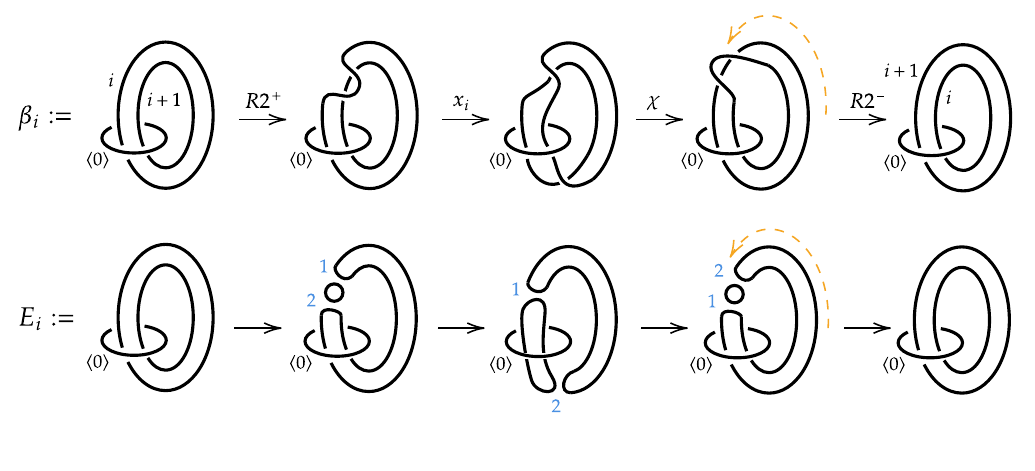}
    \caption{\textbf{TOP:} $\beta_{i}$ movie, \textbf{BOTTOM:} $E_{i}$ movie. Components other than the $i$th and $(i+1)$th omitted.}
    \label{fig:1}
\end{figure}

\begin{proposition}
\label{prop:braid_group_action_factors_through_symmetic_group}
    Let $K_{n}$ be the $n$-cable of an admissible knot $K$ in $\#^{k}_{i=1}(S^{1}\times{S^{2}})_{i}$, then there exists a braid group representation $\rho:\mathfrak{B}_{n}\rightarrow\mathrm{End}(\KhR_{2,\diamond}^{-}(K_{n}))$ that factors through the symmetric group $\mathfrak{S}_{n}$.
\end{proposition}

\begin{proof}
    If $k=0$ then we are done by \cite{GLW-schur-weyl}, so suppose that $k=1$, and suppose that $K$ is the 0-framed knot in $S^{1}\times{S^{2}}$ that intersects the surgery region geometrically once. Let $D_{K}$ be a crossingless admissible diagram of $K$, and let $D_{K_{n}}$ be the admissible, crossingless diagram of the $n$-cable of $K$. Define endocobordism maps $\beta_{i}$ and $E_{i}$ in $(S^{1}\times{S^{2}})\times{I}$ on $K_{n}$ as described my the movies in Figure \ref{fig:1}.

    We remark that these cobordisms are the same as those discussed in \cite{GLW-schur-weyl}, with the insertion of a \emph{crossing-move} and \emph{finger-move} for $\beta_{i}$ and $E_{i}$ respectively. Passing to Khovanov complexes, the movies for these cobordism maps is given by replacing the surgery regions in Figure \ref{fig:1} with Rozansky projectors. Letting $F$ and $G$ denote the Reidemeister II maps that increase and decrease crossing number respectively, letting $\chi$ denote an isotopy that slides a crossing around, and letting $x_{i}$ and $s_{i}$ denote the crossing-move and finger-move chain maps respectively, we have that the chain map corresponding to $\beta_{i}$ can be written as $[\beta_{i}]=G\circ{x_{i}}\circ{\chi}\circ{F}$. We may then apply the argument presented in \cite{GLW-schur-weyl} directly after verifying that the map $x_{i}$ is non-decreasing in homological degree. 
    
    Let $\delta_{i,i+1}$ denote the $(2n,2n)$-tangle given by a single positive crossing between the $i$th and $(i+1)$th strands, then the complex $\delta_{i,i+1}\otimes{P_{n,0}}$ can be realized as a mapping telescope with each term of the form $\delta_{i,i+1}\otimes{\mathrm{FT^{\otimes{k}}_{2n}}}$ with some possible shifts and for a specified choice of connecting maps \cite{EH17}. In particular, we may write the above as a 2-term complex of the form
    \[
    \delta_{i,i+1}\otimes{P_{n,0}}:=(0\rightarrow\underline{\bigoplus_{j=0}^{\infty}(\delta_{i,i}\otimes{\mathrm{FT}_{2n}^{\otimes{k}}})}\rightarrow\bigoplus_{j=0}^{\infty}(\delta_{i,i+1}\otimes{\mathrm{FT}_{2n}^{\otimes{k}}})\rightarrow{0})
    \]
    
    with underlined term in homological degree 0 and potential shifts omitted. The isotopy $\rho_{i,k}:\delta_{i,i+1}\otimes{\mathrm{FT}_{2n}^{\otimes{k}}}\rightarrow{\mathrm{FT}_{2n}^{\otimes{k}}}\otimes{\delta_{i,i+1}}$ that slides a crossing through some number of full-twists consists entirely of Reidemeister moves, and therefore the induced map on chain complexes is non-decreasing in homological degree. Realizing the crossing move map $x_{i}:\delta_{i,i+1}\otimes{P_{n,0}}\rightarrow{P_{n,0}}\otimes{\delta_{i,i+1}}$ as the map induced on the above mapping cylinder by the collection $\{\rho_{i,k}\}_{k\geq{0}}$, we have that $x_{i}$ is also non-decreasing in homological degree. As promised, we may then apply the same degree splitting argument presented in \cite{GLW-schur-weyl}. Let $D_{U_{n},j}$, $j\in{\{1,...,5\}}$ be the admissible link diagrams in the $\beta_{i}$ movie, then we may regard each diagram as a bicomplex, with differentials corresponding to the crossings introduced by the Reidemeister II move, and differentials of the projector $P_{n,0}$ as a mapping telescope. Define two homological degrees, $i_{1}$ and $i_{2}$, to be homological degrees corresponding to Reidemeister II crossing differentials and $P_{n,0}$ differentials. Thus, the chain map $[\beta_{i}]:[D_{U_{n},1}]\rightarrow{[D_{U_{n},5}]}$ can be written as the sum $[\beta_{i}]=[\beta_{i}]_{0}+[\beta_{i}]_{+}$, an $i_{2}$-degree preserving component and an $i_{2}$-degree increasing component. Since the first and last still in the $\beta_{i}$-movie are supported in $i_{2}$-degree $0$, we have that $[\beta_{i}]_{+}=0$, thus, we may write $[\beta_{i}]=\mathrm{id}+[E_{i}]=[\beta_{i}^{-1}]$ after fixing signs. The extension to an arbitrary framed knot $K$ with an admissible diagram follows from \cite{GLW-schur-weyl} and the above case. Furthermore, by locality, the argument extends to framed knots in $\#^{n}S^{1}\times{S^{2}}$ for arbitrary $n\geq{0}$.
\end{proof}

\subsection{Cabled 1-dimensional inputs skein lasagna}

Let $X$ be a compact oriented four-manifold with boundary $\del X=Y$, let $K=K_{1}\cup\cdots\cup K_{m}\subset Y$ be a framed oriented link, and let $X'$ be the 4-manifold with boundary obtained by attaching $m$ 4-dimensional 2-handles to $Y$ along $K$. Furthermore, let $L\subset Y$ be a (possibly empty) framed oriented link disjoint from $K$ and let $L'\subset Y':=\del X'$ denote the image of $L$ after performing the 2-handle attachments. In this section we will describe how to calculate the 1-dimensional inputs skein lasagna module of $(X',L')$ in terms of a certain cabled 1-dimensional input skein lasagna module associated to the triple $(X,K,L)$.

Given $\mbf{k}^{\pm}=(k_{1}^{\pm},\cdots,k_{m}^{\pm})\in\bb{N}^{m}$, let $K(\mbf{k}^{+},\mbf{k}^{-})\cup L\subset Y$ denote the framed oriented link obtained from $K\cup L$
by replacing each component $K_{i}$ of $K$ with $k_{i}^{+}+k_{i}^{-}$ parallel copies of itself, such that $k_{i}^{-}$ of the copies are given the orientation opposite to the orientation on $K$. For $\mbf{k}^{+},\mbf{k}^{-}\in\bb{N}^{m}$ and $\diamond\in\{O,T\}$, there is a natural action of 
\[\frak{B}_{\mbf{k}^{\pm}}:=\prod_{i=1}^{m}(\frak{B}_{k_{i}^{+}}\times\frak{B}_{k_{i}^{-}})\]
on the 1-dimensional inputs skein lasagna module $\ol{\skein}_{0}^{2,\diamond}\big(X;K(\mbf{k}^{+},\mbf{k}^{-})\cup L\big)$ as in (\cite{MWWhandles}, Section 3). Analogous to the cobordism maps defined in (\cite{MN22}, Section 3.1), we may consider the homomorphisms 
\[\psi_{\diamond,i}^{[j]}:\ol{\skein}_{0}^{2,\diamond}\big(X;K(\mbf{k}^{+},\mbf{k}^{-})\cup L\big)\to \ol{\skein}_{0}^{2,\diamond}\big(X;K(\mbf{k}^{+}+\mbf{e}_{i},\mbf{k}^{-}+\mbf{e}_{i})\cup L\big)\]
induced by (possibly dotted) annuli that cobound oppositely oriented parallel strands for $i\in\{1,\dots,m\}$, $j\in{\{0,1\}}$, where $\psi_{\diamond,i}^{[0]}$ denotes the undotted annulus cobordism (resp. $\psi_{\diamond,i}^{[1]}$ is induced by the dotted annulus cobordism) \footnote{Bundt cake and Bundt cake with a raisin.}. The map $\psi_{\diamond,i}^{[j]}$ changes the bigrading by $(0,2j)$.

Let $R=\bb{Z}$ or $\bb{Z}/2$, and let $H_{2}^{L}(X;R)$ denote the set $\partial^{-1}([L])\subseteq{H_{2}(X,L;R)}$ where $\partial$ is the connecting map in the relative homology long exact sequence. As in \cite{MWWhandles}, we can identify $H_{2}^{L'}(X';R)$ with a subgroup of $R^{m}\oplus H_{2}(X,K\cup L;R)$, and subsequently write any element of $H_{2}^{L'}(X';\bb{Z})$ (respectively, $H_{2}^{L'}(X';\bb{Z}/2)$) as an element of the form $(\bm{\alpha},\eta)$ (resp., $(\ol{\bm{\alpha}},\ol{\eta})$) with $\bm{\alpha}\in\bb{Z}^{m}$, $\eta\in H_{2}(X,K\cup L;\bb{Z})$, (resp,  $\ol{\bm{\alpha}}\in(\bb{Z}/2)^{m}$, $\ol{\eta}\in H_{2}(X,K\cup L;\bb{Z}/2)$).

Given $\bm{\alpha}=(\alpha_{1},\dots,\alpha_{m})\in\bb{Z}^{m}$, we let $\bm{\alpha}^{\pm}=(\alpha_{1}^{\pm},\dots,\alpha_{m}^{\pm})\in\bb{N}^{m}$ be given by $\alpha_{i}^{\pm}=\max\{0,\pm\alpha_{i}\}$, $i=1,\dots,m$. Suppose $\mbf{k}^{+},\mbf{k}^{-}\in\bb{N}^{m}$ are such that: 
\begin{enumerate}
    \item[(C1)] if $R=\bb{Z}$, then $\mbf{k}^{+}-\mbf{k}^{-}=\bm{\alpha}$, and
    \item[(C2)] if $R=\bb{Z}/2$, then $\mbf{k}^{+}-\mbf{k}^{-}\equiv\bm{\alpha}\pmod{2}$.
\end{enumerate}
Let $F_{\bm{\alpha}^{\pm},\mbf{r};R}$ denote the composition
\begin{align}
\label{eq:eta_map}
\begin{split}
    &H_{2}^{K(\mbf{k}^{+},\mbf{k}^{-})\cup L}(X;R)\hookrightarrow H_{2}(X,K(\mbf{k}^{+},\mbf{k}^{-})\cup L;R) \\
    &\qquad\qquad\qquad\qquad\qquad\qquad\to H_{2}(X,\nu(K)\cup L;R)\cong H_{2}(X,K\cup L;R).
\end{split}
\end{align}
For any class $(\bm{\alpha},\eta)\in H_{2}^{L'}(W;\bb{Z})$ (respectively, $(\ol{\bm{\alpha}},\ol{\eta})\in H_{2}^{L'}(W';\bb{Z}/2)$) and any $\mbf{k}^{\pm}\in\bb{N}^{m}$ satisfying (C1) (respectively, (C2)), there exists a unique element $\eta_{\mbf{k}^{\pm}}\in H_{2}^{K(\mbf{k}^{+},\mbf{k}^{-})\cup L}(X;\bb{Z})$ (resp., $\ol{\eta}_{\mbf{k}^{\pm}}\in H_{2}^{K(\mbf{k}^{+},\mbf{k}^{-})\cup L}(X;\bb{Z}/2)$) which is sent to $\eta$ (resp., $\ol{\eta}$) under $F_{\mbf{k}^{\pm};\bb{Z}}$ (resp., $F_{\mbf{k}^{\pm};\bb{Z}/2}$).

\begin{definition}
\label{def:cabled_1d_input_lasagna}
Let $(X,K,L)$ be as above. We define the \emph{type $O$ cabled 1-dimensional inputs skein lasagna module of the triple $(X,K,L)$ at level $\bm{\alpha}$ and in class $\eta$} to be
\begin{equation}
    \ul{\ol{\skein}}_{0}^{2,O,\bm{\alpha}}(X;K,L;\eta):=\Big(\bigoplus_{\mbf{r}\in\bb{N}^{m}}\ol{\skein}_{0}^{2,O}\big(X;K(\bm{\alpha}^{+}+\mbf{r},\bm{\alpha}^{-}+\mbf{r})\cup L;\eta_{\bm{\alpha}^{\pm}+\mbf{r}}\big)\big\{-|\bm{\alpha}|-2|\mbf{r}|\big\}\Big)/\sim,
\end{equation}

where the equivalence relation $\sim$ is the transitive and linear closure of the relations
\begin{align*}
    &\beta\cdot v\sim v, & &\psi_{O,i}^{[0]}(v)\sim 0, & &\psi_{O,i}^{[1]}(v)\sim v
\end{align*}
for all $\beta\in\frak{B}_{\bm{\alpha}^{\pm}+\mbf{r}}$, $v\in \ol{\skein}_{0}^{2,O}(X;K(\bm{\alpha}^{+}+\mbf{r},\bm{\alpha}^{-}+\mbf{r})\cup L;\eta_{\bm{\alpha}^{\pm}+\mbf{r}})\{-|\bm{\alpha}|-2|\mbf{r}|\}$, $i=1,\dots,m$. Similarly, we define the \emph{type $T$ cabled 1-dimensional inputs skein lasagna module of the triple $(X,K,L)$ at level $\ol{\bm{\alpha}}$ and in class $\ol{\eta}$} to be

\begin{equation}
    \ul{\ol{\skein}}_{0}^{2,T,\ol{\bm{\alpha}}}(X;K,L;\ol{\eta}):=\Big(\bigoplus_{\substack{\mbf{k}^{+},\mbf{k}^{-}\in\bb{N}^{m} \\ \mbf{k}^{+}-\mbf{k}^{-}\equiv\ol{\bm{\alpha}}\;(\text{mod }2)}}\ol{\skein}_{0}^{2,T}\big(X;K(\mbf{k}^{+},\mbf{k}^{-})\cup L;\ol{\eta}_{\mbf{k}^{\pm}}\big)\big\{-|\mbf{k}^{+}|-|\mbf{k}^{-}|\big\}\Big)/\sim,
\end{equation}
where the equivalence relation $\sim$ is the transitive and linear closure of the relations
\begin{align*}
    &\beta\cdot v\sim v, & &\psi_{T,i}^{[0]}(v)\sim 0, & &\psi_{T,i}^{[1]}(v)\sim v,
\end{align*}
for all $\beta\in\frak{B}_{\bm{\alpha}^{\pm}+\mbf{r}}$, $v\in \ol{\skein}_{0}^{2,T}(X;K(\mbf{k}^{+},\mbf{k}^{-})\cup L;\ol{\eta}_{\mbf{k}^{\pm}})\{-|\mbf{k}^{+}|-|\mbf{k}^{-}|\}$, $i=1,\dots, m$. Given $(\bm{\alpha},\eta)\in H_{2}^{L'}(X,\bb{Z})$, we set
\[\ul{\ol{\skein}}_{0}^{2,T,\bm{\alpha}}(X;K,L;\eta):=\ul{\ol{\skein}}_{0}^{2,T,\ol{\bm{\alpha}}}(X;K,L;\ol{\eta}),\]
where $\ol{\bm{\alpha}}\in(\bb{Z}/{2})^{m}$ and $\ol{\eta}\in H_{2}(X,K\cup L;\bb{Z}/2)$ denote the mod $2$ reductions of $\bm{\alpha}$ and $\eta$, respectively.
\end{definition}

Note that, by construction (\cite{RSWWZ25}, discussion before Remark 3.5), the invariant $\overline{\skein}_{0}^{2,\diamond}$ recovers Rozansky-Willis homology when $X=\natural^{n}S^{1}\times{B^{3}}$. In particular there is a natural isomorphism
\[\ol{\skein}_{0}^{2,\diamond}(\textstyle{\natural^{n}}S^{1}\times B^{3};L)\cong\KhR_{2,\diamond}^{-}(L)\]
for any link $L\subset\del(\natural^{n}S^{1}\times B^{3})=\#^{n}S^{1}\times S^{2}$. Furthermore, by Proposition \ref{prop:braid_group_action_factors_through_symmetic_group} the braid group action on $\KhR_{2,\diamond}^{-}(L)$ factors through the symmetric group. Given $\mbf{k}^{+},\mbf{k}^{-}\in\bb{N}^{m}$, we write
\[\frak{S}_{\mbf{k}^{\pm}}:=\prod_{i=1}^{m}(\frak{S}_{k_{i}^{+}}\times\frak{S}_{k_{i}^{-}})\]
in analogy with the corresponding notation for the braid group action.

In terms of homology classes, in this case we have that $\eta$ is uniquely determined by $\bm{\alpha}$ (and similarly in the $\diamond=T$ case), and so we drop it from the notation. (See \cite{MWWhandles}, discussion before Section 3.2.)

\begin{definition}
\label{def:cabled_RW}
Let $K\cup L\subset\#^{n}S^{1}\times S^{2}$ be as above. The \emph{type $O$ cabled Rozansky-Willis homology of the pair $(K,L)$ at level $\bm{\alpha}\in\bb{Z}^{m}$} is given by
\[\ul{\KhR}^{-}_{2,O,\bm{\alpha}}(K,L):=\big(\bigoplus_{\mbf{r}\in\bb{N}^{m}}\KhR_{2,O}^{-}\big(K(\bm{\alpha}^{+}+\mbf{r},\bm{\alpha}^{-}+\mbf{r})\cup L\big)\big\{-|\bm{\alpha}|-2|\mbf{r}|\big\}\big)/\sim,\]
where the equivalence relation $\sim$ is the transitive and linear closure of the relations
\begin{align*}
    &\sigma\cdot v\sim v, & &\psi_{O,i}^{[0]}(v)\sim 0, & &\psi_{O,i}^{[1]}(v)\sim v
\end{align*}
for all $\sigma\in\frak{S}_{\bm{\alpha}^{\pm}+\mbf{r}}$, $v\in \KhR_{2,O}^{-}(K(\bm{\alpha}^{+}+\mbf{r},\bm{\alpha}^{-}+\mbf{r})\cup L)\{-|\bm{\alpha}|-2|\mbf{r}|\}$, $i=1,\dots,m$. Similarly, the \emph{type $T$ cabled Rozansky-Willis homology of the pair $(K,L)$ at level $\ol{\bm{\alpha}}\in(\bb{Z}/2)^{m}$} is given by
\[\ul{\KhR}^{-}_{2,T,\ol{\bm{\alpha}}}(K,L):=\Big(\bigoplus_{\substack{\mbf{k}^{+},\mbf{k}^{-}\in\bb{N}^{m} \\ \mbf{k}^{+}-\mbf{k}^{-}\equiv\ol{\bm{\alpha}}\;(\text{mod }2)}}\KhR_{2,T}^{-}\big(K(\mbf{k}^{+},\mbf{k}^{-})\cup L\big)\big\{-|\mbf{k}^{+}|-|\mbf{k}^{-}|\big\}\Big)/\sim,\]
where the equivalence relation $\sim$ is the transitive and linear closure of the relations
\begin{align*}
    &\sigma\cdot v\sim v, & &\psi_{T,i}^{[0]}(v)\sim 0, & &\psi_{T,i}^{[1]}(v)\sim v
\end{align*}
for all $\sigma\in\frak{S}_{\mbf{k}^{\pm}}$, $v\in \KhR_{2,T}^{-}(K(\mbf{k}^{+},\mbf{k}^{-})\cup L)\{-|\mbf{k}^{+}|-|\mbf{k}^{-}|\}$, $i=1,\dots, m$. Given $\bm{\alpha}\in\bb{Z}^{m}$, we set
\[\ul{\KhR}^{-}_{2,T,\bm{\alpha}}(K,L):=\ul{\KhR}^{-}_{2,T,\ol{\bm{\alpha}}}(K,L),\]
where $\ol{\bm{\alpha}}\in(\bb{Z}/2)^{m}$ denotes the mod $2$ reduction of $\bm{\alpha}$.
\end{definition}

Note that $\ul{\KhR}^{-}_{2,\diamond,\bm{\alpha}}(K,L)\neq 0$ only if
\begin{align}
    &\sum_{i=1}^{m}\alpha_{i}[K_{i}]+[L]=0\in H_{1}(\#^{n}(S^{1}\times S^{2});\bb{Z}) &\text{if }\diamond=O, \label{eq:alpha_nullhomologous_condition} \\ 
    &\sum_{i=1}^{m}\ol{\alpha}_{i}[K_{i}]+[L]=0\in H_{1}(\#^{n}(S^{1}\times S^{2});\bb{Z}/2) &\text{if }\diamond=T. \label{eq:alpha_bar_2-divisible_condition}
\end{align}
%

\subsection{The lasso relation}
\label{subsec:lasso}

Let $(X^{\prime},L^{\prime})$ denote the 4-manifold and boundary link pair obtained by attaching some number of $4$-dimensional 2-handles to $(X,L)$. In this section we spell out the additional relation that we must impose on $\ul{\ol{\skein}}_{0}^{2,\diamond}(X;K,L)$ in order to ultimately relate this object to the invariant $\ol{\skein}_{0}^{2,\diamond}(X';L')$.

\begin{figure}
    \centering
    \includegraphics[width=0.8\linewidth]{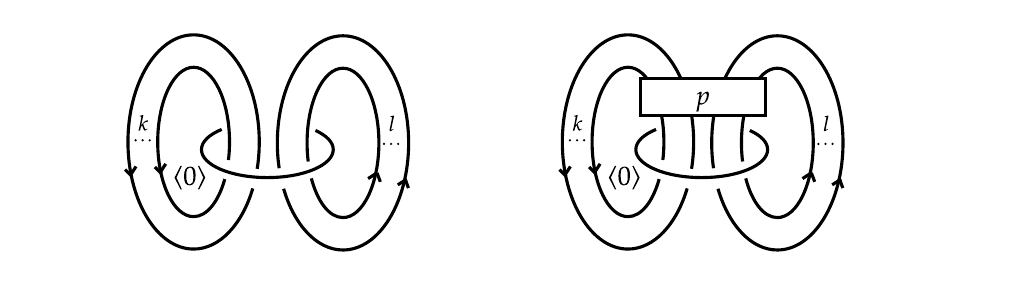}
    \caption{The standard and twisted belt links $\wt{\bbm{1}}(k,l)$ and $\wt{\bbm{1}}(k,l)_{p}$ in $S^{1}\times{S^{2}}$ for $p\in{\mathbb{Z}}$.}
    \label{fig:beltlink}
\end{figure}

We denote by $\belt{k,l}\subset S^{1}\times S^{2}$ the 0-framed standard belt link in $S^{1}\times S^{2}$ with $k$ positively-oriented and $\ell$ negatively-oriented strands (see Figure \ref{fig:beltlink}.) Let $X$ and $K\cup L\subset Y=\del X$ be as in the previous section, and let $i\in\{1,\dots, m\}$. We define a new link $K(i)\subset Y\#S^{1}\times S^{2}$ as follows: first consider the link
\[\wt{K}(i):=K\sqcup\belt{1,1}\subset Y\#S^{1}\times S^{2},\]
where we have a disjoint copy of $K$ inside of $Y\setminus B^{3}\subset Y\#S^{1}\times S^{2}$, along with a disjoint copy of $\belt{1,1}\subset S^{1}\times S^{2}\setminus B^{3}\subset Y\#S^{1}\times S^{2}$. Choose a framed arc $a$ connecting the $i$-th component $K_{i}\subset K$ with the positively oriented component $\belt{1,0}\subset\belt{1,1}$. We then define $K(i)\subset Y\# S^{1}\times S^{2}$ to be the framed link obtained by performing a (framed, oriented) connected sum along $a$. More explicitly, we can write $K(i)=K(i)_{1}\cup\cdots\cup K(i)_{m+1}$ where
\[K(i)_{j}=\threepartdef{K_{i}\#\belt{1,0}}{j=i}{\belt{0,1}}{j=m+1}{K_{j}}{j\neq i,m+1.}\]
By isotoping $a$ if necessary, we can assume that the connect sum region is disjoint from $L\subset Y\setminus B^{3}\subset Y\# S^{1}\times S^{2}$, and hence $K(i)\cup L$ is a well-defined link in $Y\# S^{1}\times S^{2}$, independent of the choice of arc $a$ up to diffeomorphism of pairs.

Next, we construct two cobordisms of pairs
\[(W_{i},\Sigma_{i}),(W'_{i},\Sigma'_{i}):(Y\#(S^{1}\times S^{2}),K(i)\cup L)\to(Y,K\cup L)\]
as follows: let $(Z_{i},S_{i})$ and $(Z'_{i},S'_{i})$ be the cobordisms of pairs given by attaching 4-dimensional 2-handles along parallel copies $\gamma$ and $\gamma'$ of $K(i)_{m+1}=\belt{0,1}$ and $K(i)_{i}=K_{i}\#\belt{1,0}$, respectively, with surfaces $S_{i},S'_{i}$ given by the corresponding traces of $K(i)\cup{L}$ in each cobordism. Since $K(i)_{m+1}$ and $K(i)_{i}$ each have geometric intersection number $1$ with the $S^{1}\times S^{2}$-factor, it follows that $Z_{i}$ and $Z'_{i}$ are both cobordisms from $Y\#S^{1}\times S^{2}$ to $Y$.

We observe that
\[S_{i}|_{Y}=S'_{i}|_{Y}=K\cup U\cup L\subset Y,\]
where $U=(U,0)$ denotes a 0-framed unknot contained in a $3$-ball. The surfaces $S_{i}$ and $S_{i}'$ each consist of a union of cylinders of the following forms:
\begin{enumerate}
    \item Trivial cylinders connecting the components of $(K(i)\cup L)\setminus(K(i)_{m+1}\cup K(i)_{i})$ to their corresponding components in $(K\cup L)\setminus K_{i}$.
    \item In $S_{i}$: a pair of cylinders connecting $K(i)_{m+1}$ to $U$ and $K(i)_{i}$ to $K_{i}$.
    \item In $S_{i}'$: a pair of cylinders connecting $K(i)_{i}$ to $U$ and $K(i)_{m+1}$ to $K_{i}$.
\end{enumerate}
One can see that the above description holds by noting that $U$ can be obtained by sliding $K(i)_{m+1}$ over the 2-handle attached along $\gamma$, as well as by sliding $K(i)_{i}$ over the 2-handle attached along $\gamma'$ (see Figure \ref{fig:S_i_and_S'_i}).
\begin{figure}
    \centering
    \includegraphics[width=16cm]{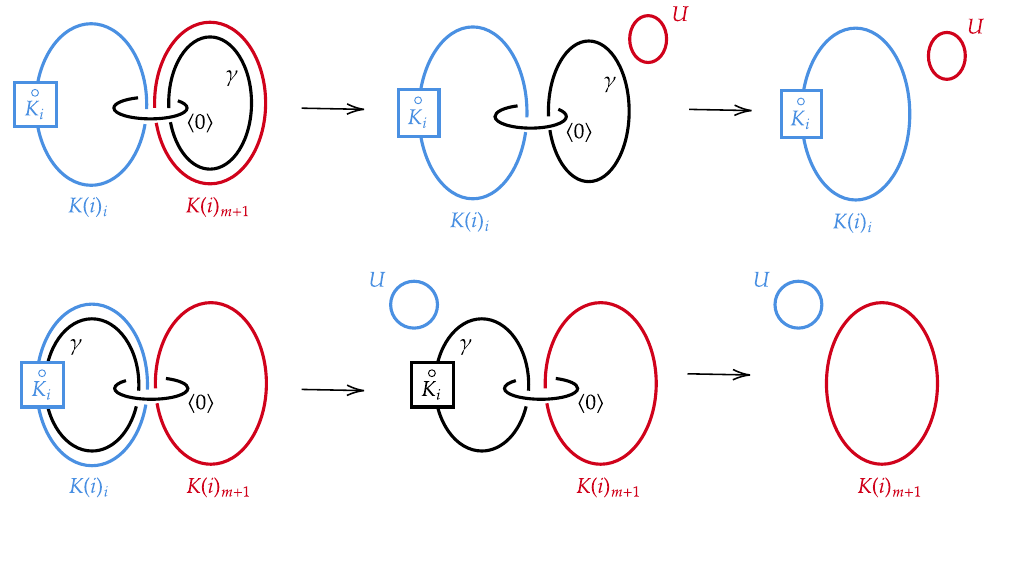}
    \caption{Cancelling one- and two-handles. Here, $\mathring{K}_{i}$ denotes the $(1,1)$-tangle in $Y\setminus B_{3}$ obtained from $(Y,K_{i})$ by deleting a $3$-ball/arc pair.}
    \label{fig:S_i_and_S'_i}
\end{figure}

Let 
\[\big(Y\times[0,1],C\big):\big(Y,K\cup U\cup L\big)\to\big(Y,K\cup L\big)\]
denote the cobordism of pairs given by capping off the unknot component via a disk contained in $B^{3}\times[0,\frac{1}{2}]\subset Y\times[0,1]$, and attaching trivial cylinders to remaining link components. We then define
\begin{align*}
    &(W_{i},\Sigma_{i}):=(Y\times[0,1],C)\circ(Z_{i},S_{i}), &
    &(W'_{i},\Sigma'_{i}):=(Y\times[0,1],C)\circ(Z'_{i},S'_{i}).
\end{align*}
Now given $(\bm{\alpha},\eta)\in H_{2}^{L'}(X';\bb{Z})<\bb{Z}^{m}\oplus H_{2}(X,K\cup L;\bb{Z})$, one can show that there exists a unique element $\eta_{i}\in H_{2}(X\natural S^{1}\times B^{3},K(i)\cup L;\bb{Z})$ such that $\eta$ is sent to $(\eta_{i},[\Sigma_{i}])$ and $(\eta_{i},[\Sigma'_{i}])$ under the natural maps
\begin{align*}
    &H_{2}(X,K\cup L;\bb{Z})\to H_{2}(X\natural S^{1}\times B^{3},K(i)\cup L;\bb{Z})\oplus H_{2}(W_{i},-(K(i)\cup L)\sqcup (K\cup L);\bb{Z}), \\
    &H_{2}(X,K\cup L;\bb{Z})\to H_{2}(X\natural S^{1}\times B^{3},K(i)\cup L;\bb{Z})\oplus H_{2}(W'_{i},-(K(i)\cup L)\sqcup (K\cup L);\bb{Z}),
\end{align*}
respectively. In the mod $2$ homology case, an element $(\ol{\bm{\alpha}},\ol{\eta})\in H_{2}^{L'}(X';\bb{Z}/2)$ gives rise to an analogously characterized element $\ol{\eta}_{i}\in H_{2}(X\natural S^{1}\times B^{3},K(i)\cup L;\bb{Z}/2)$.

Next, we extend the cobordism maps and $K(i)$'s to cables. Let $\mbf{k}^{+},\mbf{k}^{-}\in\bb{N}^{m}$, $\ell^{+},\ell^{-}\in\bb{N}$, and consider the $(m+1)$-tuples $(\mbf{k}^{\pm},\ell^{\pm}):=(k_{1}^{\pm},\dots,k_{m}^{\pm},\ell^{\pm})\in\bb{N}^{m+1}$. Given $i\in\{1,\dots, m\}$, we define
\[K(\mbf{k}^{\pm};\ell^{\pm};i):=K(i)((\mbf{k}^{+},\ell^{+}),(\mbf{k}^{-},\ell^{-}))\subset Y\#S^{1}\times S^{2}.\]
(See Figure \ref{fig:new_K}.)
\begin{figure}
    \centering
    \includegraphics[width=16cm]{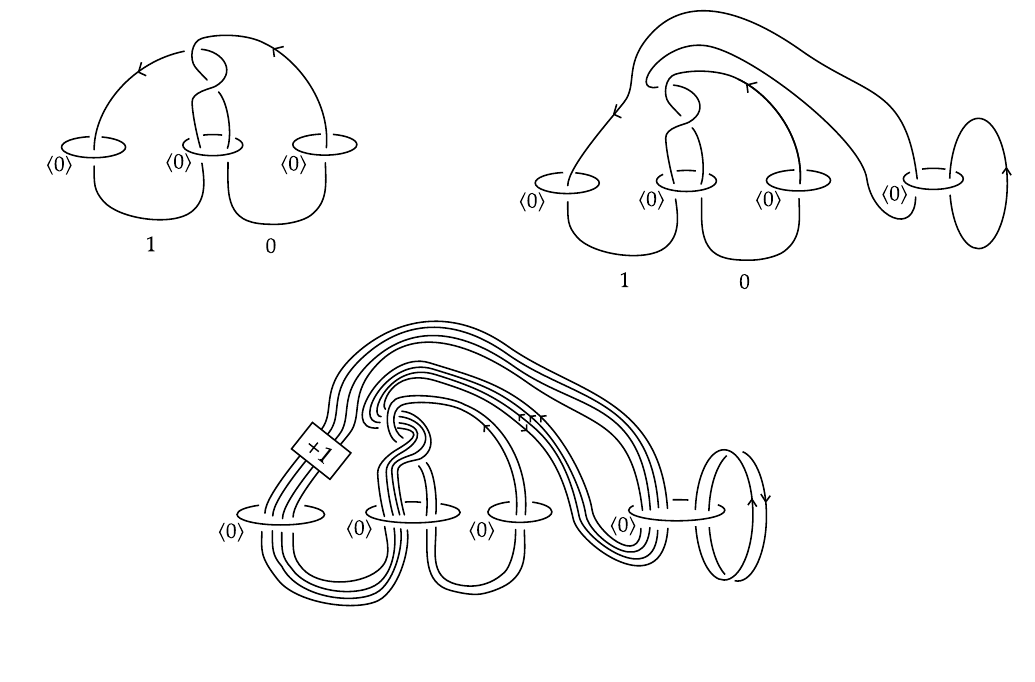}
    \caption{\textbf{Top Left:} A link $K=K_{1}\cup K_{2}\subset\#^{3}S^{1}\times S^{2}$. \textbf{Top Right:} the link $K(1)\subset \#^{4}S^{1}\times S^{2}$. \textbf{Bottom:} A cable $K(\mbf{k}^{\pm};\ell^{\pm};1)\subset \#^{4}S^{1}\times S^{2}$ of $K(1)$.}
    \label{fig:new_K}
\end{figure}
We have analogously defined corbordisms of pairs
\begin{align*}
    &(W_{i},\Sigma_{\mbf{k}^{\pm},\ell^{\pm},i}):\big(Y\#S^{1}\times S^{2},K(\mbf{k}^{\pm};\ell^{\pm};i)\cup L\big) \\
    &\qquad\qquad\qquad\qquad\qquad\qquad\qquad\qquad\qquad\quad\,\to\big(Y,K(\mbf{k}^{+},\mbf{k}^{-})\cup L\big), \\
    &(W'_{i},\Sigma'_{\mbf{k}^{\pm},\ell^{\pm},i}):\big(Y\#S^{1}\times S^{2},K(\mbf{k}^{\pm};\ell^{\pm};i)\cup L\big) \\
    &\qquad\qquad\qquad\qquad\to\big(Y,K(\mbf{k}^{+}+(\ell^{+}-k_{i}^{+})\mbf{e}_{i},\mbf{k}^{-}+(\ell^{-}-k_{i}^{-})\mbf{e}_{i})\cup L\big),
\end{align*}
obtained by appropriately taking framed push-offs of the surface components of the cobordisms $(W_{i},\Sigma_{i})$ and $(W'_{i},\Sigma'_{i})$ defined above, with analogous decompositions
\begin{align}
    &(W_{i},\Sigma_{\mbf{k}^{\pm},\ell^{\pm},i})=(Y\times[0,1],C_{\ell^{\pm}})\circ(Z_{i},S_{\mbf{k}^{\pm},\ell^{\pm},i}), 
    \label{eq:W_i_cabled_decomp} \\
    &(W'_{i},\Sigma'_{\mbf{k}^{\pm},\ell^{\pm},i})=(Y\times[0,1],C_{k_{i}^{\pm}})\circ(Z'_{i},S'_{\mbf{k}^{\pm},\ell^{\pm},i}). \label{eq:W'_i_cabled_decomp}
\end{align}
Now let $\diamond\in\{O,T\}$. For the rest of this section we will denote
\begin{align*}
    &\ol{\skein}_{0}^{2,\diamond}(\mbf{k}^{\pm}):=\ol{\skein}_{0}^{2,\diamond}\big(X;K(\mbf{k}^{+},\mbf{k}^{-})\cup L\big)\big\{-|\mbf{k}^{+}|-|\mbf{k}^{-}|\big\}, \\
    &\ol{\skein}_{0}^{2,\diamond}(\mbf{k}^{\pm};\ell^{\pm};i):=\ol{\skein}_{0}^{2,\diamond}\big(X\natural S^{1}\times B^{3};K(\mbf{k}^{\pm};\ell^{\pm};i)\cup L\big)\big\{-|\mbf{k}^{+}|-|\mbf{k}^{-}|-\ell^{+}-\ell^{-}\big\}.
\end{align*}
Next, consider the maps
\begin{align}
    &\Phi_{\diamond,\mbf{k}^{\pm},\ell^{\pm},i}:=\ol{\skein}_{0}^{2,\diamond}(W_{i},\Sigma_{\mbf{k}^{\pm},\ell^{\pm},i}):\ol{\skein}_{0}^{2,\diamond}(\mbf{k}^{\pm};\ell^{\pm};i)\to\ol{\skein}_{0}^{2,\diamond}(\mbf{k}^{\pm}) \label{eq:Phi} \\ 
    &\Psi_{\diamond,\mbf{k}^{\pm},\ell^{\pm},i}:=\ol{\skein}_{0}^{2,\diamond}(W'_{i},\Sigma'_{\mbf{k}^{\pm},\ell^{\pm},i}):\ol{\skein}_{0}^{2,\diamond}(\mbf{k}^{\pm};\ell^{\pm};i)\to\ol{\skein}_{0}^{2,\diamond}(\mbf{k}^{\pm}+(\ell^{\pm}-k_{i}^{\pm})\mbf{e}_{i}) \label{eq:Psi}
\end{align}
obtained as the gluing maps associated to $(W_{i},\Sigma_{\mbf{k}^{\pm},\ell^{\pm},i})$ and $(W'_{i},\Sigma'_{\mbf{k}^{\pm},\ell^{\pm},i})$, respectively. (See Section \ref{subsec:properties_1d_inputs_skein_lasagna}).

Let $\eta_{\mbf{k}^{\pm},\ell^{\pm},i}\in H_{2}^{K(\mbf{k}^{\pm};\ell^{\pm};i)\cup L}(X\natural S^{1}\times B^{3};\bb{Z})$ be the unique element which is sent to $\eta_{i}\in H_{2}(X\natural (S^{1}\times B^{3}),K(i)\cup L;\bb{Z})$ under the analogue of the map from (\ref{eq:eta_map}). Then one can show that $\eta_{\mbf{k}^{\pm}}$ and $\eta_{\mbf{k}^{\pm}+(\ell^{\pm}-k_{i}^{\pm})\mbf{e}_{i}}$ are sent to $(\eta_{\mbf{k}^{\pm},\ell^{\pm},i},[\Sigma_{\mbf{k}^{\pm},\ell^{\pm},i}])$ and $(\eta_{\mbf{k}^{\pm},\ell^{\pm},i},[\Sigma'_{\mbf{k}^{\pm},\ell^{\pm},i}])$ under the natural maps
\begin{align*}
    &H_{2}^{K(\mbf{k}^{+},\mbf{k}^{-})\cup L}(X;\bb{Z})\to \\
    &\qquad\qquad H_{2}(X\natural S^{1}\times B^{3},K(\mbf{k}^{\pm};\ell^{\pm};i)\cup L;\bb{Z}) \\
    &\qquad\qquad\qquad\oplus H_{2}(W,-(K(\mbf{k}^{\pm};\ell^{\pm};i)\cup L)\sqcup (K(\mbf{k}^{+},\mbf{k}^{-})\cup L);\bb{Z}), \\
    &H_{2}^{K(\mbf{k}^{+}+(\ell^{+}-k_{i}^{+})\mbf{e}_{i},\mbf{k}^{-}+(\ell^{-}-k_{i}^{-})\mbf{e}_{i})\cup L}(X;\bb{Z})\to \\
    &\qquad\qquad H_{2}(X\natural S^{1}\times B^{3},K(\mbf{k}^{\pm};\ell^{\pm};i)\cup L;\bb{Z}) \\
    &\qquad\qquad\qquad\oplus H_{2}(W',-(K(\mbf{k}^{\pm};\ell^{\pm};i)\cup L)\sqcup (K(\mbf{k}^{+}+(\ell^{+}-k_{i}^{+})\mbf{e}_{i},\mbf{k}^{-}+(\ell^{-}-k_{i}^{-})\mbf{e}_{i})\cup L);\bb{Z}),
\end{align*}
respectively, with a similar statement in the mod $2$ homology case. In particular, we see that $\Phi_{\diamond,\mbf{k}^{\pm},\ell^{\pm},i}$ and $\Psi_{\diamond,\mbf{k}^{\pm},\ell^{\pm},i}$ decompose along skein gradings as follows:
\begin{align*}
    &\Phi_{\diamond,\mbf{k}^{\pm},\ell^{\pm},i}^{\eta}:\ol{\skein}_{0}^{2,\diamond}(\mbf{k}^{\pm};\ell^{\pm};i;\eta_{\mbf{k}^{\pm},\ell^{\pm},i})\to\ol{\skein}_{0}^{2,\diamond}(\mbf{k}^{\pm};\eta_{\mbf{k}^{\pm}}) \\
    &\Psi_{\diamond,\mbf{k}^{\pm},\ell^{\pm},i}^{\eta}:\ol{\skein}_{0}^{2,\diamond}(\mbf{k}^{\pm};\ell^{\pm};i;\eta_{\mbf{k}^{\pm},\ell^{\pm},i})\to\ol{\skein}_{0}^{2,\diamond}(\mbf{k}^{\pm}+(\ell^{\pm}-k_{i}^{\pm})\mbf{e}_{i};\eta_{\mbf{k}^{\pm}+(\ell^{\pm}-k_{i}^{\pm})\mbf{e}_{i}}).
\end{align*}
%

\subsection{A 2-handle formula}
\label{subsec:2-h_formula}

In this section, we construct an isomorphism between the 1-dimensional inputs skein lasagna module of a 4-manifold with a 2-handle attached and a quotient of the cabled 1-dimensional inputs skein lasagna module. 

Consider the action of the braid group
\[\frak{B}_{\mbf{k}^{\pm},\ell^{\pm},i}:=\frak{B}_{(\mbf{k}^{\pm},\ell^{\pm})}\cong\frak{B}_{\mbf{k}^{\pm}}\times\frak{B}_{\ell^{+}}\times\frak{B}_{\ell^{-}}\]
on $K(\mbf{k}^{\pm};\ell^{\pm};i)$. We have natural projection maps
\begin{align*}
    &p_{\mbf{k}^{\pm},\ell^{\pm},i}:\frak{B}_{\mbf{k}^{\pm},\ell^{\pm},i}\to\frak{B}_{\mbf{k}^{\pm}}, & &p'_{\mbf{k}^{\pm},\ell^{\pm},i}:\frak{B}_{\mbf{k}^{\pm},\ell^{\pm},i}\to\frak{B}_{\mbf{k}^{\pm}+(\ell^{\pm}-k_{i}^{\pm})\mbf{e}_{i}},
\end{align*}
where $p_{\mbf{k}^{\pm},\ell^{\pm},i}$ is given by the projection onto the $\frak{B}_{\mbf{k}^{\pm}}$ factor, and $p'_{\mbf{k}^{\pm},\ell^{\pm},i}$ is projection onto the factor
\[(\prod_{\substack{j=1 \\ j\neq i}}^{m}\frak{B}_{k_{j}^{+}})\times(\prod_{\substack{j=1 \\ j\neq i}}^{m}\frak{B}_{k_{j}^{-}})\times\frak{B}_{\ell^{+}}\times\frak{B}_{\ell^{-}}\cong\frak{B}_{\mbf{k}^{\pm}+(\ell^{\pm}-k_{i}^{\pm})\mbf{e}_{i}}.\]
Furthermore, for $j=1,\dots,m$, and $k=0,1$, the annulus homomorphisms take the form of maps
\[\psi_{\diamond,j}^{[k]}:\ol{\skein}_{0}^{2,\diamond}(\mbf{k}^{\pm};\ell^{\pm};i)\to\ol{\skein}_{0}^{2,\diamond}(\mbf{k}^{\pm}+\mbf{e}_{j};\ell^{\pm};i),\]
whereas for $j=m+1$, the annulus homomorphisms for $k=0,1$ are of the form
\[\psi_{\diamond,m+1}^{[k]}:\ol{\skein}_{0}^{2,\diamond}(\mbf{k}^{\pm};\ell^{\pm};i)\to\ol{\skein}_{0}^{2,\diamond}(\mbf{k}^{\pm};\ell^{\pm}+1;i).\]
Each of these maps are $\frak{B}_{\mbf{k}^{\pm},\ell^{\pm},i}$-equivariant, where the actions on the codomains of $\psi_{\diamond,j}^{[k]}$, $j=1,\dots,m$ and $\psi_{\diamond,i,m+1}^{[k]}$ respectively, $k=0,1$, are induced by the canonical inclusions
\begin{align*}
    &i_{\mbf{k}^{\pm},\ell^{\pm},i;j}:\frak{B}_{\mbf{k}^{\pm},\ell^{\pm},i}\hookrightarrow\frak{B}_{\mbf{k}^{\pm}+\mbf{e}_{j},\ell^{\pm},i}, & &i_{\mbf{k}^{\pm},\ell^{\pm},i;m+1}:\frak{B}_{\mbf{k}^{\pm},\ell^{\pm},i}\hookrightarrow\frak{B}_{\mbf{k}^{\pm},\ell^{\pm}+1,i}.
\end{align*}
By construction we see that $\Phi_{\diamond,\mbf{k}^{\pm},\ell^{\pm},i}$ and $\Psi_{\diamond,\mbf{k}^{\pm},\ell^{\pm},i}$ are both $\frak{B}_{\mbf{k}^{\pm},\ell^{\pm},i}$-equivariant, where the action of $\frak{B}_{\mbf{k}^{\pm},\ell^{\pm},i}$ on their respective codomains are induced by the projections $p_{\mbf{k}^{\pm},\ell^{\pm},i}$ and $p'_{\mbf{k}^{\pm},\ell^{\pm},i}$. Furthermore, these maps fit into the following commutative diagrams for $i=1,\dots,m$, $j=1,\dots,m$, $j\neq i$, $k=0,1$ (and decompose along skein gradings appropriately):

\begin{center}
\begin{tikzcd}[sep=huge]
\ol{\skein}_{0}^{2,\diamond}(\mbf{k}^{\pm}+\mbf{e}_{j}) & \ol{\skein}_{0}^{2,\diamond}(\mbf{k}^{\pm}+\mbf{e}_{j};\ell^{\pm};i)\arrow[l,"\Phi_{\diamond,\mbf{k}^{\pm}+\mbf{e}_{j},\ell^{\pm},i}"'] \arrow[r,"\Psi_{\diamond,\mbf{k}^{\pm}+\mbf{e}_{j},\ell^{\pm},i}"] & \ol{\skein}_{0}^{2,\diamond}(\mbf{k}^{\pm}+(\ell^{\pm}-k_{i}^{\pm})\mbf{e}_{i}+\mbf{e}_{j}) \\
\ol{\skein}_{0}^{2,\diamond}(\mbf{k}^{\pm})\arrow[u,"\psi_{\diamond,j}^{[k]}"] & \ol{\skein}_{0}^{2,\diamond}(\mbf{k}^{\pm};\ell^{\pm};i)\arrow[l,"\Phi_{\diamond,\mbf{k}^{\pm},\ell^{\pm},i}"'] \arrow[r,"\Psi_{\diamond,\mbf{k}^{\pm},\ell^{\pm},i}"]\arrow[u,"\psi_{\diamond,j}^{[k]}"'] & \ol{\skein}_{0}^{2,\diamond}(\mbf{k}^{\pm}+(\ell^{\pm}-k_{i}^{\pm})\mbf{e}_{i}),\arrow[u,"\psi_{\diamond,j}^{[k]}"'] \\
\ol{\skein}_{0}^{2,\diamond}(\mbf{k}^{\pm}+\mbf{e}_{i}) & \ol{\skein}_{0}^{2,\diamond}(\mbf{k}^{\pm}+\mbf{e}_{i};\ell^{\pm};i)\arrow[l,"\Phi_{\diamond,\mbf{k}^{\pm}+\mbf{e}_{i},\ell^{\pm},i}"'] \arrow[r,"\Psi_{\diamond,\mbf{k}^{\pm}+\mbf{e}_{i},\ell^{\pm},i}"] & \ol{\skein}_{0}^{2,\diamond}(\mbf{k}^{\pm}+(\ell^{\pm}-k_{i}^{\pm})\mbf{e}_{i}) \\
\ol{\skein}_{0}^{2,\diamond}(\mbf{k}^{\pm})\arrow[u,"\psi_{\diamond,i}^{[k]}"] & \ol{\skein}_{0}^{2,\diamond}(\mbf{k}^{\pm};\ell^{\pm};i)\arrow[l,"\Phi_{\diamond,\mbf{k}^{\pm},\ell^{\pm},i}"'] \arrow[r,"\Psi_{\diamond,\mbf{k}^{\pm},\ell^{\pm},i}"]\arrow[u,"\psi_{\diamond,i}^{[k]}"'] & \ol{\skein}_{0}^{2,\diamond}(\mbf{k}^{\pm}+(\ell^{\pm}-k_{i}^{\pm})\mbf{e}_{i}),\arrow[u,"\text{id}"'] \\
\ol{\skein}_{0}^{2,\diamond}(\mbf{k}^{\pm}) & \ol{\skein}_{0}^{2,\diamond}(\mbf{k}^{\pm};\ell^{\pm}+1;i)\arrow[l,"\Phi_{\diamond,\mbf{k}^{\pm},\ell^{\pm}+1,i}"'] \arrow[r,"\Psi_{\diamond,\mbf{k}^{\pm},\ell^{\pm}+1,i}"] & \ol{\skein}_{0}^{2,\diamond}(\mbf{k}^{\pm}+(\ell^{\pm}-k_{i}^{\pm}+1)\mbf{e}_{i}) \\
\ol{\skein}_{0}^{2,\diamond}(\mbf{k}^{\pm})\arrow[u,"\text{id}"] & \ol{\skein}_{0}^{2,\diamond}(\mbf{k}^{\pm};\ell^{\pm};i)\arrow[l,"\Phi_{\diamond,\mbf{k}^{\pm},\ell^{\pm},i}"'] \arrow[r,"\Psi_{\diamond,\mbf{k}^{\pm},\ell^{\pm},i}"]\arrow[u,"\psi_{\diamond,m+1}^{[k]}"'] & \ol{\skein}_{0}^{2,\diamond}(\mbf{k}^{\pm}+(\ell^{\pm}-k_{i}^{\pm})\mbf{e}_{i}).\arrow[u,"\psi_{\diamond,i}^{[k]}"']
\end{tikzcd}
\end{center}
In particular, this implies that the collections of maps $\{\Phi^{\eta}_{\diamond,\mbf{k}^{\pm},\ell^{\pm},i}\}$ and $\{\Psi^{\eta}_{\diamond,\mbf{k}^{\pm},\ell^{\pm},i}\}$ are compatible with the relations defining the cabled 1-dimensional inputs skein lasagna modules, and so induce well-defined maps
\begin{align*}
    &\Phi^{(\bm{\alpha},\eta)}_{O},\Psi^{(\bm{\alpha},\eta)}_{O}:\bigoplus_{i=1}^{m}\ul{\ol{\skein}}^{2,O,(\bm{\alpha},\alpha_{i})}_{0}(X\natural S^{1}\times B^{3};K(i),L;\eta_{i})\to \ul{\ol{\skein}}^{2,O,\bm{\alpha}}_{0}(X;K,L;\eta) \\
    &\Phi^{(\ol{\bm{\alpha}},\ol{\eta})}_{T},\Psi^{(\ol{\bm{\alpha}},\ol{\eta})}_{T}:\bigoplus_{i=1}^{m}\ul{\ol{\skein}}^{2,T,(\ol{\bm{\alpha}},\ol{\alpha}_{i})}_{0}(X\natural S^{1}\times B^{3};K(i),L;\ol{\eta}_{i})\to \ul{\ol{\skein}}^{2,T,\ol{\bm{\alpha}}}_{0}(X;K,L;\ol{\eta})
\end{align*}
between the corresponding cabled 1-dimensional inputs skein lasagna modules for each $(\bm{\alpha},\eta)\in H_{2}^{L'}(X';\bb{Z})$ and $(\ol{\bm{\alpha}},\ol{\eta})\in H_{2}^{L'}(X';\bb{Z}/2)$. With this in mind, we have the following theorem which calculates the 1-dimensional inputs skein lasagna modules of $(X',L')$:

\begin{theorem}
\label{theorem:tortellini_2_handle_formula}
For $\diamond\in\{O,T\}$, the type $\diamond$ $1$-dimensional input skein lasagna modules of $(X',L')$ at level $(\bm{\alpha},\eta)$ can be expressed by the formula
\[\ol{\cal{S}}_{0}^{2,\diamond}(X';L';(\bm{\alpha},\eta))\cong\text{Coeq}\Big(\bigoplus_{i=1}^{m}\ul{\ol{\skein}}^{2,\diamond,(\bm{\alpha},\alpha_{i})}_{0}(X\natural S^{1}\times B^{3};K(i),L;\eta_{i})\doublearrow{\Phi^{(\bm{\alpha},\eta)}_{\diamond}}{\Psi^{(\bm{\alpha},\eta)}_{\diamond}}\ul{\ol{\skein}}^{2,\diamond,\bm{\alpha}}_{0}(X;K,L;\eta)\Big).\]
\end{theorem}

As a corollary, we spell out the above theorem in the special case where $X=\natural^{n}S^{1}\times B^{3}$ is a 1-handlebody:

\begin{corollary}
\label{cor:tortellini_1_2_handlebody}
Let $(X,K,L)$ and $(X',L')$ be as above with $X=\natural^{n}S^{1}\times B^{3}$. Then for $\diamond\in\{O,T\}$, the type $\diamond$ $1$-dimensional input skein lasagna module of $(X',L')$ at level $\bm{\alpha}$ can be expressed by the formula
\[\ol{\cal{S}}_{0}^{2,\diamond}(X';L';\bm{\alpha})\cong\text{Coeq}\Big(\bigoplus_{i=1}^{m}\ul{\KhR}^{-}_{2,\diamond,(\bm{\alpha},\alpha_{i})}(\#^{n+1}S^{1}\times S^{2};K(i),L)\doublearrow{\Phi^{\bm{\alpha}}_{\diamond}}{\Psi^{\bm{\alpha}}_{\diamond}}\ul{\KhR}^{-}_{2,\diamond,\bm{\alpha}}(\#^{n}S^{1}\times S^{2};K,L)\Big).\]
\end{corollary}

The rest of this section is devoted to the proof of Theorem \ref{theorem:tortellini_2_handle_formula}. Before we define the relevant maps, the following lemma will prove useful:

\begin{lemma}
\label{lem:isotopy_of_embeddings}
Let:
\begin{enumerate}
    \item $W_{i}$ and $W_{i}'$ be the cobordisms from $Y\# S^{1}\times S^{2}$ to $Y$ given by 2-handle attachments along push-offs $\gamma$ and $\gamma'$ of $K(i)_{m+1}$ and $K(i)_{i}$, respectively, as described above.
    \item $Z$ be the 2-handle cobordism from $Y$ to $Y'$ induced by attaching $2$-handles along the link $K\subset Y$.
    \item $i_{0},i_{1}:X\natural (S^{1}\times B^{3})\hookrightarrow X'$ be the (codimension-0) embeddings induced by the decompositions
    \begin{align}
    \label{eq:embeddings_isotopy_1}
        &X'=X\natural (S^{1}\times B^{3})\cup_{Y\#(S^{1}\times S^{2})}W_{i}\cup_{Y}Z, & &X'=X\natural (S^{1}\times B^{3})\cup_{Y\#(S^{1}\times S^{2})}W'_{i}\cup_{Y}Z,
    \end{align}
    respectively.
\end{enumerate}
Then there exists an ambient isotopy $F:X'\times[0,1]\to X'$ inducing an isotopy of embeddings $i_{t}:X\natural (S^{1}\times B^{3})\hookrightarrow X'$, $t\in[0,1]$, interpolating between $i_{0}$ and $i_{1}$.
\end{lemma}

\begin{proof}
Let $f:W_{i}\cup_{Y}Z\hookrightarrow X'$ and $f':W'_{i}\cup_{Y}Z\hookrightarrow X'$ be the embeddings induced by (\ref{eq:embeddings_isotopy_1}). Consider the cobordisms $\wt{Z}$ and $\wt{Z}'$ from $Y\#S^{1}\times S^{2}$ to $Y'$ given by $2$-handle attachments along the links $K\cup L\cup\gamma$ and $K\cup L\cup\gamma'$, respectively, where we consider $K\cup L$ as a link in $Y\#S^{1}\times S^{2}$ via $K\cup L\subset Y\setminus B^{3}\subset Y\#S^{1}\times S^{2}$. Let  $g:\wt{Z}\hookrightarrow X'$ and $g':\wt{Z}'\hookrightarrow X'$ be the embeddings induced by the analogous decompositions
\begin{align}
\label{eq:embeddings_isotopy_2}
    &X'=X\natural (S^{1}\times B^{3})\cup_{Y\#(S^{1}\times S^{2})}\wt{Z}, & &X'=X\natural (S^{1}\times B^{3})\cup_{Y\#(S^{1}\times S^{2})}\wt{Z'}.
\end{align}
Note that there exist ambient isotopies $G,G':X'\times[0,1]\to X'$ taking $f(W_{i}\cup_{Y}Z)$ to $g(\wt{Z})$ and $f'(W'_{i}\cup_{Y}Z)$ to $g'(\wt{Z}')$, respectively. 

Finally, let $\wt{\gamma}\subset Y\#S^{1}\times S^{2}$ be the result of handle-sliding $\gamma$ over $K_{i}$, which we observe is isotopic to $\gamma'$ (see Figure \ref{fig:handleslide_gamma_over_K_i}).
\begin{figure}
    \centering
    \includegraphics[width=17cm]{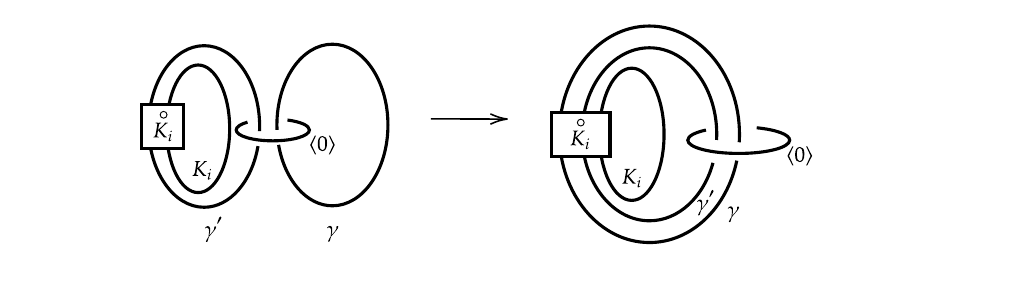}
    \caption{Handlesliding $\gamma$ over $K_{i}$.}
    \label{fig:handleslide_gamma_over_K_i}
\end{figure}
This handleslide composed with the isotopy from $\wt{\gamma}$ to $\gamma'$ induces an ambient isotopy $H:X'\times[0,1]\to X'$ taking $g(\wt{Z})$ to $g'(\wt{Z}')$. The desired ambient isotopy is then given by
\pagebreak
\begin{align*}
    F:X'\times[0,1]&\to X' \\
    (x,t)&\mapsto\threepartdef{G(x,3t)}{t\in[0,\frac{1}{3}],}{H(x,3t-1)}{t\in[\frac{1}{3},\frac{2}{3}],}{G'(x,3-3t)}{t\in[\frac{2}{3},1].}
\end{align*}
By taking complements, we see immediately that $F$ induces an isotopy of embeddings $\{i_{t}\}_{t\in[0,1]}$ via $i_{t}:=F_{t}\circ i_{0}$, as desired.
\end{proof}

We are now ready to prove Theorem \ref{theorem:tortellini_2_handle_formula}. For simplicity we treat the case where $\diamond=O$; the case $\diamond=T$ is proved similarly and is left to the interested reader. Fix a class $(\bm{\alpha},\eta)\in H_{2}^{L'}(X';\bb{Z})$. For each $\mbf{k}^{+},\mbf{k}^{-}\in\bb{N}^{m}$ such that $\mbf{k}^{+}-\mbf{k}^{-}=\bm{\alpha}$, we define a map
\[\wt{f}_{O,\mbf{k}^{\pm},\eta}:\ol{\skein}_{0}^{2,O}\big(\mbf{k}^{\pm};\eta_{\mbf{k}^{\pm}}\big)\big\{-|\mbf{k}^{+}|-|\mbf{k}^{-}|\big\}\to\ol{\skein}_{0}^{2,O}(X',L';(\bm{\alpha},\eta))\]
as follows: let
\[v=\sum_{k=1}^{N}a_{k}[F_{k}]\in\ol{\skein}_{0}^{2,O}\big(\mbf{k}^{\pm};\eta_{\mbf{k}^{\pm}}\big)\]
be a $\bb{Q}$-linear combination of some collection of type $O$ 1-dimensional input lasagna fillings $\{F_{k}\}=\{(\Sigma_{k},\{(B_{k,j},L_{k,j},v_{k,j})\})\}$ of $(X,K(\mbf{k}^{+},\mbf{k}^{-})\cup L)$. For each $k$, we have an identification
\[\del\Sigma_{k}=\sqcup_{j}(-L_{k,j})\sqcup K(\mbf{k}^{+},\mbf{k}^{-})\cup L.\]
By capping off the components of $\del\Sigma_{k}$ corresponding to the components of $K(\mbf{k}^{+},\mbf{k}^{-})$ with $k_{i}^{+}$ (respectively, $k_{i}^{-}$) positively-oriented (respectively, negatively oriented) core-parallel disks of the $i$-th 2-handle for each $i=1,\dots, m$, we obtain a 1-dimensional input lasagna filling $\wt{f}_{O,\mbf{k}^{\pm},\eta}(F_{k})$ of $(X',L')$. We then set $\wt{f}_{O,\mbf{k}^{\pm},\eta}(v):=\sum_{k=1}^{N}a_{k}[\wt{f}_{O,\mbf{k}^{\pm},\eta}(F_{k})]$. Now let
\[\wt{f}^{(\bm{\alpha},\eta)}_{O}:=\bigoplus_{\substack{\mbf{k}^{\pm}\in\bb{N}^{m} \\ \mbf{k}^{+}-\mbf{k}^{-}=\bm{\alpha}}}\wt{f}_{O,\mbf{k}^{\pm},\eta}.\]
To show that $\wt{f}^{(\bm{\alpha},\eta)}_{O}$ descends to a map
\[f^{(\bm{\alpha},\eta)}_{O}:\text{Coeq}(\Phi_{O}^{(\bm{\alpha},\eta)},\Psi_{O}^{(\bm{\alpha},\eta)})\to\ol{\skein}_{0}^{2,O}(X',L';(\bm{\alpha},\eta)),\]
it suffices to check that $\wt{f}^{(\bm{\alpha},\eta)}_{O}$ respects the braiding, annulus, and lasso relations. The fact that $\wt{f}^{(\bm{\alpha},\eta)}_{O}$ respects the braiding and annulus relations follows along the same lines as in (\cite{MN22}, Theorem 1.1), so we will omit the proof here.

For the lasso relation, let $\mbf{k}^{+},\mbf{k}^{-}\in\bb{N}^{m}$ and $\ell^{+},\ell^{-}\in\bb{N}$ be such that $\mbf{k}^{+}-\mbf{k}^{-}=\bm{\alpha}$ and $\ell^{+}-\ell^{-}=\alpha_{i}$. Let $\wt{F}=(\wt{\Sigma},\{(\wt{B}_{j},\wt{L}_{j},\wt{v}_{j})\})$ be a type $O$ 1-dimensional input lasagna filling of $(X\natural S^{1}\times B^{3},K(\mbf{k}^{\pm};\ell^{\pm};i)\cup L)$. Furthermore, let $\Phi^{(\bm{\alpha},\eta)}_{O}(F)$ (respectively, $\Psi^{(\bm{\alpha},\eta)}_{O}(F)$) denote the lasagna filling of $(X,K(\mbf{k}^{+},\mbf{k}^{-})\cup L)$ (respectively, $(X,K(\mbf{k}^{+}+(\ell^{+}-k_{i}^{+})\mbf{e}_{i},\mbf{k}^{-}+(\ell^{-}-k_{i}^{-})\mbf{e}_{i})\cup L)$) obtained by gluing the relative cobordism $(W_{i},\Sigma_{\mbf{k}^{\pm},\ell^{\pm},i})$ (respectively, $(W_{i},\Sigma_{\mbf{k}^{\pm},\ell^{\pm},i})$) to $\wt{F}$. Via the diffeotopy from Lemma \ref{lem:isotopy_of_embeddings} we obtain an isotopy from $\wt{f}^{(\bm{\alpha},\eta)}_{O}(\Phi^{(\bm{\alpha},\eta)}_{O}(\wt{F}))$ and $\wt{f}^{(\bm{\alpha},\eta)}_{O}(\Psi^{(\bm{\alpha},\eta)}_{O}(\wt{F}))$ as lasagna fillings of $(X',L')$. Hence $(\wt{f}^{(\bm{\alpha},\eta)}_{O}\circ\Phi^{(\bm{\alpha},\eta)}_{O})(v)=(\wt{f}^{(\bm{\alpha},\eta)}_{O}\circ\Psi^{(\bm{\alpha},\eta)}_{O})(v)$ for all $v\in\bigoplus_{i=1}^{m}\ul{\ol{\skein}}^{2,O,(\bm{\alpha},\alpha_{i})}_{0}(X\natural S^{1}\times B^{3};K(i),L;\eta_{i})$, as desired.

Next, we define an inverse
\[(f^{(\bm{\alpha},\eta)}_{O})^{-1}:\ol{\skein}_{0}^{2,O}(X',L';(\bm{\alpha},\eta))\to\text{Coeq}(\Phi^{(\bm{\alpha},\eta)}_{O},\Psi^{(\bm{\alpha},\eta)}_{O})\]
as follows: Given a type $O$ 1-dimensional input lasagna filling $F$ of $(X',L')$ with surface $\Sigma$ in skein grading $(\bm{\alpha},\eta)$, isotope the input 1-handlebodies of $F$ to be inside $\text{int}(X)\subset X'$, and isotope the surface $\Sigma$ so that its intersection with the 2-handles is precisely a union of core parallel disks. By removing these 2-handles along with these core-parallel disks, we obtain a lasagna filling $(f^{(\bm{\alpha},\eta)}_{O})^{-1}(F)$ of $(X,K(\mbf{k}^{+},\mbf{k}^{-})\cup L)$ for some $\mbf{k}^{\pm}\in\bb{N}^{m}$ satisfying $\mbf{k}^{+}-\mbf{k}^{-}=\bm{\alpha}$.

In order to show that $(f^{(\bm{\alpha},\eta)}_{O})^{-1}$ is well-defined, let 
\[\{F_{t}\}_{t\in[0,1]}=\{(\Sigma_{t},\{(B_{j,t},L_{j,t},v_{j,t})\}_{j\in{J}})\}_{t\in[0,1]}\]
be a 1-parameter family of type $O$ 1-dimensional input lasagna fillings of $(X',L')$ such that for each $k\in\{0,1\}$ the following holds:
\begin{enumerate}[label=(C-\arabic*)]
    \item \label{cond:isotopy_1} $\{B_{j,k}\}_{j\in J}\subset\text{int}(X)\subset X'$. 
    \item \label{cond:isotopy_2} The intersection of $\Sigma_{k}$ with the 2-handles is precisely a union of core parallel disks. 
\end{enumerate}
Let $G_{1},\dots,G_{m}$ denote the co-cores of the 2-handles, so that $\Sigma_{k}\cap{G_{i}}$ consists of $k_{i,k}^{+}$ positively-oriented (respectfully, $k_{i,k}^{-}$ negatively-oriented) points for some $k_{i,k}^{\pm}\in\bb{N}$, for $k=0,1$ and for each $i\in\{1,\dots,m\}$. Define submanifolds $Y=\bigcup_{t\in{[0,1]}}(\{t\}\times{\Sigma_{t})}$ and $B=\bigsqcup_{j\in{J}}\bigcup_{t\in{[0,1]}}(\{t\}\times{B_{j,t}})$, then $Y$ and $B$ generically intersect $[0,1]\times{G_{i}}$ in a 1- and 3-manifold respectively. Deformation retracting $B$ to its core yields a 1-parameter family of 1-complexes $\Gamma(B)=\bigsqcup_{j\in{J}}\bigcup_{t\in{[0,1]}}(\{t\}\times{\Gamma(B_{j,t})})$ intersecting $[0,1]\times{G_{i}}$ in a finite number of points (see Remark \ref{rem:core_1d_input}). After a possible perturbation of $\Gamma(B)$ and $Y$ rel $\partial$, we may assume that $\Gamma_{i}:=(Y\cup \Gamma(B))\cap([0,1]\times{G_{i}})$ is an embedded graph in $[0,1]\times G_{i}\approx[0,1]\times D^{2}$ such that if $\pi_{i}$ denotes the composition
\[\Gamma_{i}\hookrightarrow{[0,1]\times G_{i}}\twoheadrightarrow[0,1]\]
of the inclusion map followed by projection onto the first factor, then $\pi_{i}$ is a local diffeomorphism outside of finitely many critical points of the forms depicted in Figure \ref{fig:cocorecrits}.
\begin{figure}
    \centering
    \includegraphics[width=0.95\linewidth]{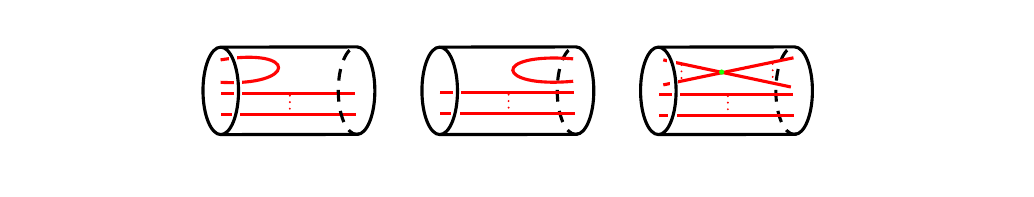}
    \caption{(From left to right) The cap, cup and switch critical points in $[0,1]\times{G_{i}}$.}
    \label{fig:cocorecrits}
\end{figure}

Similar to (\cite{MN22}, Proof of Theorem 1.1), in the intervals between the critical points, the intersections between $\Sigma_{t}$ and $G_{i}$ change by an isotopy of $G_{i}$ and hence correspond to the braid group action, and passing through the cup and cap critical points correspond to the $\psi_{O,i}^{[j]}$ maps. In the 1-dimensional inputs setting, we also have the appearance of switch critical points, which correspond to pieces of 1-dimensional inputs passing through $G_{i}$. 

The topology of the 1-dimensional inputs $\{B_{j,t}\}$ and surfaces $\{\Sigma_{j,t}\}$ may apriori be arbitrarily complicated. However we have the following lemma which allows us to greatly simplify the topology of the situation in the presence of a switch point:

\begin{lemma}
\label{lem:B0}
Let $\{F_{t}\}_{t\in[0,1]}=\{(\Sigma_{t},\{(B_{j,t},L_{j,t},v_{j,t})\}_{j\in J})\}_{t\in[0,1]}$ be a 1-parameter family of type $O$ 1-dimensional input lasagna fillings of $(X',L')$ satisfying conditions \ref{cond:isotopy_1} and \ref{cond:isotopy_2}, such that it has a single switch critical point with respect to $\pi_{i}:\Gamma_{i}\to[0,1]$ for some $i\in\{1,\dots,m\}$ and no other critical points. Furthermore, let $\mbf{k}^{+},\mbf{k}^{-}\in\bb{N}^{m}$ and $\ell^{+},\ell^{-}\in\bb{N}$ be such that:
\begin{itemize}
    \item $\mbf{k}^{+}-\mbf{k}^{-}=\bm{\alpha}$ and $\ell^{+}-\ell^{-}=\alpha_{i}$.
    \item For $k=0,1$ and $j\neq i$, $\Sigma_{k}\cap{G_{j}}$ consists of $k_{j}^{+}$ positively-oriented points and $k_{j}^{-}$ negatively-oriented points.
    \item $\Sigma_{0}\cap{G_{i}}$ consists of $k_{i}^{+}$ positively-oriented points and $k_{i}^{-}$ negatively-oriented points.
    \item $\Sigma_{1}\cap{G_{i}}$ consists of $\ell^{+}$ positively-oriented points and $\ell^{-}$ negatively-oriented points.
\end{itemize}
Then the following statements are true:
\begin{enumerate}
    \item There exists a type $O$ 1-dimensional inputs lasagna filling $\wt{F}$ of $(X\natural S^{1}\times B^{3},K(\mbf{k}^{\pm};\ell^{\pm};i)\cup L)$ such that under the enclosement relation $\{F_{t}\}_{[t\in[0,1]}$ is equivalent to the 1-parameter family of lasagna fillings $\{\wt{F}'_{t}\}_{t\in[0,1]}=\{(\wt{\Sigma}'_{t},\{(\wt{B}'_{j,t},\wt{L}'_{j,t},\wt{v}'_{j,t})\}_{j\in J})\}_{t\in[0,1]}$ induced by the isotopy of lasagna fillings from $\wt{f}_{O}^{(\bm{\alpha},\eta)}(\Phi_{O}^{(\bm{\alpha},\eta)}(\wt{F}))$ to $\wt{f}_{O}^{(\bm{\alpha},\eta)}(\Psi_{O}^{(\bm{\alpha},\eta)}(\wt{F}))$ furnished by Lemma \ref{lem:isotopy_of_embeddings}.
    \item Any two lasagna fillings $\wt{F}_{0}$ and $\wt{F}_{1}$ of $(X\natural S^{1}\times B^{3},K(\mbf{k}^{\pm};\ell^{\pm};i)\cup L)$ satisfying (1) are related by a 1-parameter family of lasagna fillings $\{\wt{F}_{t}\}_{t\in[0,1]}$ of $(X\natural S^{1}\times B^{3},K(\mbf{k}^{\pm};\ell^{\pm};i)\cup L)$.
\end{enumerate}
\end{lemma}

\begin{proof}
Let $H_{i}$ denote the $i$-th 2-handle, and fix a homeomorphism $H_{i}\cong D^{2}\times D^{2}$ where $D^{2}\times\{0\}$ corresponds to the core and $\{0\}\times D^{2}$ corresponds to the co-core. Furthermore, let $t_{0}\in[0,1]$ and $j_{0}\in J$ be such that the intersection of the one parameter family of cores $\Gamma(B)$ with $[0,1]\times G_{i}$ lies in $\Gamma(B_{j_{0},t_{0}})$. By perturbing if necessary, we can assume that $\Gamma(B_{j_{0},t_{0}})\cap G_{i}$ consists of a single point $p\in \Gamma(B_{j_{0},t_{0}})\cong\vee^{n}S^{1}$ away from the basepoint, i.e., $p$ is contained in a single wedge summand. Via an isotopy of one-parameter families, we can assume that there exists $\epsilon>0$ and some $a^{+},a^{-}\in\bb{N}$ such that the intersection $(Y\cup B)\cap([0,1]\times H_{i})$ is of the following form:
\begin{itemize}
    \item For each $t\in(0,\frac{\epsilon}{2})$:
    \begin{itemize}
        \item $B\cap(\{t\}\times H_{i})=\emptyset$.
        \item $Y\cap(\{t\}\times H_{i})=\Sigma_{t}\cap H_{i}$ is homeomorphic to a union of $k_{i}^{+}$ positively oriented and $k_{i}^{-}$ negatively oriented core-parallel disks.
    \end{itemize}
    \item For $t\in(\epsilon,1-\epsilon)$:
    \begin{itemize}
        \item $B\cap(\{t\}\times H_{i})=B_{j_{0},t}\cap H_{i}\cong[0,1]\times B^{3}$.
        \item $L_{j_{0},t}\cap H_{i}$ is homeomorphic to the identity braid in $\del B_{j_{0},t}\cap H_{i}\cong[0,1]\times S^{2}$ with $k_{i}^{+}+\ell^{-}-a^{+}-a^{-}$ positively oriented strands and $k_{i}^{-}+\ell^{+}-a^{+}-a^{-}$ negatively oriented strands.
        \item $\Sigma_{t}\cap H_{i}$ is homeomorphic to a disjoint union of the following pieces:
        \begin{itemize}
            \item $k_{i}^{+}-a^{+}$ positively oriented half disks $\{D^{+}_{L,i}\}_{i=1}^{k_{i}^{+}-a^{+}}$.
            \item $k_{i}^{-}-a^{-}$ negatively oriented half disks $\{D^{-}_{L,i}\}_{i=1}^{k_{i}^{-}-a^{-}}$.
            \item $\ell^{+}-a^{+}$ positively oriented half disks $\{D^{+}_{R,i}\}_{i=1}^{\ell^{+}-a^{+}}$.
            \item $\ell^{-}-a^{-}$ negatively oriented half disks $\{D^{-}_{R,i}\}_{i=1}^{\ell^{-}-a^{-}}$.
            \item $a^{+}$ positively oriented disks $\{D^{+}_{i}\}_{i=1}^{a^{+}}$
            \item $a^{-}$ negatively oriented disks $\{D^{-}_{i}\}_{i=1}^{a^{-}}$
        \end{itemize}
        \item for $t\in(\epsilon,t_{0}-\epsilon)$:
        \begin{itemize}
            \item $B\cap(\{t\}\times G_{i})=B_{j_{0},t}\cap G_{i}=\emptyset$.
            \item $\Sigma_{t}\cap G_{i}$ consists of a single point in each of the disks in $\{D^{+}_{L,i}\}_{i=1}^{k_{i}^{-}-a^{-}}\cup\{D^{-}_{L,i}\}_{i=1}^{k_{i}^{-}-a^{-}}\cup\{D^{+}\}_{i=1}^{a^{+}}\cup\{D^{-}\}_{i=1}^{a^{-}}$.
        \end{itemize}
        \item for $t\in(t_{0}-\epsilon,t_{0}+\epsilon)$:
        \begin{itemize}
            \item $B\cap(\{t\}\times G_{i})=B_{j_{0},t}\cap G_{i}\cong D^{2}$.
            \item $\Sigma_{t}\cap G_{i}$ onsists of a single point in each of the disks in $\{D^{+}\}_{i=1}^{a^{+}}\cup\{D^{-}\}_{i=1}^{a^{-}}$.
        \end{itemize}
        \item for $t\in(t_{0}+\epsilon,1-\epsilon)$:
        \begin{itemize}
            \item $B\cap(\{t\}\times G_{i})=B_{j_{0},t}\cap G_{i}=\emptyset$.
            \item $\Sigma_{t}\cap G_{i}$ consists of a single point in each of the disks in $\{D^{+}_{R,i}\}_{i=1}^{\ell^{-}-a^{-}}\cup\{D^{-}_{R,i}\}_{i=1}^{\ell^{-}-a^{-}}\cup\{D^{+}\}_{i=1}^{a^{+}}\cup\{D^{-}\}_{i=1}^{a^{-}}$.
        \end{itemize}
    \end{itemize}
    \item For each $t\in(1-\frac{\epsilon}{2},1)$:
    \begin{itemize}
        \item $B\cap(\{t\}\times H_{i})=\emptyset$.
        \item $Y\cap(\{t\}\times H_{i})=\Sigma_{t}\cap H_{i}$ is homeomorphic to a union of $\ell^{+}$ positively oriented and $\ell^{-}$ negatively oriented core-parallel disks.
    \end{itemize}
\end{itemize}
See Figure \ref{fig:1d_input_passing_through} for a depiction of the above. 

We proceed in two steps. First, observe that we can use the enclosement relation to ensure that $a^{+}=a^{-}=0$ (see Figure \ref{fig:swallowing_sheets}).

Next, consider the decomposition $X\natural S^{1}\times B^{3}=X\cup_{S^{0}\times S^{2}}D^{1}\times S^{2}$, and let $i_{t}:X\natural(S^{1}\times B^{3})\hookrightarrow X'$ be a 1-parameter family of embeddings such that:
\begin{itemize}
    \item for all $t\in[0,1]$, $i_{t}|_{X}$ is given by the inclusion $X\hookrightarrow X\cup_{\del X}(\del X\times[0,1])\approx X$ for some fixed collar neighborhood of $\del X$.
    \item $i_{t}|_{D^{1}\times S^{2}}$ contains a neighborhood of $B_{j_{0},t}\cap H_{i}$ for all $t\in(\epsilon,1-\epsilon)$.
\end{itemize}
By construction we have that $B\subset\cup_{t\in[0,1]}\{t\}\times\im(i_{t})$. Let $F'_{t}$ be the one parameter family of lasagna fillings of $(X\natural S^{1}\times B^{3},K(\mbf{k}^{\pm};\ell^{\pm};i)\cup L)$ obtained as the set of preimages under $i_{t}$ of the intersection of $F_{t}\cap\im(i_{t})$ for all $t\in[0,1]$. As these lasagna fillings only vary by an isotopy as $t$ ranges from $0$ to $1$, we can make a continuous identification of the family $\{F'_{t}\}_{t\in[0,1]}$ with the fixed lasagna filling $\wt{F}:=F'_{0}$. By inspection, we see that $\wt{F}$ satisfies (1) as in the statement of the lemma. 

Furthermore, (2) holds by the following observation: Let $\wt{F}_{0}$ and $\wt{F}_{1}$ be two lasagna fillings of $(X\natural S^{1}\times B^{3},K(\mbf{k}^{\pm};\ell^{\pm};i)\cup L)$ furnished by (1), with corresponding 1-parameter families $\{\wt{F}'_{0,t}\}_{t\in[0,1]}$ and $\{\wt{F}'_{1,t}\}_{t\in[0,1]}$, respectively, of lasagna fillings of $(X',L')$. Then any two choices of isotopies of one-parameter families arranging for the intersection $(Y\cup B)\cap([0,1]\times H_{i})$ to be of the above prescribed form induce a 2-parameter family of lasagna fillings $\{\wt{F}'_{s,t}\}_{(s,t)\in[0,1]^{2}}$ connecting $\{\wt{F}'_{0,t}\}_{t\in[0,1]}$ and $\{\wt{F}'_{1,t}\}_{t\in[0,1]}$, and hence yield a one-parameter family $\{\wt{F}_{s}\}_{s\in[0,1]}$ connecting $\wt{F}_{0}$ and $\wt{F}_{1}$, as desired.

\begin{center}
\begin{figure}
    \includegraphics[width=16cm]{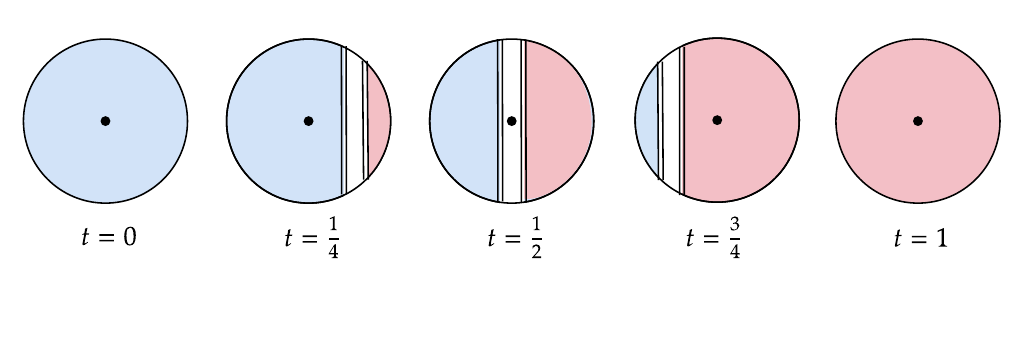}
    \caption{A schematic of a piece of $1$-dimensional input passing through a co-core of the $2$-handle $H_{i}\approx D^{2}\times D^{2}$ after projecting down to the core disk $D^{2}\times\{0\}$, as in the proof of Lemma \ref{lem:B0}. The co-core $G_{i}$ is represented by the central dot $\{0\}\times\{0\}$, the piece of $1$-dimensional input passing through $H_{i}$ is depicted by the white strip, the half disks $D^{\pm}_{L,i}$ and $D^{\pm}_{R,i}$ are shown in blue and red, respectively, and the disks $D^{\pm}_{i}$ are not depicted.}
    \label{fig:1d_input_passing_through}
\end{figure}
\end{center}
\begin{center}
\begin{figure}
    \includegraphics[width=15cm]{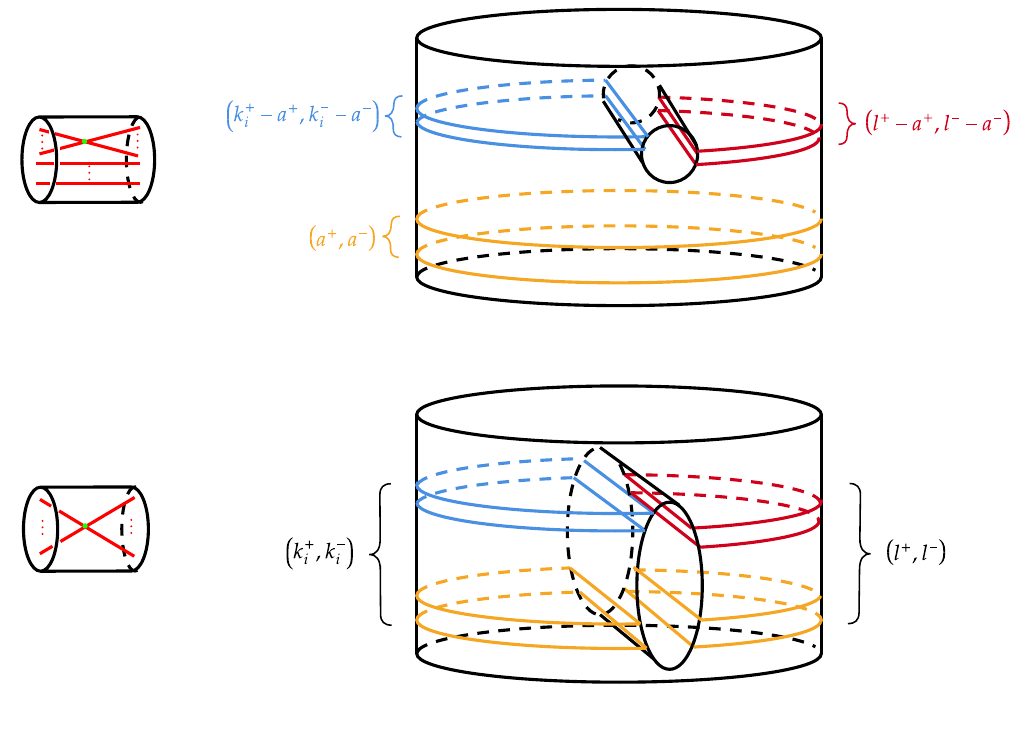}
    \caption{Using the enclosement relation to subsume the remaining sheets into the ``switch" of the switch critical point, as in the proof of Lemma \ref{lem:B0}.}
    \label{fig:swallowing_sheets}
\end{figure}
\end{center}

\end{proof}

\begin{proof}[Proof of Theorem \ref{theorem:tortellini_2_handle_formula}]
Lemma \ref{lem:B0} along with the preceding discussion implies that $(f^{(\bm{\alpha},\eta)}_{O})^{-1}$ is well-defined, and the reader can check that $(f^{(\bm{\alpha},\eta)}_{O})$ and $(f^{(\bm{\alpha},\eta)}_{O})^{-1}$ as defined above are indeed inverses of each other.
\end{proof}

\subsection{\texorpdfstring{Properties of $\Phi$ and $\Psi$}{Properties of Φ and Ψ}}
\label{subsec:Phi_Psi_properties}

In this section we establish some properties of the maps $\Phi_{\diamond}^{(\bm{\alpha},\eta)}$ and $\Psi_{\diamond}^{(\bm{\alpha},\eta)}$ from Theorem \ref{theorem:tortellini_2_handle_formula}, with components defined by (\ref{eq:Phi}) and (\ref{eq:Psi}), respectively. 

Let $i\in\{1,\dots,m\}$, and fix $\mbf{k}^{+},\mbf{k}^{-}\in\bb{N}^{m}$, $\ell^{+},\ell^{-}\in\bb{N}$ such that 
\begin{itemize}
    \item $k_{i}^{+}-k_{i}^{-}=\ell^{+}-\ell^{-}$ if $\diamond=O$,
    \item $k_{i}^{+}-k_{i}^{-}\equiv \ell^{+}-\ell^{-}\pmod{2}$ if $\diamond=T$.
\end{itemize}
First, note that the decompositions (\ref{eq:W_i_cabled_decomp}) and (\ref{eq:W'_i_cabled_decomp}) give us factorizations
\begin{align*}
    &\Phi_{\diamond,\mbf{k}^{\pm},\ell^{\pm},i}=\ol{\skein}_{0}^{2,\diamond}(Y\times[0,1],C_{\ell^{\pm}})\circ\ol{\skein}_{0}^{2,\diamond}(Z_{i},S_{\mbf{k}^{\pm},\ell^{\pm},i}), \\
    &\Psi_{\diamond,\mbf{k}^{\pm},\ell^{\pm},i}=\ol{\skein}_{0}^{2,\diamond}(Y\times[0,1],C_{k_{i}^{\pm}})\circ\ol{\skein}_{0}^{2,\diamond}(Z'_{i},S'_{\mbf{k}^{\pm},\ell^{\pm},i}).
\end{align*}
First, we calculate the maps $\ol{\skein}_{0}^{2,\diamond}(Z_{i},S_{\mbf{k}^{\pm},\ell^{\pm},i})$ and $\ol{\skein}_{0}^{2,\diamond}(Z'_{i},S'_{\mbf{k}^{\pm},\ell^{\pm},i})$. Observe that $Z_{i}$ is given by an \emph{elementary} cancelling 2-handle attachment, i.e., a 2-handle attachment along a longitude in an $S^{1}\times S^{2}$-summand. In the special case where $X=\natural^{n}S^{1}\times B^{3}$, the cobordism $(Z_{i},S_{\mbf{k}^{\pm},\ell^{\pm},i})$ corresponds to the elementary morphism (v) of (\cite{RSWWZ25}, Proposition 6.3). Hence
\begin{equation}
\label{eq:Z_i_equals_counit}
    \ol{\skein}_{0}^{2,\diamond}(Z_{i},S_{\mbf{k}^{\pm},\ell^{\pm},i})=\KhR_{2,\diamond}^{-}(Z_{i},S_{\mbf{k}^{\pm},\ell^{\pm},i})=\epsilon
\end{equation}
where
\begin{equation}
\label{eq:counit_iota}
    \epsilon:\KhR_{2,\diamond}^{-}\big(K(\mbf{k}^{\pm};\ell^{\pm};i)\big)\to\KhR_{2,\diamond}^{-}\big(K(\mbf{k}^{+},\mbf{k}^{-})\big)\otimes\KhR_{2}\big(U(\ell^{+},\ell^{-})\big)
\end{equation}
is the counit map which deletes the projector (see Remark \ref{def:Roz-cobar}).

Next, note that the 2-handle cobordism $(Z'_{i},S'_{\mbf{k}^{\pm},\ell^{\pm},i})$ is (in general) not an elementary cancelling 2-handle attachment. However, we can decompose it as a composition of a diffeomorphism of the boundary manifold/link pair followed by an elementary canceling 2-handle attachment.

\begin{lemma}
\label{lem:diff_of_pairs_Psi_swap}
There exists a diffeomorphism of pairs
\[\varphi_{i}:(Y\#S^{1}\times S^{2},K(i)\cup L)\xrightarrow{\approx} (Y\#S^{1}\times S^{2},K(i)\cup L)\]
which swaps $K(i)_{i}$ and $K(i)_{m+1}$, respecting framings, restricts to the identity on all other components of $K(i)\cup L$.
\end{lemma}

\begin{proof}
Consider the decomposition of $Y\#S^{1}\times S^{2}$ given as follows: let $\nu(K_{i})$ denote a tubular neighborhood of $K_{i}\subset Y$, and fix an identification $\nu(K_{i})=D^{2}\times S^{1}$ so that $K_{i}\subset\nu(K_{i})$ is identified with the core circle $\{0\}\times S^{1}$, and the longitude $\{1\}\times S^{1}$ agrees with the framing of $K_{i}$. Let $w$ and $z$ be the two diametrically opposite basepoints on $K_{i}$ corresponding to $\{0\}\times\{1\}$ and $\{0\}\times\{-1\}$ under the above identification, and consider the 3-ball neighborhoods $\nu(w)$ and $\nu(z)$ of $w$ and $z$ in $\nu(K_{i})$ which under the above identification correspond to $\epsilon$-balls of some fixed radius $\epsilon\in(0,1)$ centered at $w$ and $z$, respectively. Then we can write
\begin{equation}
\label{eq:Y_connect_sum_S^1xS^2_decomp}
    Y\#S^{1}\times S^{2}=A\cup_{S^{1}\times S^{1}}B\cup_{S^{1}\times S^{1}}C\cup_{S^{0}\times S^{2}}D
\end{equation}
where
\begin{enumerate}
    \item $A=Y\setminus\mathring{\nu}(K_{i})$,
    \item $B=\del\nu(K_{i})\times[0,1]\cong S^{1}\times S^{1}\times[0,1]$,
    \item $C=\nu(K_{i})\setminus(\mathring{\nu}(w)\cup\mathring{\nu}(z))\cong (D^{2}\times S^{1})\setminus(S^{0}\times B^{3})$,
    \item $D=[-1,1]\times S^{2}$.
\end{enumerate}
Furthermore, we can identify $(K(i)\cup L)\setminus(K(i)_{i}\cup K(i)_{m+1})$ with its image in $A$, and identify:
\begin{itemize}
    \item $K(i)_{i}$ with the union of the tangles 
    \begin{align*}
        &T_{i}:=(\{0\}\times[\tfrac{\pi}{2},\tfrac{3\pi}{2}])\setminus(S^{0}\times B^{3})\subset C, & &T'_{i}:=[-1,1]\times\{a\}\subset D.
    \end{align*}
    \item $K(i)_{m+1}$ with the union of the tangles 
    \begin{align*}
        &T_{m+1}:=(\{0\}\times[-\tfrac{\pi}{2},\tfrac{\pi}{2}])\setminus(S^{0}\times B^{3})\subset C, & &T'_{m+1}:=[-1,1]\times\{b\}\subset D,
    \end{align*}
\end{itemize}
where $a,b\in S^{2}$ denote the north and south poles, respectively (see Figure \ref{fig:decomposition}). Consider the diffeomorphism
\begin{align*}
    f:B&\to B \\
    (\mu,\lambda,s)&\mapsto(\mu,\lambda+s\pi,s)
\end{align*}
where $\mu$ and $\lambda$ denote the meridional and longitudinal coordinates on $\del\nu(K_{i})\cong S^{1}\times S^{1}=(\bb{R}/2\pi\bb{Z})^{2}$, and $s$ denotes the coordinate of the $[0,1]$-factor. Note that $f$ fixes the incoming boundary $\del\nu(K_{i})\times\{0\}$ and acts by a $\pi$ rotation in the longitudinal direction of the outgoing boundary $\del\nu(K_{i})\times\{1\}$ (see Figure \ref{fig:decomposition}). Next, let
\begin{align*}
    g:C&\to C & h:D&\to D \\
    (x,\lambda)&\mapsto(x,\lambda+\pi), & (t,z)&\mapsto(-t,\tfrac{1}{z}),
\end{align*}
where we interpret the map $z\mapsto\frac{1}{z}$ on $S^{2}=\bb{C}P^{1}$ as reflection across the equator. Note that $g$ and $h$ extend to diffeomorphisms of pairs
\begin{align*}
    g:(C,T_{i}\cup T_{m+1})&\to (C,T_{i}\cup T_{m+1}) & h:(D,T'_{i}\cup T'_{m+1})&\to(D,T'_{i}\cup T'_{m+1})
\end{align*}
which exchange $T_{i}$ and $T_{m+1}$ (respectively, exchange $T'_{i}$ and $T'_{m+1}$. We then define $\varphi_{i}:Y\#S^{1}\times S^{2}\to Y\#S^{1}\times S^{2}$ by
\begin{align*}
    &\varphi_{i}|_{A}=\id_{A} & &\varphi_{i}|_{B}=f & &\varphi_{i}|_{C}=g & &\varphi_{i}|_{D}=h
\end{align*}
under the decomposition (\ref{eq:Y_connect_sum_S^1xS^2_decomp}). By construction we see that $\varphi_{i}$ induces a diffeomorphism of pairs $(Y\#S^{1}\times S^{2},K(i)\cup L)\xrightarrow{\approx} (Y\#S^{1}\times S^{2},K(i)\cup L)$ that satisfies the desired properties as in the statement of the lemma.

\begin{center}
\begin{figure}
    \includegraphics[width=14cm]{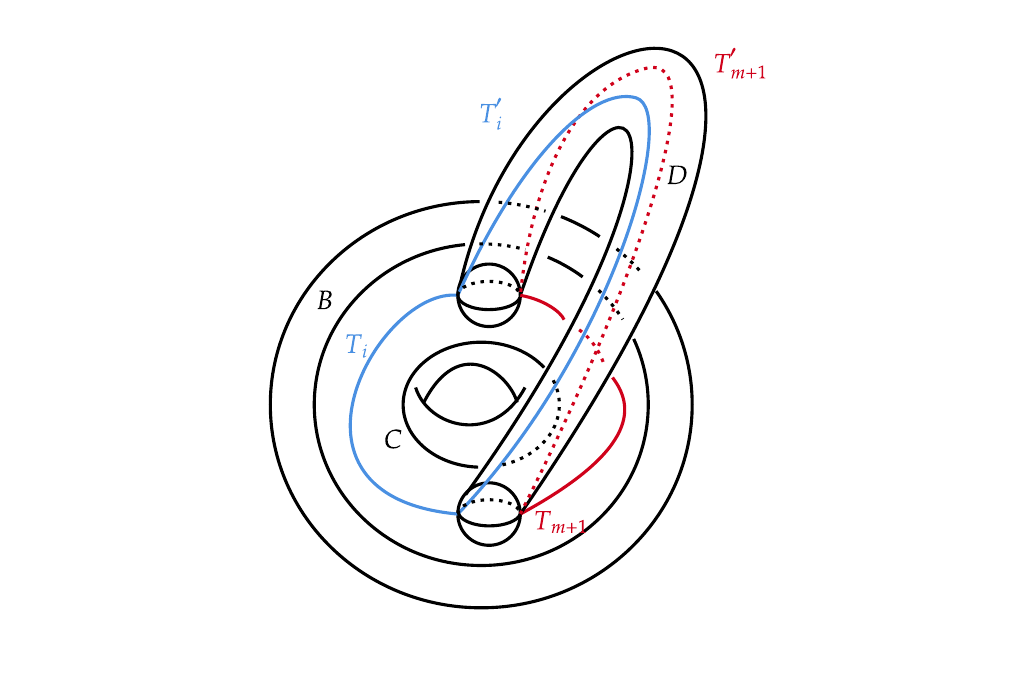}
    \caption{The decomposition of the union of $\nu(K_{i})$ and a $4$-dimensional $1$-handle as specified by (\ref{eq:Y_connect_sum_S^1xS^2_decomp}).}
    \label{fig:decomposition}
\end{figure}
\end{center}

\end{proof}

From Lemma \ref{lem:diff_of_pairs_Psi_swap} we immediately obtain a diffeomorphism of pairs
\begin{equation*}
    \varphi_{\mbf{k}^{\pm},\ell^{\pm},i}:\big(Y\#S^{1}\times S^{2},K(\mbf{k}^{\pm};\ell^{\pm};i)\cup L\big)\xrightarrow{\approx} \big(Y\#S^{1}\times S^{2},K(\mbf{k}^{\pm}+(\ell^{\pm}-k_{i}^{\pm})\mbf{e}_{i};k_{i}^{\pm};i)\cup L\big)
\end{equation*}
which sends:
\begin{itemize}
    \item $K(i)_{i}(k_{i}^{+},k_{i}^{-})$ to $K(i)_{m+1}(k_{i}^{+},k_{i}^{-})$,
    \item $K(i)_{m+1}(\ell^{+},\ell^{-})$ to $K(i)_{i}(\ell^{+},\ell^{-})$, and
    \item every other component of $K(\mbf{k}^{\pm};\ell^{\pm};i)\cup L$ to themselves,
\end{itemize}
and which is well-defined up to composition with a diffeomorphism of pairs which braids cable components.

\begin{proposition}
\label{prop:Psi_relate_to_Phi}
We have an equality of maps
\[\ol{\skein}_{0}^{2,\diamond}(Z'_{i},S'_{\mbf{k}^{\pm},\ell^{\pm},i})=\ol{\skein}_{0}^{2,\diamond}(Z_{i},S_{\mbf{k}^{\pm}+(\ell^{\pm}-k_{i}^{\pm})\mbf{e}_{i},k_{i}^{\pm},i})\circ\ol{\skein}_{0}^{2,\diamond}(\varphi_{\mbf{k}^{\pm},\ell^{\pm},i}).\]
Hence in particular, we have that
\[\Psi_{\diamond,\mbf{k}^{\pm},\ell^{\pm},i}=\Phi_{\diamond,\mbf{k}^{\pm}+(\ell^{\pm}-k_{i}^{\pm})\mbf{e}_{i},k_{i}^{\pm},i}\circ\ol{\skein}_{0}^{2,\diamond}(\varphi_{\mbf{k}^{\pm},\ell^{\pm},i}).\]
\end{proposition}

\begin{proof}
Consider the cobordism $(Z''_{i},S''_{\mbf{k}^{\pm},\ell^{\pm},i})$ obtained as the composition of the mapping cylinder of $\varphi_{\mbf{k}^{\pm},\ell^{\pm},i}$ and the cobordism $(Z_{i},S_{\mbf{k}^{\pm}+(\ell^{\pm}-k_{i}^{\pm})\mbf{e}_{i},k_{i}^{\pm},i})$ given by a 2-handle cobordism given by attachment along $\gamma$. Observe that we can write
\[Z''_{i}=\big((Y\#S^{1}\times S^{2})\times[0,1]\big)\cup_{\varphi_{\mbf{k}^{\pm},\ell^{\pm},i}}\big((Y\#S^{1}\times S^{2})\times[0,1]\big)\cup_{\nu(\gamma)}(D^{2}\times D^{2})\]
By performing a deformation retraction of the second factor of $(Y\#S^{1}\times S^{2})\times[0,1]$ onto $(Y\#S^{1}\times S^{2})\times\{0\}$, we see that
\[Z''_{i}=\big((Y\#S^{1}\times S^{2})\times[0,1]\big)\cup_{\nu(\gamma')}(D^{2}\times D^{2})=Z'_{i},\]
as desired.
\end{proof}

For the moment, let us specialize to the case where $X=\natural^{n}S^{1}\times B^{3}$. By (\ref{eq:Z_i_equals_counit}) and Proposition \ref{prop:Psi_relate_to_Phi} we have that
\begin{equation}
\label{eq:Z'_i_connected_sums_S_1xS_2}
    \KhR_{2,\diamond}^{-}(Z'_{i},S'_{\mbf{k}^{\pm},\ell^{\pm},i})=\epsilon\circ\KhR_{2,\diamond}^{-}(\varphi_{\mbf{k}^{\pm},\ell^{\pm},i}).
\end{equation}

In order to compute the map $\KhR_{2,\diamond}^{-}(\varphi_{\mbf{k}^{\pm},\ell^{\pm},i})$ in practice, i.e., using diagrammatics,
it will be helpful to be able to replace $\varphi_{\mbf{k}^{\pm},\ell^{\pm},i}$ instead with a diffeomorphism of pairs $\psi_{\mbf{k}^{\pm},\ell^{\pm},i}$ satisfying the same abstract properties as $\varphi_{\mbf{k}^{\pm},\ell^{\pm},i}$. Before spelling out the necessary conditions on $\psi_{\mbf{k}^{\pm},\ell^{\pm},i}$, the following lemma will be useful: 

\begin{lemma}
\label{lem:diffeos_identity_on_homology}
Let $L\subset\#^{n}S^{1}\times S^{2}$ be a link, and suppose
\begin{align*}
    &f:(\#^{n}S^{1}\times S^{2},L)\xrightarrow{\cong}(\#^{n}S^{1}\times S^{2},f(L)) & &g:(\#^{n}S^{1}\times S^{2},L)\xrightarrow{\cong}(\#^{n}S^{1}\times S^{2},g(L))
\end{align*}
are diffeomorphisms of pairs such that $f(L)=g(L)$. Furthermore, suppose the
composite map
\[g^{-1}\circ f:(\#^{n}S^{1}\times S^{2},L)\xrightarrow{\cong}(\#^{n}S^{1}\times S^{2},L)\]
satisfies the following properties:
\begin{enumerate}
    \item $(g^{-1}\circ f)|_{\nu(L)}=\id$ for some tubular neighborhood $\nu(L)$ of $L$.
    \item $g^{-1}\circ f$ acts trivially on $H_{*}(\#^{n}S^{1}\times S^{2};\bb{Z})$.
\end{enumerate}
Then
\[\KhR_{2,\diamond}^{-}(f)=\KhR_{2,\diamond}^{-}(g)\]
for $\diamond\in\{O,T\}$.
\end{lemma}

\begin{proof}
Via an isotopy, we can assume $L$ and $f(L)=g(L)$ are admissible. Assumption (2) combined with a classical theorem of Laudenbach \cite{Lau74} implies that $g^{-1}\circ f$ must be isotopic to some (possibly empty) composition $h$ of Gluck twist diffeomorphisms, and by assumption (1) we can assume that $h$ is supported away from $L$. Hence $h\circ g^{-1}\circ f\sim\id$. Furthermore, by the description of $h$ we can conclude that it induces the identity map on $\KhR_{2,\diamond}^{-}$. Hence by the functoriality of $\KhR_{2,\diamond}^{-}$ under diffeomorphisms (\cite{RSWWZ25}, Theorem 7.6) it follows that
\begin{align*}
    \id&=\KhR_{2,\diamond}^{-}(h\circ g^{-1}\circ f)=\KhR_{2,\diamond}^{-}(h)\circ\big(\KhR_{2,\diamond}^{-}(g)\big)^{-1}\circ\KhR_{2,\diamond}^{-}(f) \\
    &=\big(\KhR_{2,\diamond}^{-}(g)\big)^{-1}\circ\KhR_{2,\diamond}^{-}(f)\implies\KhR_{2,\diamond}^{-}(f)=\KhR_{2,\diamond}^{-}(g),
\end{align*}
as desired.
\end{proof}

We now return to our task at hand:

\begin{proposition}
\label{prop:decompose_diffeomorphism}
Let
\begin{equation*}
    \psi_{\mbf{k}^{\pm},\ell^{\pm},i}:\big(Y\#S^{1}\times S^{2},K(\mbf{k}^{\pm};\ell^{\pm};i)\cup L\big)\xrightarrow{\approx} \big(Y\#S^{1}\times S^{2},K(\mbf{k}^{\pm}+(\ell^{\pm}-k_{i}^{\pm})\mbf{e}_{i};k_{i}^{\pm};i)\cup L\big)
\end{equation*}
be a diffeomorphism of pairs which sends
\begin{itemize}
    \item $K(i)_{i}(k_{i}^{+},k_{i}^{-})$ to $K(i)_{m+1}(k_{i}^{+},k_{i}^{-})$,
    \item $K(i)_{m+1}(\ell^{+},\ell^{-})$ to $K(i)_{i}(\ell^{+},\ell^{-})$, and
    \item every other component of $K(\mbf{k}^{\pm};\ell^{\pm};i)\cup L$ to themselves,
\end{itemize}
and suppose $\psi_{\mbf{k}^{\pm},\ell^{\pm},i}$ decomposes as
\[\psi_{\mbf{k}^{\pm},\ell^{\pm},i}=\psi'_{\mbf{k}^{\pm},\ell^{\pm},i}\circ\sf{inv}_{n+1}\]
where:
\begin{enumerate}
    \item The map
    \[\sf{inv}_{n+1}:\#^{n+1}S^{1}\times S^{2}\to\#^{n+1}S^{1}\times S^{2}\]
    denotes the diffeomorphism which inverts the final $S^{1}\times S^{2}$ summand by reflecting both the $S^{1}$- and $S^{2}$-factors (see \cite{RSWWZ25} Proposition 5.5 (iii)).
    \item The map
    \begin{align*}
        &\psi'_{\mbf{k}^{\pm},\ell^{\pm},i}:\big(\#^{n+1}S^{1}\times S^{2},\sf{inv}_{n+1}(K(\mbf{k}^{\pm};\ell^{\pm};i)\cup L)\big) \\
        &\qquad\qquad\qquad\qquad\to\big(\#^{n+1}S^{1}\times S^{2},K(\mbf{k}^{\pm}+(\ell^{\pm}-k_{i}^{\pm})\mbf{e}_{i};k_{i}^{\pm};i)\cup L\big)
    \end{align*}
    is a diffeomorphism of pairs whose underlying diffeomorphism of $\#^{n+1}S^{1}\times S^{2}$ acts trivially on $H_{*}(\#^{n}S^{1}\times S^{2};\bb{Z})$.
\end{enumerate}
Then
\[\KhR_{2,\diamond}^{-}(\varphi_{\mbf{k}^{\pm},\ell^{\pm},i})=\KhR_{2,\diamond}^{-}(\psi_{\mbf{k}^{\pm},\ell^{\pm},i})\]
for $\diamond\in\{O,T\}$.
\end{proposition}

\begin{proof}
From the description of $\varphi_{i}$ in the proof of Lemma \ref{lem:diff_of_pairs_Psi_swap} and our assumptions on $\psi_{\mbf{k}^{\pm},\ell^{\pm},i}$, we see that $(\psi^{-1}_{\mbf{k}^{\pm},\ell^{\pm},i}\circ\varphi_{\mbf{k}^{\pm},\ell^{\pm},i})|_{\nu(K(\mbf{k}^{\pm};\ell^{\pm};i)\cup L)}=\id$. Furthermore, the maps $\sf{inv}_{n+1}\circ\varphi_{\mbf{k}^{\pm},\ell^{\pm},i}$ and $\psi'_{\mbf{k}^{\pm},\ell^{\pm},i}$ both act trivially on $H_{*}(\#^{n}S^{1}\times S^{2};\bb{Z})$, and hence $\psi^{-1}_{\mbf{k}^{\pm},\ell^{\pm},i}\circ\varphi_{\mbf{k}^{\pm},\ell^{\pm},i}$ acts trivially on $H_{*}(\#^{n}S^{1}\times S^{2};\bb{Z})$ as well. The result then follows from Lemma \ref{lem:diffeos_identity_on_homology}.
\end{proof}

We can therefore replace (\ref{eq:Z'_i_connected_sums_S_1xS_2}) with the equation
\begin{equation}
\label{eq:Z'_i_connected_sums_S_1xS_2_redux}
    \KhR_{2,\diamond}^{-}(Z'_{i},S'_{\mbf{k}^{\pm},\ell^{\pm},i})=\epsilon\circ\KhR_{2,\diamond}^{-}(\psi_{\mbf{k}^{\pm},\ell^{\pm},i}).
\end{equation}
where $\psi_{\mbf{k}^{\pm},\ell^{\pm},i}$ is as in Proposition \ref{prop:decompose_diffeomorphism}.

\begin{example}
\label{ex:psi_i_S^2xD^2}
Let $U=(U,0)\subset\del B^{4}$ be the $0$-framed unknot. We can take $\psi_{\mbf{k}^{\pm},\ell^{\pm},i}$ as in Proposition \ref{prop:decompose_diffeomorphism} to be given by $\sf{inv}$ followed by a planar isotopy as pictured in Figure \ref{fig:psi_i_S_2xD_2}.
\begin{center}
\begin{figure}
    \includegraphics[width=15cm]{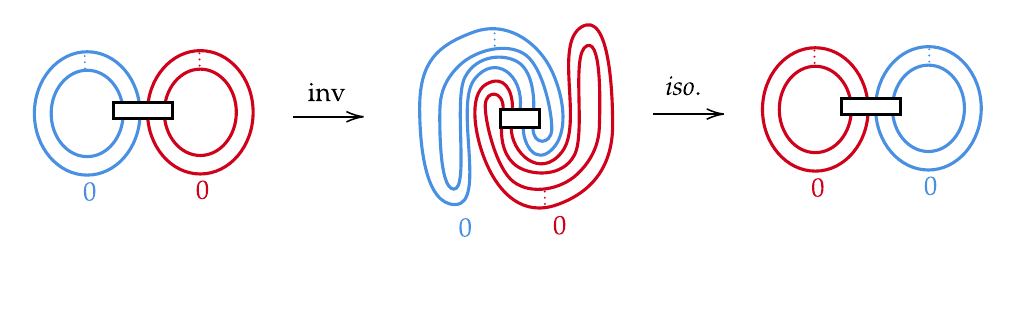}
    \caption{A depiction of $\psi_{k^{\pm},\ell^{\pm},1}$ for the $0$-framed unknot $U$.}
    \label{fig:psi_i_S_2xD_2}
\end{figure}
\end{center}
\end{example}

\begin{example}
\label{ex:psi_i_D(p)}
Let $(U,p)\subset\del B^{4}$ be the $p$-framed unknot for some $p\in\bb{Z}$. Then we can take $\psi_{\mbf{k}^{\pm},\ell^{\pm},i}$ as in Proposition \ref{prop:decompose_diffeomorphism} to be given by the sequence of moves pictured in Figure \ref{fig:psi_i_D(p)}.
\begin{center}
\begin{figure}
    \includegraphics[width=16cm]{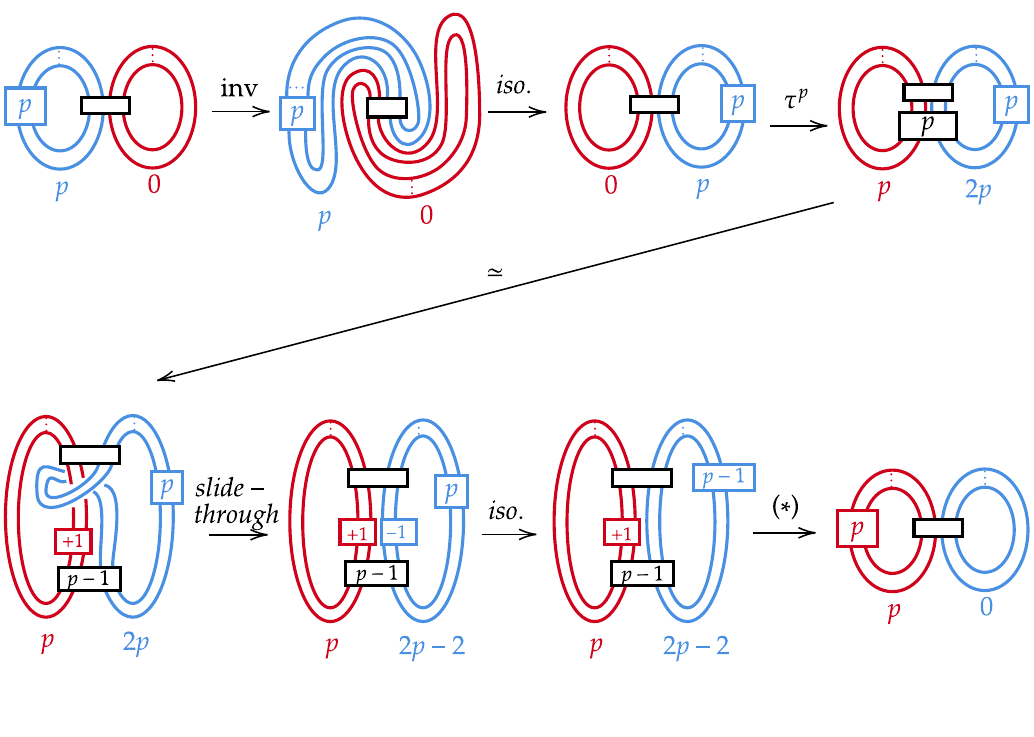}
    \caption{A depiction of $\psi_{k^{\pm},\ell^{\pm},1}$ for the $p$-framed unknot $(U,p)$. The map $(*)$ is given by repeating steps $4$--$6$ $|p-1|$ times. Depicted is the case where $p\geq 0$; the case $p<0$ can be obtained by changing all of the crossings in the bottom left picture, and changing the $+1$-twists to $-1$-twists and vice-versa in the first three pictures on the bottom.}
    \label{fig:psi_i_D(p)}
\end{figure}
\end{center}
\end{example}

We now return to the case of more general $X$, and calculate the map $\ol{\skein}_{0}^{2,\diamond}(Y\times[0,1],C_{\ell^{\pm}})$. Consider the identification
\[\ol{\skein}_{0}^{2,\diamond}\big(X;K(\mbf{k}^{+},\mbf{k}^{-})\cup U(\ell^{+},\ell^{-})\cup L\big)\cong\ol{\skein}_{0}^{2,\diamond}\big(X;K(\mbf{k}^{+},\mbf{k}^{-})\cup L\big)\otimes\KhR_{2}\big(U(\ell^{+},\ell^{-})\big)\]
induced by the identification of pairs
\[\big(X,K(\mbf{k}^{+},\mbf{k}^{-})\cup U(\ell^{+},\ell^{-})\cup L\big)\cong \big(X,K(\mbf{k}^{+},\mbf{k}^{-})\cup L)\big)\natural\big(B^{4},U(\ell^{+},\ell^{-})\big)\]
and the canonical isomorphism
\[\ol{\skein}_{0}^{2,\diamond}\big(B^{4};U(\ell^{+},\ell^{-})\big)\cong\KhR_{2}\big(U(\ell^{+},\ell^{-})\big).\]
Then we can write
\pagebreak
\begin{align*}
    \ol{\skein}_{0}^{2,\diamond}(Y\times[0,1],C_{\ell^{\pm}})=\text{id}\otimes\varepsilon_{\ell^{\pm}}:\ol{\skein}_{0}^{2,\diamond}\big(X,K(\mbf{k}^{+},\mbf{k}^{-})\cup L\big)\otimes&\KhR_{2}\big(U(\ell^{+},\ell^{-})\big) \\
    &\to\ol{\skein}_{0}^{2,\diamond}\big(X,K(\mbf{k}^{+},\mbf{k}^{-})\cup L\big)
\end{align*}
where
\begin{equation}
\label{eq:counit_epsilon}
    \varepsilon_{\ell^{\pm}}:\KhR_{2}\big(U(\ell^{+},\ell^{-})\big)\to\bb{Q}
\end{equation}
is the counit map associated to the cobordism which caps off all of the components with disks. More concretely, $\varepsilon_{\ell^{\pm}}$ factors as
\[\KhR_{2}\big(U(\ell^{+},\ell^{-})\big)\xrightarrow{\cong}(\KhR_{2}(U))^{\otimes(\ell^{+}+\ell^{-})}\xrightarrow{\varepsilon^{\otimes(\ell^{+}+\ell^{-})}}\bb{Q}\]
where $\varepsilon:\KhR_{2}(U)\to\bb{Q}$ is the (Frobenius) counit map which sends $\mbf{1}\mapsto 0$ and $\mbf{X}\mapsto 1$. Similarly we have that
\[\ol{\skein}_{0}^{2,\diamond}(Y\times[0,1],C_{k_{i}^{\pm}})=\text{id}\otimes\varepsilon_{k_{i}^{\pm}}.\]
Putting this all together, in the case where $X=\natural^{n}S^{1}\times B^{3}$ we can write
\begin{align}
\label{eq:Phi_Psi_decomps_1_hbody}
    &\Phi_{\diamond,\mbf{k}^{\pm},\ell^{\pm},i}=(\id\otimes\varepsilon_{\ell^{\pm}})\circ\epsilon, & &\Psi_{\diamond,\mbf{k}^{\pm},\ell^{\pm},i}=(\id\otimes\varepsilon_{k_{i}^{\pm}})\circ\epsilon\circ\KhR_{2,\diamond}^{-}(\psi_{\mbf{k}^{\pm},\ell^{\pm},i})
\end{align}
where $\psi_{\mbf{k}^{\pm},\ell^{\pm},i}$ is as in Proposition \ref{prop:decompose_diffeomorphism}.

It will be helpful to establish some further algebraic properties of the collections of maps $\{\Phi_{\diamond,\mbf{k}^{\pm},\ell^{\pm},i}\}$ and $\{\Psi_{\diamond,\mbf{k}^{\pm},\ell^{\pm},i}\}$. The following proposition posits a symmetry property between the $\Phi$ and $\Psi$ maps:

\begin{proposition}
\label{prop:symmetry_Phi_Psi}
Let $K\subset\#^{n} S^{1}\times S^{2}$, $\mbf{k}^{\pm}\in\bb{N}$, $\ell^{\pm}\in\bb{N}$, $i\in\{1,\dots, m\}$. Then 
\[\text{im}(\Phi_{\diamond,\mbf{k}^{\pm},\ell^{\pm},i},\Psi_{\diamond,\mbf{k}^{\pm},\ell^{\pm},i})=\text{im}(\Psi_{\diamond,\mbf{k}^{\pm}+(\ell^{\pm}-k_{i}^{\pm})\mbf{e}_{i},k_{i}^{\pm},i},\Phi_{\diamond,\mbf{k}^{\pm}+(\ell^{\pm}-k_{i}^{\pm})\mbf{e}_{i},k_{i}^{\pm},i})\]
as subgroups of
\[\KhR_{2,\diamond}^{-}\big(K(\mbf{k}^{+},\mbf{k}^{-})\big)^{\frak{S}_{\mbf{k}^{\pm}}}\oplus\KhR_{2,\diamond}^{-}\big(K(\mbf{k}^{+}+(\ell^{+}-k_{i}^{+})\mbf{e}_{i},\mbf{k}^{-}+(\ell^{-}-k_{i}^{-})\mbf{e}_{i})\big)^{\frak{S}_{\mbf{k}^{\pm}+(\ell^{\pm}-k_{i}^{\pm})\mbf{e}_{i}}}.\]
\end{proposition}

\begin{proof}
Follows immediately from Proposition \ref{prop:Psi_relate_to_Phi}.
\end{proof}

Next, we show that the lasso relation subsumes the annulus relations:

\begin{proposition}
\label{prop:lasso_subsumes_bundt_cake}
Let $v\in\ol{\skein}_{0}^{2,\diamond}(X;K(\mbf{k}^{+},\mbf{k}^{-})\cup L)$. Then for $\diamond\in\{O,T\}$ and every $i=1,\dots, m$ there exist elements $v',v''\in\ol{\skein}_{0}^{2,\diamond}(X\natural S^{1}\times B^{3};K(\mbf{k}^{\pm};k_{i}^{\pm}+1;i)\cup L)$ such that:
\begin{align*}
    &\Phi_{\diamond,\mbf{k}^{\pm},k_{i}^{\pm}+1,i}(v')=0, & &\Psi_{\diamond,\mbf{k}^{\pm},k_{i}^{\pm}+1,i}(v')=\psi_{\diamond,i}^{[0]}(v), \\
    &\Phi_{\diamond,\mbf{k}^{\pm},k_{i}^{\pm}+1,i}(v'')=v, & &\Psi_{\diamond,\mbf{k}^{\pm},k_{i}^{\pm}+1,i}(v'')=\psi_{\diamond,i}^{[1]}(v).
\end{align*}
\end{proposition}

\begin{proof}
Let $\diamond\in\{O,T\}$, and let $X''$ denote the 4-dimensional 1-handle cobordism from $Y$ to $Y\#S^{1}\times S^{2}$ which attaches a one handle along $\nu(w)\sqcup\nu(z)\subset\nu(K_{i})$, using the notation from the proof of Lemma \ref{lem:diff_of_pairs_Psi_swap}. We construct for each $u\in\KhR_{2,\diamond}(U)$ a (type $\diamond$) lasagna filling $F_{u}$ for the pair
\[\big(X'',-(K(\mbf{k}^{+},\mbf{k}^{-})\cup L)\sqcup(K(\mbf{k}^{\pm};k_{i}^{\pm}+1;i)\cup L)\big)\]
with precisely one input $B^{4}$. Fix a 4-ball $B\subset\text{int}(X'')$, and let $U\subset\del B$ denote the $0$-framed unknot. Consider the (framed, oriented) pair of pants surface $P\subset X''\setminus\mathring{B}$ with boundary
\[\del P=-K_{i}\sqcup(K(i)_{i}\cup K(i)_{m+1})\subset -Y\sqcup (Y\#S^{1}\times S^{2})\subset\del(X''\setminus\mathring{B})\]
as pictured in Figure \ref{fig:lasso_subsumes_bundt_1}. Let $P_{k_{i}^{\pm}}$ denote the surface obtained by taking $k_{i}^{+}$ positively-oriented and $k_{i}^{-}$ negatively-oriented disjoint push-offs of $P$. Similarly, let $\delta,\delta'\subset Y\#S^{1}\times S^{2}$ be two push-offs of $K(i)_{m+1})$, and let $P'\subset X''\setminus\mathring{B}$ be the pair of pants surface with boundary
\[\del P'=U\sqcup(-\delta\cup \delta')\subset S^{3}\sqcup (Y\#S^{1}\times S^{2})\subset\del(X''\setminus\mathring{B}),\]
constructed as follows: let $A\subset Y\#S^{1}\times S^{2}$ denote the natural annulus co-bounded by $\delta,\delta'$. Push the interior of $A$ into $\text{int}(X''\setminus B)$, remove a disk from its interior and attach a tube in $X''$ connecting it to $U\subset\del B$.

After possibly perturbing $P'$ we can assume that $P_{k_{i}^{\pm}}\cap P'=\emptyset$. Finally, let $S\subset X''\setminus\mathring{B}$ be the surface obtained as the union of the traces of the remaining components of $K(\mbf{k}^{+},\mbf{k}^{-})\cup L$. The desired lasagna filling is then given by
\[F_{u}=(P_{k_{i}^{\pm}}\cup P'\cup S,(B,U,u)),\qquad u\in\KhR_{2,\diamond}(U).\]
Using $F_{u}$, we can construct a map
\[f_{u}:\ol{\skein}_{0}^{2,\diamond}\big(X;K(\mbf{k}^{+},\mbf{k}^{-})\cup L\big)\to\ol{\skein}_{0}^{2,\diamond}\big(X\natural S^{1}\times B^{3};K(\mbf{k}^{\pm};k_{i}^{\pm}+1;i)\cup L\big)\]
via $f_{u}:=\sf{gl}(-\otimes[F_{u}])$ where $\sf{gl}$ denotes the gluing map from Proposition \ref{prop:gluing_surjection}. Given $v\in\ol{\skein}_{0}^{2,\diamond}(X;K(\mbf{k}^{+},\mbf{k}^{-})\cup L)$, the elements $v',v''$ as in the statement of the proposition are given by $f_{\mbf{1}}(v)$, $f_{\mbf{X}}(v)$, respectively. Observe that
\begin{align*}
    &(\Phi_{\diamond,\mbf{k}^{\pm},k_{i}^{\pm}+1,i}\circ f_{\mbf{1}})(v)=0 & &(\Phi_{\diamond,\mbf{k}^{\pm},k_{i}^{\pm}+1,i}\circ f_{\mbf{X}})(v)=v
\end{align*}
for any $v\in\ol{\skein}_{0}^{2,\diamond}(X;K(\mbf{k}^{+},\mbf{k}^{-})\cup L)$,
since any lasagna filling $F$ of $(X;K(\mbf{k}^{+},\mbf{k}^{-})\cup L)$ is sent under $(\Phi_{\diamond,\mbf{k}^{\pm},k_{i}^{\pm}+1,i}\circ f_{\mbf{1}})$ and $(\Phi_{\diamond,\mbf{k}^{\pm},k_{i}^{\pm}+1,i}\circ f_{\mbf{X}})$ to the union of $F$ with an un-dotted and dotted sphere, respectively (see Figure \ref{fig:lasso_subsumes_bundt_2}). On the other hand, by construction we see that
\begin{align*}
    &(\Psi_{\diamond,\mbf{k}^{\pm},k_{i}^{\pm}+1,i}\circ f_{\mbf{1}})(v)=\psi^{[0]}_{\diamond,i}(v) & &(\Psi_{\diamond,\mbf{k}^{\pm},k_{i}^{\pm}+1,i}\circ f_{\mbf{X}})(v)=\psi^{[1]}_{\diamond,i}(v)
\end{align*}
as desired.
\begin{center}
\begin{figure}
    \includegraphics[width=16cm]{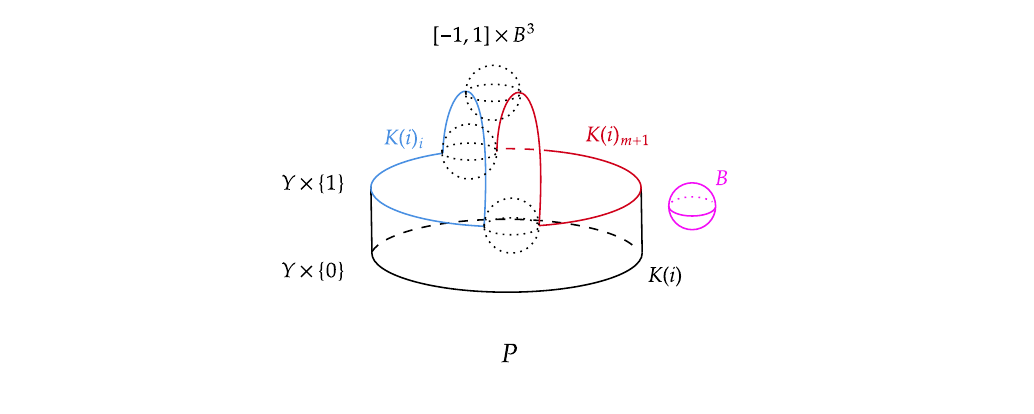}
    \caption{The surface $P$ used in the proof of Proposition \ref{prop:lasso_subsumes_bundt_cake}.}
    \label{fig:lasso_subsumes_bundt_1}
\end{figure}
\begin{figure}
    \includegraphics[width=11cm]{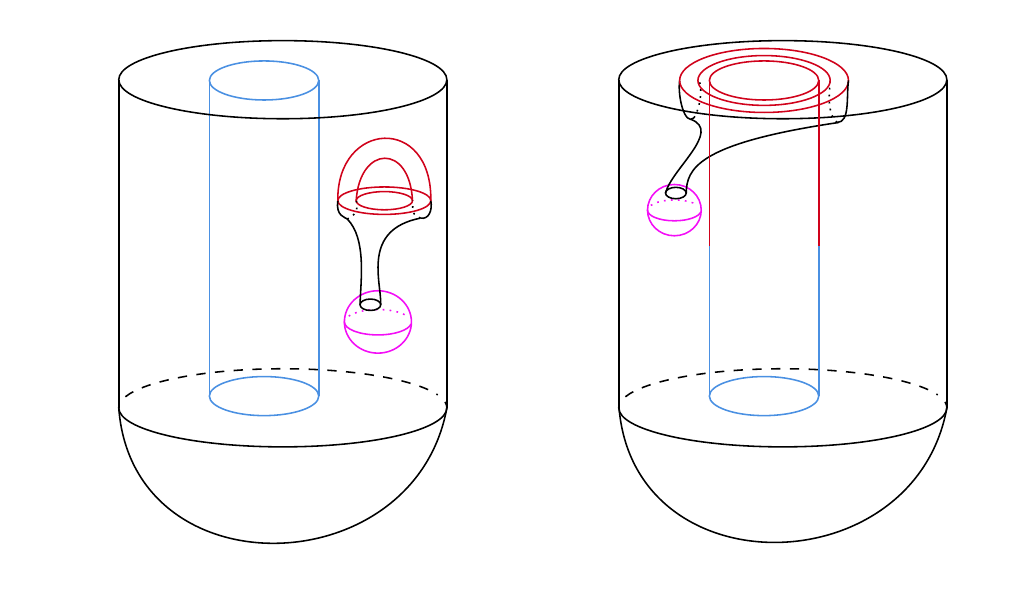}
    \caption{The cobordisms $(W_{i},\Sigma_{\mbf{k}^{\pm},k_{i}^{\pm}+1,i})$ and $(W'_{i},\Sigma'_{\mbf{k}^{\pm},k_{i}^{\pm}+1,i})$, respectively, after gluing in the lasagna filling $F_{u}$ as defined in the proof of Proposition \ref{prop:lasso_subsumes_bundt_cake}.}
    \label{fig:lasso_subsumes_bundt_2}
\end{figure}
\end{center}
\end{proof}

Next, we specialize to the $\diamond=T$ theory. In this setting, the lasso relation includes two new sets of relations, which we call the \emph{disoriented} annulus relations. Let $i\in\{1,\dots,m\}$, let $\mbf{k}^{\pm}\in\bb{N}^{m}$, and let $B$ be a 4-dimensional tubular neighborhood of a push off of $K_{i}\times\{\frac{1}{2}\}\subset Y\times[0,1]$. Let $\belt{2,0}\subset\del B\cong S^{1}\times S^{2}$ denote the belt link with two positively oriented strands, and consider the surface $\Sigma_{+}$ in $Y\times[0,1]\setminus B$ with boundary
\[\del\Sigma_{+}=\big(-(K(\mbf{k}^{+},\mbf{k}^{-})\cup L)\times\{0\}\big)\sqcup\big(-\belt{2,0}\big)\sqcup\big((K(\mbf{k}^{+}+2\mbf{e}_{i},\mbf{k}^{-})\cup L)\times\{1\}\big)\]
which consists of:
\begin{itemize}
    \item trivial cylinders connecting the components of $K(\mbf{k}^{+},\mbf{k}^{-})\cup L\subset Y\times\{0\}$ to their corresponding components in the sublink $K(\mbf{k}^{+},\mbf{k}^{-})\cup L\subset K(\mbf{k}^{+}+2\mbf{e}_{i},\mbf{k}^{-})\cup L\subset Y\times\{1\}$.
    \item two cylinders parallel to $K_{i}\times[\frac{1}{2},1]$ connecting $\belt{2,0}\subset\del B$ to the remaining two positively-oriented components of $K(\mbf{k}^{+}+2\mbf{e}_{i},\mbf{k}^{-})\cup L\subset Y\times\{1\}$.
\end{itemize}
For each $u\in\KhR_{2,T}^{-}(\belt{2,0})$ consider the type $T$ 1-dimensional inputs lasagna filling
\[F^{+}_{u}:=(\Sigma_{+},(B,\belt{2,0},v))\]
of the pair
\[\big(Y\times[0,1],-(K(\mbf{k}^{+},\mbf{k}^{-})\cup L)\sqcup(K(\mbf{k}^{+}+2\mbf{e}_{i},\mbf{k}^{-})\cup L)\big).\]
Similar to the proof of Proposition \ref{prop:lasso_subsumes_bundt_cake}, we can glue this lasagna filling to the boundary of a lasagna filling of $(X,K(\mbf{k}^{+},\mbf{k}^{-})\cup L)$ to obtain a lasagna filling of $(X,K(\mbf{k}^{+}+2\mbf{e}_{i},\mbf{k}^{-})\cup L)$, defining a map
\[\psi^{+,u}_{T,i}:\ol{\skein}_{0}^{2,T}(X;K(\mbf{k}^{+},\mbf{k}^{-})\cup L)\to\ol{\skein}_{0}^{2,T}(X;K(\mbf{k}^{+}+2\mbf{e}_{i},\mbf{k}^{-})\cup L)\]
for each $u\in\KhR_{2,T}^{-}(\belt{2,0})$ with grading shift equal to the bigrading of $u$. By replacing $\belt{2,0}$ with $\belt{0,2}$ in the above discussion, we can similarly define a map
\[\psi^{-,u}_{T,i}:\ol{\skein}_{0}^{2,T}(X;K(\mbf{k}^{+},\mbf{k}^{-})\cup L)\to\ol{\skein}_{0}^{2,T}(X;K(\mbf{k}^{+},\mbf{k}^{-}+2\mbf{e}_{i})\cup L)\]
for each $u\in\KhR_{2,T}^{-}(\belt{0,2})$, again with grading shift equal to the bigrading of $u$.

Next, consider the maps $\varepsilon_{+}$ and $\varepsilon_{-}$ obtained as the compositions of the maps
\begin{align*}
    &\KhR_{2,T}^{-}\big(\belt{2,0}\big)\xrightarrow{\epsilon}\KhR_{2}\big(U(2,0)\big)\xrightarrow{\varepsilon^{\otimes 2}}\bb{Q}, \\
    &\KhR_{2,T}^{-}\big(\belt{0,2}\big)\xrightarrow{\epsilon}\KhR_{2}\big(U(0,2)\big)\xrightarrow{\varepsilon^{\otimes 2}}\bb{Q},
\end{align*}
respectively. The proof of the following proposition is very similar to that of Proposition \ref{prop:lasso_subsumes_bundt_cake}, so we leave it as an exercise to the motivated reader:

\begin{proposition}
\label{prop:disoriented_bundt_cake}
Let $v\in\ol{\skein}_{0}^{2,T}(X;K(\mbf{k}^{+},\mbf{k}^{-})\cup L)$. Then for every $i=1,\dots, m$, the following statements are true:
\begin{enumerate}
    \item For each $u_{+}\in\KhR_{2,T}^{-}(\belt{2,0})$ there exists an element $w_{u_{+}}\in\ol{\skein}_{0}^{2,\diamond}(X\natural S^{1}\times B^{3};K(\mbf{k}^{\pm};k_{i}^{\pm}+1\pm 1;i)\cup L)$ such that:
    \begin{align*}
        &\Phi_{T,\mbf{k}^{\pm},k_{i}^{\pm}+1\pm 1,i}(w_{u_{+}})=\varepsilon_{+}(u_{+})\cdot v, & &\Psi_{T,\mbf{k}^{\pm},k_{i}^{\pm}+1\pm 1,i}(w_{u_{+}})=\psi_{T,i}^{+,u_{+}}(v).
    \end{align*}
    \item For each $u_{-}\in\KhR_{2,T}^{-}(\belt{0,2})$ there exists an element $w_{u_{-}}\in\ol{\skein}_{0}^{2,\diamond}(X\natural S^{1}\times B^{3};K(\mbf{k}^{\pm};k_{i}^{\pm}+1\mp 1;i)\cup L)$ such that:
    \begin{align*}
        &\Phi_{T,\mbf{k}^{\pm},k_{i}^{\pm}+1\mp 1,i}(w_{u_{-}})=\varepsilon_{-}(u_{-})\cdot v, & &\Psi_{T,\mbf{k}^{\pm},k_{i}^{\pm}+1\mp 1,i}(w_{u_{-}})=\psi_{T,i}^{-,u_{-}}(v).
    \end{align*}
\end{enumerate}
\end{proposition}

One can show (e.g., by comparing with the corresponding calculation in Section 4.1 in \cite{WillisS1xS2}) that the maps $\varepsilon_{\pm}$ vanish in bidegrees $\neq (0,2)$, and in bigrading $(0,2)$ the maps $\varepsilon_{\pm}$ induce isomorphisms
\begin{align*}
    &\KhR_{2,T}^{-,0,2}\big(\belt{2,0}\big)\underset{\cong}{\xrightarrow{\varepsilon_{+}}}\bb{Q}, & &\KhR_{2,T}^{-,0,2}\big(\belt{0,2}\big)\underset{\cong}{\xrightarrow{\varepsilon_{-}}}\bb{Q}.
\end{align*}
Define $\mbf{X}_{+}\in\KhR_{2,T}^{-,0,2}(\belt{2,0})$ and $\mbf{X}_{-}\in\KhR_{2,T}^{-,0,2}(\belt{0,2})$, respectively, as the preimages of $1\in\bb{Q}$ under the above isomorphisms. Proposition \ref{prop:disoriented_bundt_cake} then implies the relations
\begin{align}
\label{eq:disoriented_annulus_relations}
    &\psi_{T,i}^{+,u_{+}}(v)\sim\psi_{T,i}^{-,u_{-}}(v)\sim 0, & &\psi_{T,i}^{+,\mbf{X}_{+}}(v)\sim\psi_{T,i}^{-,\mbf{X}_{-}}(v)\sim v,
\end{align}
in the coequalizer from Theorem \ref{theorem:tortellini_2_handle_formula} for all $u_{+}\in\KhR_{2,T}^{-,h,q}(\belt{2,0})$ and $u_{-}\in\KhR_{2,T}^{-,h,q}(\belt{0,2})$ such that $(h,q)\neq(0,2)$.

\section{1-, 3-, \& 4-Handles}
\label{sec:13and4handles}

In this section we provide formulae for handle attachments of index $\neq{2}$. Let $X$ admit a handle decomposition $X_{1}\subset{X_{2}}\subset{X_{3}}\subset{X_{4}}=X$ where $X_{i+1}$ is obtained from $X_{i}$ by attaching the index $(i+1)$ handles of $X$. Let $L$ denote a link in the boundary $\partial{X_{1}}$, and let $K$ denote the framed attaching link of the 2-handles of $X_{2}$. For the pair $(X_{1},L)$, and in analogy with the 0-dimensional inputs invariant evaluated on the pair $(B^{4},L)$, the 1-dimensional inputs skein lasagna module recovers the Rozansky-Willis homology of $L\subset{\sqcup^{m}\#^{n}S^{1}\times{S^{2}}}$ on pairs $(X_{1},L)$ where $\partial{X_{1}}=\sqcup^{m}\#^{n}S^{1}\times{S^{2}}$.
See (\cite{RSWWZ25}, discussion before Remark 3.5): As previously stated there is a natural isomorphism
\begin{equation}
\label{eq:tortellini_1_h-body}
    \ol{\skein}_{0}^{2,\diamond}(\textstyle{\natural^{n}}S^{1}\times B^{3};L)\cong\KhR_{2,\diamond}^{-}(\#^{n}S^{1}\times{S^{2}};L)
\end{equation}
given by $[(L\times{I},B,v)]\mapsto{v}$, where $(L\times{I},B,v)$ is the 1-dimensional inputs filling with input manifold $B$ diffeomorphic to a smaller copy of $X_{1}$ embedded in the original, and $L\times{I}$ the product cobordism in the relative handlebody $X_{1}\setminus{\text{int}(B)}$.

\begin{remark}
    In the 0-dimensional inputs setting, the inclusion map $(X,L)\hookrightarrow{(X\cup{h_{1}},L)}$, where $h_{1}$ is a 1-handle attached such that $L$ is disjoint from the cocore, induces an isomorphism (Lemma 4.1 of \cite{MWWhandles}). We remark that the analogous statement for the 1-dimensional inputs skein lasagna module is not true on the nose, as an arbitrary 1-dimensional inputs lasagna filling in $X\cup{h_{1}}$ may now have an input manifold that intersect the cocore of $h_{1}$ non-trivially, and therefore it may not be possible to perform an isotopy so that the resulting filling contained entirely in $X$, disjoint from the 1-handle.
\end{remark}

\begin{remark} (3-handle attachments)
When relating $\overline{\skein}_{0}^{2,\diamond}(X_{2})$ and $\overline{\skein}_{0}^{2,\diamond}(X_{3})$, possibly with additional boundary link data, we obtain a formula nearly identical in spirit to that of \cite[Theorem 3.7]{MWWhandles} as the argument does not depend on the dimension of the input manifold. For completeness we sketch the argument here.

Given $X_{2}$ as above with a framed link $L$ in the boundary, let $S\subset{\partial{X_{2}}}$ be an embedded 2-dimensional sphere that does not intersect $L$. Write $X^{\prime}=X_{2}\cup{Z}$ where $Z$ is the cobordism corresponding to attaching a 3-handle along $S$, then there is a product cobordism $I\times{L}$ in $Z$ mapping $L$ to the link $L^{\prime}$ in the outgoing boundary of $Z$. Letting $\gamma$ denote the oriented equator of the sphere $S$, Manolescu-Walker-Wedrich define cobordism maps

\[\Delta^{\pm}:(\partial{X_{2}},\gamma\cup{L})\rightarrow{(\partial{X_{2},L^{\prime})}} 
\]

obtained by first obtaining cobordisms corresponding to pushing the ``top'' and ``bottom'' hemispheres of $S$ into the collar neighborhood $I\times{\partial{X_{2}}}$ and taking their union with $I\times{L}$ in the same collar. The orientations of $\Delta^{+}$ and $\Delta^{-}$ are chosen so that they induce the chosen orientation on the boundary $\gamma$. We then consider the corresponding lasagna gluing maps for pairs $(I\times{\partial{X_{2}}},\Delta^{\pm})$, denoted $\Psi_{I\times{\partial{X_{2}}},\Delta^{\pm}}$ and obtain the following.
\end{remark}

\begin{prop}\label{prop:3-handles}
Let $(X_{3},L')$ be obtained from $(X_{2},L)$ via a 3-handle attachment along an embedded 2-sphere $S\subset\del X_{2}\setminus L$ as above. The 1-dimensional inputs skein lasagna module of $(X_{3},L')$ at level $\bm{\alpha}'\in H_{2}^{L'}(X_{3};R)$ is isomorphic to the coequalizer
\[\ol{\cal{S}}_{0}^{2,\diamond}(X_{3};L';\bm{\alpha}')\cong\text{Coeq}\Big(\bigoplus_{\substack{\bm{\alpha}\in H_{2}^{L}(X_{2};R) \\ \pi(\bm{\alpha})=\bm{\alpha}'}}\ol{\skein}^{2,\diamond}_{0}(X_{2},L\cup \gamma;\bm{\alpha})\doublearrow{\Psi_{I\times{\partial{X_{2}}},\Delta^{+}}}{\Psi_{I\times{\partial{X_{2}}},\Delta^{-}}}\bigoplus_{\substack{\bm{\alpha}\in H_{2}^{L}(X_{2};R) \\ \pi(\bm{\alpha})=\bm{\alpha}'}}\ol{\skein}^{2,\diamond}_{0}(X_{2},L;\bm{\alpha})\Big),\]
where 
\[\pi:H_{2}^{L}(X;R)\to H_{2}^{L'}(X';R)\cong H_{2}^{L}(X;R)/\<[S]\>\]
denotes the canonical projection map, and the coefficient ring $R$ is taken to be $\bb{Z}$ if $\diamond=O$ and $\bb{Z}/2$ if $\diamond=T$. 
\end{prop}

\begin{proof}
    The proof follows from an argument identical to the argument presented in the proof of Theorem 3.7 in \cite{MWWhandles}. To apply the argument, we require the following topological statements. Given a 1-dimensional inputs lasagna filling $F^{\prime}$ for the pair $(X_{3},L^{\prime})$, the 1-dimensional input manifold of $F^{\prime}$ may be isotoped away from the cocore of the attached 3-handle and be made to lie in $\text{int}(X_{2})$, hence, the filling $F^{\prime}$ may be isotoped away from the 1-dimensional cocore of the attached 3-handle. Next, we remark that an analogue of \cite[Lemma 2.1]{MWWhandles} holds for 1-dimensional input fillings, with the caveat that there are many choices for the homotopy type of a fixed input manifold for each connected component of $X_{2}$. Despite this, the same argument presented in \cite{MWWhandles} follows.
\end{proof}

We are now ready to prove Theorem \ref{theorem:main}:

\begin{proof}[Proof of Theorem \ref{theorem:main}]
The attachment of a $4$-handle does not affect $\overline{\skein}_{0}^{2,\diamond}$ by the same reasoning as for the $0$-dimensional inputs skein lasagna modules. The theorem then follows from the results of this section along with the main result of Section \ref{sec:2-handles}.
\end{proof}

\section{Calculations}
\label{sec:calculations}

In this section we use Corollary \ref{cor:tortellini_1_2_handlebody} to compute $\ol{\skein}^{2,\diamond}_{0}(X;L)$ for various $4$-manifolds $X$ built out of $1$-and $2$-handles.

\subsection{\texorpdfstring{$\mathbf{S^{2}\times D^{2}}$}{S2xD2}}
\label{subsec:S^2xD^2}

Recall that $S^{2}\times D^{2}$ can be constructed by attaching a 2-handle along the zero-framed unknot $U=(U,0)\subset\del B^{4}$. Fix $\alpha\in\bb{Z}$, and let $r,s\in\bb{N}$. Note that $U(1)$\footnote{not to be confused with the other $U(1)$.} can be identified with the belt link $\belt{1,1}\subset S^{1}\times S^{2}$, and hence we have the following identification:
\[U(\alpha^{\pm}+r;\alpha^{\pm}+s;1)=\belt{\alpha^{+}+r,\alpha^{+}+r}\cup\belt{\alpha^{-}+s,\alpha^{+}+s}\subset S^{1}\times S^{2}.\]
From (\ref{eq:Phi_Psi_decomps_1_hbody}) we have that
\begin{align*}
    &\Phi_{\diamond,\alpha^{\pm}+r,\alpha^{\pm}+s,1}=(\text{id}\otimes\varepsilon_{\alpha^{\pm}+s})\circ\epsilon, &
    &\Psi_{\diamond,\alpha^{\pm}+r,\alpha^{\pm}+s,1}=(\id\otimes\varepsilon_{\alpha^{\pm}+r})\circ\epsilon\circ\KhR_{2,\diamond}^{-}(\psi_{\alpha^{\pm}+r,\alpha^{\pm}+s,1}).
\end{align*}
where $\psi_{\alpha^{\pm}+r,\alpha^{\pm}+s,1}$ is as in Example \ref{ex:psi_i_S^2xD^2}. Observe that we can rewrite the map $\Psi_{\diamond,\alpha^{\pm}+r,\alpha^{\pm}+s,1}$ as
\[\Psi_{\diamond,\alpha^{\pm}+r,\alpha^{\pm}+s,1}=(\varepsilon_{\alpha^{\pm}+r}\otimes\id)\circ(\id\otimes g)\circ\epsilon.\]
where
\[g:\KhR_{2}\big(U(\alpha^{-}+s,\alpha^{+}+s)\big)\{-|\alpha|-2s\}\xrightarrow{\cong}\KhR_{2}\big(U(\alpha^{+}+s,\alpha^{-}+s)\big)\{-|\alpha|-2s\}\]
is the canonical isomorphism given by
\begin{align*}
    &\cal{V}^{\otimes(\alpha^{-}+s)}\otimes\cal{V}^{\otimes(\alpha^{+}+s)}\xrightarrow{\cong}\cal{V}^{\otimes(\alpha^{+}+s)}\otimes\cal{V}^{\otimes(\alpha^{-}+s)}, & &\cal{V}:=\KhR_{2}(U)\{-1\}.
\end{align*}
From Corollary \ref{cor:tortellini_1_2_handlebody} and Proposition \ref{prop:lasso_subsumes_bundt_cake} it follows that
\[\ol{\skein}^{2,O}_{0}(S^{2}\times D^{2};\alpha)\cong\Big(\bigoplus_{r=0}^{\infty}\KhR_{2}\big(U(\alpha^{+}+r,\alpha^{-}+r)\big)^{\frak{S}_{\alpha^{\pm}+r}}\{-|\alpha|-2r\}\Big)/\sim\]
where $\sim$ is the transitive and linear closure of the relations
\[\big[(\id\otimes\varepsilon_{\alpha^{\pm}+s})(v)\big]\sim\big[(\varepsilon_{\alpha^{\pm}+r}\otimes\id)(v)\big]\]
for all $r,s\in\bb{N}$ and all $v\in\text{im}(F_{\alpha^{\pm}+r,\alpha^{\pm}+s})$, where
\[F_{\alpha^{\pm}+r,\alpha^{\pm}+s}:=(\text{id}\otimes g)\circ\epsilon.\]
Finally, note that $\ol{\cal{S}}^{2,\diamond}_{0}(S^{2}\times D^{2})$ can be given the structure of a $\bb{Q}$-algebra via the gluing
\[S^{2}\times D^{2}\cup_{S^{2}\times I}S^{2}\times D^{2}\approx S^{2}\times D^{2}\]
as in the case of the ordinary skein lasagna module (see \cite{MN22}, Section 5).

\begin{proposition}
\label{prop:tortellini_S^2xD^2}
We have algebra isomorphisms
\begin{align*}
    &\ol{\cal{S}}^{2,O}_{0}(S^{2}\times D^{2})\cong\bb{Q}[A_{0},A_{0}^{-1}] & &\ol{\cal{S}}^{2,T}_{0}(S^{2}\times D^{2})\cong\bb{Q}[A_{0}]/(A_{0}^{2}-1)
\end{align*}
where $A_{0}^{\pm 1}$ is concentrated in tri-degree $(0,0,\pm 1)$. In particular,
\begin{align*}
    &\ol{\cal{S}}^{2,O}_{0}(S^{2}\times D^{2};\alpha)\cong\bb{Q} & &\ol{\cal{S}}^{2,T}_{0}(S^{2}\times D^{2};\ol{\alpha})\cong\bb{Q}
\end{align*}
for each $\alpha\in\bb{Z}$ (respectively, $\ol{\alpha}\in\bb{Z}/2$), concentrated in bi-degree $(h,q)=(0,0)$.
\end{proposition}

\begin{proof}[Proof of Proposition \ref{prop:tortellini_S^2xD^2}]
Recall from (\cite{MN22}, Theorem 1.2) that we have a $\bb{Q}$-algebra isomorphism
\[\skein^{2}_{0}(S^{2}\times D^{2})\cong\bb{Q}[A_{0},A_{0}^{-1},A_{1}]\]
where
\begin{align*}
    &A_{0}=\big[\big(\mbf{X}\big)\otimes\big(1\big)\big], & &A_{0}^{-1}=\big[\big(1\big)\otimes\big(\mbf{X}\big)\big], & &A_{1}=\big[\big(\mbf{1}\big)\otimes\big(1\big)\big],
\end{align*}
with $[\cdot]$ denoting their corresponding equivalence classes in $\ul{\KhR}_{2}(U)$. We claim that the surjections $f_{0,\diamond}:\skein^{2}_{0}(S^{2}\times D^{2})\twoheadrightarrow\ol{\skein}^{2,\diamond}_{0}(S^{2}\times D^{2})$, $\diamond\in\{O,T\}$ from Proposition \ref{prop:surjections_lasagna_simply_connected} can be identified with the $\bb{Q}$-algebra homomorphisms
\begin{align*}
    &\bb{Q}[A_{0},A_{0}^{-1},A_{1}]\twoheadrightarrow\bb{Q}[A_{0},A_{0}^{-1},A_{1}]/(A_{1})\cong\bb{Q}[A_{0},A_{0}^{-1}], & &\diamond=O, \\
    &\bb{Q}[A_{0},A_{0}^{-1},A_{1}]\twoheadrightarrow\bb{Q}[A_{0},A_{0}^{-1},A_{1}]/(A_{0}^{2}-1,A_{1})\cong\bb{Q}[A_{0}]/(A_{0}^{2}-1), & &\diamond=T.
\end{align*}
We first treat the case $\diamond=O$. Now for each $k^{+},k^{-}\in\bb{N}$, let
\[h_{k^{\pm}}:\KhR_{2}\big(U(k^{+},k^{-})\big)\{-k^{+}-k^{-}\}\to\ul{\KhR}_{2}(U)\]
denote the canonical map, and let $\varepsilon_{k^{\pm}}$ be the Frobenius counit map as in (\ref{eq:counit_epsilon}). By inspection, we see that:
\begin{itemize}
    \item The sub-algebra $\bb{Q}[A_{0},A_{0}^{-1}]\subset\bb{Q}[A_{0},A_{0}^{-1},A_{1}]$ is generated as a $\bb{Q}$-vector space by all elements of the form $h_{k^{\pm}}(v)$ where $\varepsilon_{k^{\pm}}(v)=1$.
    \item The ideal $(A_{1})\subset\bb{Q}[A_{0},A_{0}^{-1},A_{1}]$ is generated as a $\bb{Q}$-vector space by all elements of the form $h_{k^{\pm}}(v)$ where $\varepsilon_{k^{\pm}}(v)=0$.
\end{itemize}
To prove the theorem, it suffices to show that the following hold for every $\alpha\in\bb{Z}$:
\begin{enumerate}
    \item For every $r\in\bb{N}$ and $x\in\ker(\varepsilon_{\alpha^{\pm}+r})$, there exists some $s\in\bb{N}$ and some element $v\in\im(F_{\alpha^{\pm}+r,\alpha^{\pm}+s})$ such that:
    \begin{enumerate}
        \item $(\id\otimes\varepsilon_{\alpha^{\pm}+s})(v)\sim Cx$ under the symmetry and annulus relations for some $C\in\bb{Q}$.
        \item $(\varepsilon_{\alpha^{\pm}+r}\otimes\id)(v)=0$.
    \end{enumerate}
    \item For any $r,s\in\bb{N}$ and any element $w\in\im(F_{\alpha^{\pm}+r,\alpha^{\pm}+s})$ the following implication holds:
    \begin{equation}
    \begin{split}
    \label{eq:w_S^2xD^2}
        &(\id\otimes\varepsilon_{\alpha^{\pm}+s})(w)=\big(\mbf{X}^{\otimes(\alpha^{+}+r)}\big)\otimes\big(\mbf{X}^{\otimes(\alpha^{-}+r)}\big) \\
        &\qquad\qquad\qquad\qquad\implies(\varepsilon_{\alpha^{\pm}+r}\otimes\id)(w)=\big(\mbf{X}^{\otimes(\alpha^{+}+s)}\big)\otimes\big(\mbf{X}^{\otimes(\alpha^{-}+s)}\big)+x
    \end{split}
    \end{equation}
    for some $x\in\ker(\varepsilon_{\alpha^{\pm}+s})$.
\end{enumerate}
To show that (1) holds, by linearity it suffices to prove the result in the case where 
\[x=x_{r,k^{\pm}}:=(\mbf{1}^{\otimes k^{+}}\otimes\mbf{X}^{\otimes(\alpha^{+}+r-k^{+})})\otimes(\mbf{1}^{\otimes k^{-}}\otimes\mbf{X}^{\otimes(\alpha^{-}+r-k^{-})})\in\ker(\varepsilon_{\alpha^{\pm}+r})\]
for some $0\le k^{\pm}\le\alpha^{\pm}+r$ such that at least one of $k^{+}$ and $k^{-}$ is non-zero. Note that for a fixed pair $k^{\pm}\in\bb{N}$ we have that $x_{r,k^{\pm}}\sim x_{r',k^{\pm}}$ under the annulus relations for all $r,r'\in\bb{N}$ such that $r,r'\geq k^{\pm}-\alpha^{\pm}$. So without loss of generality it suffices to show the corresponding result under the extra hypothesis that $r\geq 2k^{\pm}-\alpha^{\pm}$.

Recall from Definition \ref{def:Roz-cobar} that the image of the counit map
\[\epsilon:\KhR_{2,O}^{-}\big(\belt{n,n}\big)\to\KhR_{2}\big(U(n,n)\big)\]
for $n\in\bb{N}$ can be identified with the image of the sum of saddle maps $\bigoplus_{\delta\in\cal{B}_{2n}}\nu_{\delta}$, where $\nu_{\delta}$ denotes the map
\[\KhR_{2}\big(\Tr(\delta\otimes \delta^{t})\big)\{n\}\to\KhR_{2}\big(U(n,n)\big)\]
induced by the canonical saddle cobordism $q^{n}(\delta\otimes\delta^{t})\to \bbm{1}_{2n}$.

Consider the matching $\delta_{r,k^{\pm}}$ of $2|\alpha|+4r$ points as pictured in Figure \ref{fig:matching}, and identify $\Tr(\delta_{r,k^{\pm}}\otimes\delta_{r,k^{\pm}}^{t})$ with the $|\alpha|+2r$-component unlink $U_{|\alpha|+2r}$ by letting the $j$-th component be the component formed by the matched pairs given as follows:
\begin{itemize}
    \item $\{j,2k^{+}-j+1\}$ for $1\le j\le k^{+}$,
    \item $\{j+k^{+},2|\alpha|+4r-j-k^{+}+1\}$ for $k^{+}+1\le j\le \alpha^{+}+r-k^{+}$,
    \item $\{j+k^{+},2\alpha^{+}+2r-k^{+}+2k^{-}-j+1\}$ for $\alpha^{+}+r-k^{+}+1\le j\le\alpha^{+}+r-k^{+}+k^{-}$,
    \item $\{j+k^{+}+k^{-},2|\alpha|+4r-k^{+}-k^{-}-j+1\}$ for $j=\alpha^{+}+r-k^{+}+k^{-}+1\le j\le|\alpha|+2r-k^{+}-k^{-}$,
    \item $\{j+\alpha^{-}+r+k^{+}-k^{-},2|\alpha|+\alpha^{-}+5r-k^{+}-k^{-}-j+1\}$ for $|\alpha|+2r-k^{+}-k^{-}+1\le j\le|\alpha|+2r-k^{+}$,
    \item $\{j+|\alpha|+2r-k^{+}+k^{-},3|\alpha|+6r-k^{+}-j+1\}$ for $|\alpha|+2r-k^{+}+1\le j\le|\alpha|+2r$.
\end{itemize}
\begin{center}
\begin{figure}
    \includegraphics[width=18cm]{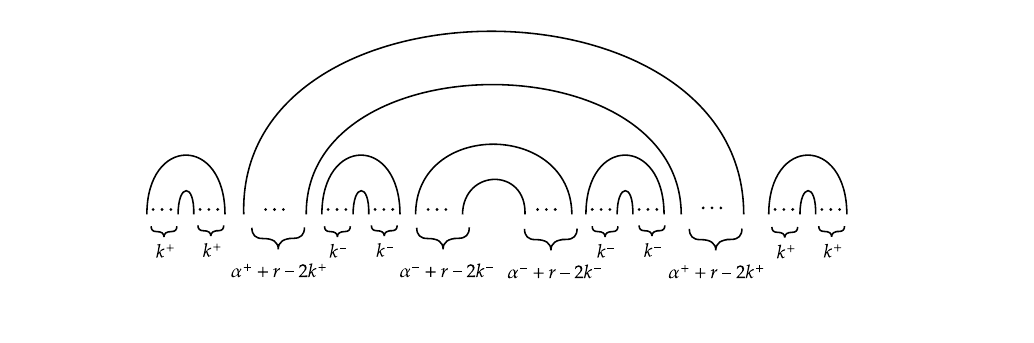}
    \caption{The matching $\delta_{r,k^{\pm}}$}
    \label{fig:matching}
\end{figure}
\end{center}
Let
\begin{align*}
    v:=&(g\circ\nu_{\delta_{k,r^{\pm}}})\big(\mbf{1}^{\otimes k^{+}}\otimes\mbf{X}^{\otimes(\alpha^{+}+r-2k^{+})}\otimes\mbf{1}^{\otimes k^{-}}\otimes\mbf{X}^{\otimes(\alpha^{-}+r+k^{+}-k^{-})}\big) \\
    =&g\Big(\big((\mbf{1}\otimes\mbf{X}+\mbf{X}\otimes\mbf{1})^{\otimes k^{+}}\otimes\mbf{X}^{\otimes(\alpha^{+}+r-2k^{+})}\big)\otimes\big((\mbf{1}\otimes\mbf{X}+\mbf{X}\otimes\mbf{1})^{\otimes k^{-}} \\
    &\qquad\qquad\qquad\qquad\qquad\qquad\otimes\mbf{X}^{\otimes(\alpha^{-}+r-2k^{-})}\big)\otimes\big(\mbf{X}^{\otimes(\alpha^{-}+r)}\big)\otimes\big(\mbf{X}^{\otimes(\alpha^{+}+r)}\big)\Big) \\
    =&\big((\mbf{1}\otimes\mbf{X}+\mbf{X}\otimes\mbf{1})^{\otimes k^{+}}\otimes\mbf{X}^{\otimes(\alpha^{+}+r-2k^{+})}\big)\otimes\big((\mbf{1}\otimes\mbf{X}+\mbf{X}\otimes\mbf{1})^{\otimes k^{-}} \\
    &\qquad\qquad\qquad\qquad\qquad\qquad\otimes\mbf{X}^{\otimes(\alpha^{-}+r-2k^{-})}\big)\otimes\big(\mbf{X}^{\otimes(\alpha^{+}+r)}\big)\otimes\big(\mbf{X}^{\otimes(\alpha^{-}+r)}\big),
\end{align*}
so that in particular we have that $v\in\im(g\circ\epsilon)=\im(F_{\alpha^{\pm}+r,\alpha^{\pm}+r})$ by the preceding discussion. Then
\begin{align*}
    (\id\otimes\varepsilon_{\alpha^{\pm}+r})(v)&=\big((\mbf{1}\otimes\mbf{X}+\mbf{X}\otimes\mbf{1})^{\otimes k^{+}}\otimes\mbf{X}^{\otimes(\alpha^{+}+r-2k^{+})}\big) & & \\
    &\qquad\qquad\otimes\big((\mbf{1}\otimes\mbf{X}+\mbf{X}\otimes\mbf{1})^{\otimes k^{-}}\otimes\mbf{X}^{\otimes(\alpha^{-}+r-2k^{-})}\big) & & \\
    &\sim2^{k^{+}+k^{-}}x_{r,k^{\pm}}, & &(\text{symmetry relations})
\end{align*}
whereas $(\varepsilon_{\alpha^{\pm}+r}\otimes\id)(v)=0$, as desired.

For (2), suppose $w\in\im(F_{\alpha^{\pm}+r,\alpha^{\pm}+s})$ is such that the hypothesis of (\ref{eq:w_S^2xD^2}) holds. Then
\[w=\big(\mbf{X}^{\otimes(\alpha^{+}+r)}\big)\otimes\big(\mbf{X}^{\otimes(\alpha^{-}+r)}\big)\otimes\big(\mbf{X}^{\otimes(\alpha^{+}+s)}\big)\otimes\big(\mbf{X}^{\otimes(\alpha^{-}+s)}\big)+\wt{x}\]
for some $\wt{x}\in\ker(\id\otimes\varepsilon_{\alpha^{\pm}+s})$, and hence the conclusion of (\ref{eq:w_S^2xD^2}) holds for $x=(\varepsilon_{\alpha^{\pm}+r}\otimes\id)(\wt{x})\in\ker(\varepsilon_{\alpha^{\pm}+s})$, which is what was to be shown.

For $\diamond=T$, consider the disoriented annulus maps
\begin{align*}
    &\psi_{T}^{+,\mbf{X}_{+}}:\KhR_{2}\big(U(k^{+},k^{-})\big)\to\KhR_{2}\big(U(k^{+}+2,k^{-})\big) \\
    &\psi_{T}^{-,\mbf{X}_{-}}:\KhR_{2}\big(U(k^{+},k^{-})\big)\to\KhR_{2}\big(U(k^{+},k^{-}+2)\big)
\end{align*}
from Proposition \ref{prop:disoriented_bundt_cake}. We see that
\[\psi_{T}^{\pm,\mbf{X}_{\pm}}(v)=A_{0}^{\pm 2}\cdot v\qquad\text{for all }v\in\ol{\skein}_{0}^{2,T}(S^{2}\times D^{2}),\]
and hence the disoriented annulus relations $\psi_{T}^{\pm,\mbf{X}_{\pm}}(v)\sim v$ from (\ref{eq:disoriented_annulus_relations}) imply the relation $A_{0}^{2}-1=0$. That this is the only additional relation follows from a similar argument as for (1) and (2) above in the $\diamond=O$ case, which we leave as an exercise for the inspired reader. 
\end{proof}

We have the following corollary:

\begin{corollary}\label{cor:sumsS2xD2}
Let $m_{1},\dots,m_{k}\in\bb{N}$ for some $k\in\bb{N}$, and let $D_{\std}=\#_{i=1}^{k}\natural_{j=1}^{m_{i}}S^{2}\times D^{2}$. Then we have algebra isomorphisms
\begin{align*}
    &\ol{\cal{S}}^{2,O}_{0}(D_{\std})\cong\bb{Q}[A_{0,i,j},A_{0,i,j}^{-1}\;|\;1\le i\le k, 1\le j\le m_{i}] \\
    &\ol{\cal{S}}^{2,T}_{0}(D_{\std})\cong\bb{Q}[A_{0,i,j}\;|\;1\le i\le k, 1\le j\le m_{i}]/(A^{2}_{0,i,j}-1\;|\;1\le i\le k, 1\le j\le m_{i})
\end{align*}
where each $A_{0,i,j}^{\pm 1}$ is concentrated in tri-degree $(0,0,\pm e_{i,j})$, with $e_{i,j}$ denoting the second homology class of $\#_{i=1}^{k}\natural_{j=1}^{m_{i}}S^{2}\times D^{2}$ represented by the $j$-th core sphere of the $i$-th connected summand (compare \cite{RSWWZ25}, Example 2.2).
\end{corollary}

\begin{proof}
The fact that $\ol{\skein}^{2,\diamond}_{0}(D_{\std})$ carries an algebra structure for $\diamond\in\{O,T\}$ follows from the observation that the inclusion map $i:\del D_{\std}\times I\hookrightarrow D_{\std}$ as a collar neighborhood of the boundary induces an isomorphism
\[\ol{\skein}^{2,\diamond}_{0}(i):\ol{\skein}^{2,\diamond}_{0}(\del D_{\std}\times I)\xrightarrow{\cong}\ol{\skein}^{2,\diamond}_{0}(D_{\std})\]
and that $\ol{\skein}^{2,\diamond}_{0}(\del D_{\std}\times I)$ carries a natural algebra structure under stacking along the $I$-direction (see \cite{RSWWZ25}, Example 2.3). The desired result then follows by Proposition \ref{prop:tortellini_S^2xD^2} combined with Proposition \ref{prop:connected_sum}.
\end{proof}

We also have the following analogue of (\cite{RSWWZ25}, Theorem 2.7):

\begin{prop}
\label{prop:D_st_w_link}
Let $D_{\std}=\#_{i=1}^{k}\natural^{m_{i}}S^{2}\times D^{2}$ as above, and let $L\subset\del X=\sqcup_{i=1}^{k}\#^{m_{i}}S^{1}\times S^{2}$ be an admissible link. Then for $\diamond\in\{O,T\}$ and each $\bm{\alpha}\in H_{2}^{L}(X;\bb{Z})$ there exists a canonical isomorphism
\begin{equation*}
    \ol{\skein}_{0}^{2,\diamond}(D_{\std};L;\bm{\alpha})\cong\KhR_{2,\diamond}^{+}(\del D_{\std},L)\otimes\ol{\skein}_{0}^{2,\diamond}(D_{\std};\emptyset;\bm{\alpha}-\bm{\alpha}_{L}), \\
\end{equation*}
where $\bm{\alpha}_{L}\in H_{2}^{L}(D_{\std};\bb{Z})$ is the unique class having trivial intersection with the cocores.
\end{prop}

\begin{proof}
As in the proof of (\cite{RSWWZ25}, Theorem 2.7), we can add dual projectors (see Remark \ref{rmk:dual_projector}) near the belts and ``slide" the belts off of the projectors. The belt-sliding isomorphisms commute with the lasso relation, and hence for $\diamond\in\{O,T\}$ we obtain an isomorphism
\[\ol{\cal{S}}_{0}^{2,\diamond}(D_{\std};L;\bm{\alpha})\cong\KhR_{2,\diamond}^{-}(\sqcup_{i=1}^{k}S^{3},L^{\circ}\otimes(\otimes_{j=1}^{m}P^{\vee}_{\ell_{j},0}))\otimes\ol{\cal{S}}_{0}^{2,\diamond}(D_{\std};\emptyset;\bm{\alpha}-\bm{\alpha}_{L}),\]
where $L^{\circ}\subset\sqcup_{i=1}^{k}S^{3}\setminus\sqcup_{j=1}^{m}B^{3}$ denotes the tangle obtained from $L$ obtained by removing neighborhoods of the surgery regions, and $P^{\vee}_{\ell_{j},0}$ denotes a dual projector which is inserted in the $j$th surgery region for $j=1,\dots, m$. The result then follows from the observation that
\begin{align*}
    \KhR_{2,\diamond}^{-}(\sqcup_{i=1}^{k}S^{3},L^{\circ}\otimes(\otimes_{j=1}^{m}P^{\vee}_{\ell_{j},0}))
    &\cong\KhR_{2,\diamond}(\sqcup_{i=1}^{k}S^{3},L^{\circ}\otimes(\otimes_{j=1}^{m}P^{\vee}_{\ell_{j},0})) \\
    &\cong\KhR_{2,\diamond}^{+}(\del D_{\std},L).
\end{align*}
\end{proof}

\subsection{\texorpdfstring{$\mathbf{D(p)}$, $p\neq 0$}{D(p), p≠0}}
\label{subsec:non_zero_disk_bundles}

Let $D(p)$ denote the $D^{2}$-bundle over $S^{2}$ with Euler number $p\in\bb{Z}$. As calculated in (\cite{MN22}, Theorem 1.3), we have that
\begin{equation}
\label{eq:D(p)_lasagna_vanishing_p>0}
    \skein_{0,0,*}^{2}(D(p);0)=0
\end{equation}
for $p>0$, and for $p<0$ it was shown that
\begin{equation}
\label{eq:D(p)_lasagna_p<0}
    \skein_{0,0,*}^{2}(D(p);0)\cong\bb{Q}
\end{equation}
concentrated in $q$-degree $0$. One can say more about the former case. Indeed, (\cite{RW24}, Theorem 1.4) implies that $\skein_{0}^{2}(D(p))$ vanishes in all tri-gradings. As mentioned in the introduction, these same statements hold for $1$-dimensional inputs lasagna:

\begin{proposition}
\label{prop:tortellini_D(p)}
We have the following isomorphisms:
\begin{align*}
    &\ol{\skein}_{0,}^{2,\diamond}(D(p))=0\qquad\text{for }p>0, &
    &\ol{\skein}_{0,0,*}^{2,\diamond}(D(p);0)\cong\bb{Q}\qquad\text{for }p<0,
\end{align*}
where the latter is concentrated in $q$-degree $0$.
\end{proposition}

We will dedicate the remainder of this section to proving Proposition \ref{prop:tortellini_D(p)}. The first statement follows from (\ref{eq:D(p)_lasagna_vanishing_p>0}) and Proposition \ref{prop:gluing_surjection}. For the rest of this section, let $p<0$. In order to prove the second statement, it suffices to show that the surjection
\[f_{0,\diamond}:\skein_{0,0,0}^{2}(D(p);0)\twoheadrightarrow\ol{\skein}_{0,0,0}^{2,\diamond}(D(p);0)\]
from Proposition \ref{prop:surjections_lasagna_simply_connected} in tri-degree $(0,0,0)$ is an isomorphism for $\diamond\in\{O,T\}$. Let $(U,p)$ denote the $p$-framed unknot, and note that a $(k^{+},k^{-})$-cable of $(U,p)$ is given by
\[(U,p)(k^{+},k^{-})=T(k^{+}+k^{-},p(k^{+}+k^{-}))_{k^{+},k^{-}},\]
where $T(k^{+}+k^{-},p(k^{+}+k^{-}))_{k^{+},k^{-}}$ denotes the corresponding torus link, with each component having framing $p$ and with the subscripts denoting that $k^{+}$ of the strands are oriented positively, and $k^{-}$ of the strands are oriented negatively. Next, let $\belt{k^{+},k^{-}}_{p}\subset S^{1}\times S^{2}$ denote the $p$-twisted standard belt link, as described in Section \ref{subsec:conventions}. Note that we can identify 
\[(U,p)(1)=\belt{1,0}_{p}\cup\belt{0,1}\subset S^{1}\times S^{2},\]
and so consequently for any $k^{\pm},\ell^{\pm}\in\bb{N}$ we have an identification
\[(U,p)(k^{\pm};\ell^{\pm};1)=\belt{k^{+},k^{-}}_{p}\cup\belt{\ell^{+},\ell^{-}}\subset S^{1}\times S^{2}.\]
As in (\ref{eq:Phi_Psi_decomps_1_hbody}) we can identify $\Phi_{\diamond,k^{\pm},\ell^{\pm},1}$ and $\Psi_{\diamond,k^{\pm},\ell^{\pm},1}$ with the compositions of maps as pictured in Figure \ref{fig:D(p)_calc_1}.

\begin{center}
\begin{figure}
    \includegraphics[width=14cm]{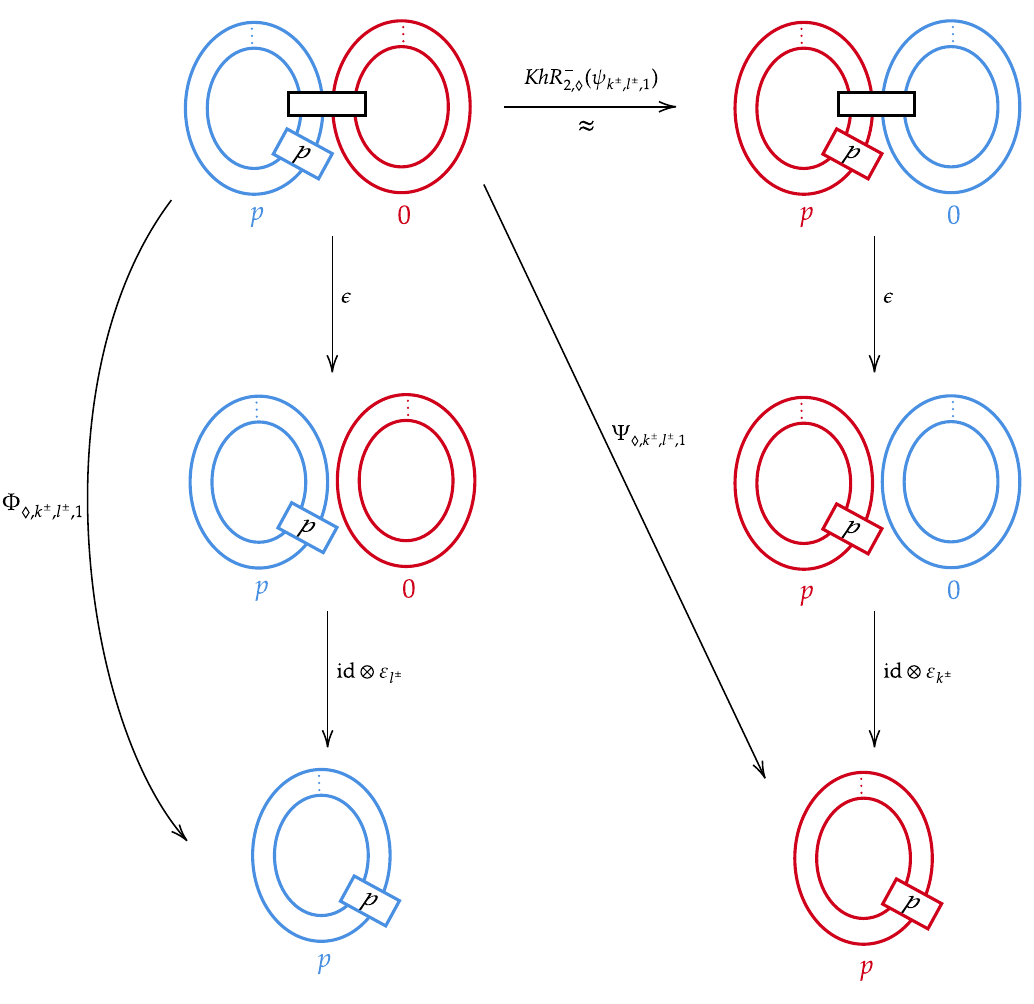}
    \caption{The relevant diagram for computing $\Phi_{\diamond,k^{\pm},\ell^{\pm},1}$ and $\Psi_{\diamond,k^{\pm},\ell^{\pm},1}$ for $D(p)$.}
    \label{fig:D(p)_calc_1}
\end{figure}
\end{center}

For all $k^{+},k^{-}\in\bb{N}$ such that
\begin{itemize}
    \item $k^{+}=k^{-}$ if $\diamond=O$,
    \item $k^{+}\equiv k^{-}\pmod{2}$ if $\diamond=T$,
\end{itemize}
we define $v_{\diamond,k^{\pm},p}\in\KhR_{2,\diamond}^{0,k^{+}+k^{-}}(T(k^{+}+k^{-},p(k^{+}+k^{-}))_{k^{+},k^{-}})$ to be the element
\begin{equation}
\label{eq:v_diamond,k_pm_p}
    v_{\diamond,k^{\pm},p}:=
    {
	\left\{
		\begin{array}{ll}
			\underbrace{(\psi^{[1]}\circ\cdots\circ\psi^{[1]})}_{k}(1) & \mbox{if } k^{+}=k^{-}=:k, \\
			\underbrace{(\psi^{[1]}\circ\cdots\circ\psi^{[1]})}_{k^{-}}\underbrace{(\psi_{T}^{+,\mbf{X}_{+}}\circ\cdots\circ\psi_{T}^{+,\mbf{X}_{+}})}_{\frac{1}{2}(k^{+}-k^{-})}(1) & \mbox{if } k^{+}>k^{-}, \\
            \underbrace{(\psi^{[1]}\circ\cdots\circ\psi^{[1]})}_{k^{+}}\underbrace{(\psi_{T}^{-,\mbf{X}_{-}}\circ\cdots\circ\psi_{T}^{-,\mbf{X}_{-}})}_{\frac{1}{2}(k^{-}-k^{+})}(1) & \mbox{if } k^{+}<k^{-},
		\end{array}
	\right.
    }
\end{equation}
where $1\in\KhR_{2}^{0,0}(S^{3},\emptyset)\cong\bb{Q}$ and $\psi_{T}^{\pm,\mbf{X}_{\pm}}$ are the disoriented annulus maps from Proposition \ref{prop:disoriented_bundt_cake} (note that the last two cases only occur when $\diamond=T$). We claim that $\ol{\skein}_{0,0,0}^{2,\diamond}(D(p);0)\cong\bb{Q}$ is generated by the images of the elements $v_{T,k^{\pm},p}$, with each element (after fixing a global sign) corresponding to $1\in\bb{Q}$.

\begin{lemma}
\label{lem:v_diamond,k_pm,p}
Let $\diamond\in O,T$, and let $k^{+},k^{-}\in\bb{N}$ be such that:
\begin{itemize}
    \item $k^{+}=k^{-}$ if $\diamond=O$,
    \item $k^{+}\equiv k^{-}\pmod{2}$ if $\diamond=T$.
\end{itemize}
The following statements are true:
\begin{enumerate}
    \item The counit map
    \[\epsilon:\KhR_{2,\diamond}^{-}\big(\belt{k^{+},k^{-}}_{p}\big)\to\KhR_{2,\diamond}\big(T(k^{+}+k^{-},p(k^{+}+k^{-}))_{k^{+},k^{-}}\big)\]
    is an isomorphism in bi-degree $(0,k^{+}+k^{-})$.
    \item Let $\wt{\varepsilon}_{\diamond,k^{\pm},p}$ denote the composition of maps
    \begin{align*}
        &\KhR_{2,\diamond}^{0,k^{+}+k^{-}}\big(T(k^{+}+k^{-},p(k^{+}+k^{-}))_{k^{+},k^{-}}\big)\xrightarrow{(\epsilon)^{-1}}\KhR_{2,\diamond}^{-,0,k^{+}+k^{-}}\big(\belt{k^{+},k^{-}}_{p}\big) \\
        &\qquad\qquad\qquad\qquad\xrightarrow{\tau_{\diamond}^{-p}}\KhR_{2,\diamond}^{-,0,k^{+}+k^{-}}\big(\belt{k^{+},k^{-}}\big)\xrightarrow{\epsilon}\KhR_{2,\diamond}^{0,k^{+}+k^{-}}\big(U(k^{+},k^{-})\big)\xrightarrow{\varepsilon_{k^{\pm}}}\bb{Q}.
    \end{align*}
    Then $\wt{\varepsilon}_{\diamond,k^{\pm},p}$ induces an isomorphism $\KhR_{2,\diamond}^{0,k^{+}+k^{-}}(T(k^{+}+k^{-},p(k^{+}+k^{-}))_{k^{+},k^{-}})\cong\bb{Q}$, and
    \[(\wt{\varepsilon}_{\diamond,k^{\pm},p})^{-1}(1)=v_{\diamond,k^{\pm},p}\in\KhR_{2,\diamond}^{-,0,k^{+}+k^{-}}(T(k^{+}+k^{-},p(k^{+}+k^{-}))_{k^{+},k^{-}})\]
    where $v_{\diamond,k^{\pm},p}$ is as in (\ref{eq:v_diamond,k_pm_p}).
\end{enumerate}
\end{lemma}

\begin{proof}

We first prove (1). To facilitate computations we will restate everything in terms of Rozansky-Willis conventions.
Let $r=\frac{1}{2}(k^{+}+k^{-})$. Note that we have isomorphisms
\begin{align*}
    &\KhR_{2,\diamond}^{-,0,k^{+}+k^{-}}\big(\belt{k^{+},k^{-}}_{p}\big)\cong\KhR_{2,O}^{-,0,2r}\big(\belt{r,r}_{p}\big)\cong\Kh_{RW}^{0,-2r}\big(\belt{r,r}_{|p|}\big), \\
    &\KhR_{2,\diamond}^{0,k^{+}+k^{-}}\big(T(k^{+}+k^{-},p(k^{+}+k^{-}))_{k^{+},k^{-}}\big)\cong\KhR_{2,O}^{0,2r}\big(T(2r,2rp)_{r,r}\big)\cong\Kh^{0,-2r}\big(T(2r,2r|p|)_{r,r}\big),
\end{align*}
where the first isomorphism in each line follows from (\ref{eq:KhR_T_invariant_under_orientations}) and the observation that the writhes of $T(2r,2rp)_{r,r}$ and $\belt{r,r}_{p}$ are equal to zero, whereas the second isomorphism in each line follows from the isomorphism given by (\ref{eq:conventions}).

It therefore suffices to show that the corresponding counit map
\begin{equation}
\label{eq:D(p)_lemma_counit}
    \epsilon:\Kh_{RW}^{0,-2r}\big(\belt{r,r}_{|p|}\big)\to\Kh^{0,-2r}\big(T(2r,2r|p|)_{r,r}\big)
\end{equation}
in Rozansky-Willis homology is an isomorphism. One can check that the chain group $\text{CKh}^{0,-2r}(T(2r,2r|p|)_{r,r}$ is generated by the element $y$ corresponding to the all-$0$ resolution with all circles labeled $\mbf{X}$. One can check that $y$ survives after taking homology and hence
\[Kh^{0,-2r}\big(T(2r,2r|p|)_{r,r}\big)\cong\bb{Q}.\]
Next, let $\FT^{|p|}_{2r}$ denote the $|p|$-full-twist braid on $2r$ strands, and let $\cal{B}_{2r}$ denote the set of crossingless matchings of $2r$ points (see Definition \ref{def:Roz-cobar}). Interpreting $\FT^{|p|}_{2r}$ as a $(0,4r)$-tangle and $(\delta\otimes\delta^{t})$ as a $(4r,0)$ tangle for each $\delta\in\cal{B}_{2r}$, we can make an identification of chain groups
\[\text{CKh}_{RW}^{0,-2r}(\belt{r,r}_{|p|})=\bigoplus_{\delta\in\cal{B}_{2r}}\text{CKh}^{0,-r}\big(\FT^{|p|}_{2r}\otimes(\delta\otimes\delta^{t})\big),\]
under which we see that $\text{CKh}_{RW}^{0,-2r}(\belt{r,r}_{|p|})$ is generated by the collection of elements $x_{\delta}\in\text{CKh}^{0,-r}\big(\FT^{|p|}_{2r}\otimes(\delta\otimes\delta^{t})\big)$ where $x_{\delta}$ corresponds to the all-$0$ resolution with all circles labeled $\mbf{X}$. One can check that $\del x_{\delta}=0$ for all $\delta\in\cal{B}_{2r}$. Furthermore, the image of the differential in the relevant bigrading is generated by the collection of elements $\{x_{\delta}-x_{\delta'}\}_{\delta,\delta'\in\cal{B}_{2r}}$, and hence all the $x_{\delta}$ generators are identified together after taking homology. Denote this generator by $x\in\Kh_{RW}^{0,-2r}(\belt{r,r}_{|p|})$.

We can identify $\epsilon$ at the chain level with the sum of maps
\[\bigoplus_{\delta\in\cal{B}_{2r}}\wt{\nu}_{\delta}:\text{CKh}^{0,-r}\big(\FT^{|p|}_{2r}\otimes(\delta\otimes\delta^{t})\big)\to\text{CKh}^{0,-2r}\big(T(2r,2r|p|)_{r,r}\big)\]
induced by the saddle cobordism maps $\nu_{\delta}$ as in Definition \ref{def:Roz-cobar}. One can check that each of the $x^{\delta}$ generators are mapped to $y$ under the above map, and so after passing to homology we see that $\epsilon(x)=y$, hence the map in (\ref{eq:D(p)_lemma_counit}) is an isomorphism.

Next, we prove (2). Without loss of generality, assume $k^{+}\geq k^{-}$; the case $k^{+}\le k^{-}$ is analogous. Let:
\begin{itemize}
    \item $A(a,b):=\KhR_{2,\diamond}^{0,a+b}(T(a+b,p(a+b))_{a,b})$,
    \item $B(a,b):=\KhR_{2,\diamond}^{-,0,a+b}(\belt{a,b}_{p})$,
    \item $C(a,b):=\KhR_{2,\diamond}^{-,0,a+b}(\belt{a,b})$,
    \item $D(a,b):=\KhR_{2}^{0,a+b}(U(a,b))$,
\end{itemize}
and consider the following diagram:

\begin{center}
\begin{tikzcd}
\KhR_{2,\diamond}^{0,0}(\emptyset)\arrow[r, "\psi_{\diamond}^{+,\mbf{X}_{+}}"] & A(2,0)\arrow[r, "\psi_{\diamond}^{+,\mbf{X}_{+}}"] & \cdots\arrow[r, "\psi_{\diamond}^{+,\mbf{X}_{+}}"] & A(k^{+}-k^{-},0)\arrow[r, "\psi_{\diamond}^{[1]}"] & \cdots \arrow[r, "\psi_{\diamond}^{[1]}"] & A(k^{+},k^{-}) \\
\KhR_{2,\diamond}^{0,0}(\emptyset)\arrow[r, "\psi_{\diamond}^{+,\mbf{X}_{+}}"]\arrow[u, "\epsilon"]\arrow[d, "\tau_{\diamond}^{-p}"] & B(2,0)\arrow[u, "\epsilon"]\arrow[d, "\tau_{\diamond}^{-p}"]\arrow[r, "\psi_{\diamond}^{+,\mbf{X}_{+}}"] & \cdots\arrow[r, "\psi_{\diamond}^{+,\mbf{X}_{+}}"] & B(k^{+}-k^{-},0)\arrow[u, "\epsilon"]\arrow[d, "\tau_{\diamond}^{-p}"]\arrow[r, "\psi_{\diamond}^{[1]}"] & \cdots\arrow[r, "\psi_{\diamond}^{[1]}"] & B(k^{+},k^{-})\arrow[u, "\epsilon"]\arrow[d, "\tau_{\diamond}^{-p}"] \\
\KhR_{2,\diamond}^{0,0}(\emptyset)\arrow[d, "\epsilon"]\arrow[r, "\psi_{\diamond}^{+,\mbf{X}_{+}}"] & C(2,0)\arrow[d, "\epsilon"]\arrow[r, "\psi_{\diamond}^{+,\mbf{X}_{+}}"] & \cdots\arrow[r, "\psi_{\diamond}^{+,\mbf{X}_{+}}"] & C(k^{+}-k^{-},0)\arrow[d, "\epsilon"]\arrow[r, "\psi_{\diamond}^{[1]}"] & \cdots\arrow[r, "\psi_{\diamond}^{[1]}"] & C(k^{+},k^{-})\arrow[d, "\epsilon"] \\
\KhR_{2,\diamond}^{0,0}(\emptyset)\arrow[r, "\psi_{\diamond}^{+,\mbf{X}_{+}}"]\arrow[dd, "\varepsilon_{0,0}", near start] & D(2,0)\arrow[ddl, "\varepsilon_{2,0}", near start]\arrow[r, "\psi_{\diamond}^{+,\mbf{X}_{+}}"] & \cdots\arrow[r, "\psi_{\diamond}^{+,\mbf{X}_{+}}"] & D(k^{+}-k^{-},0)\arrow[ddlll, "\varepsilon_{k^{+}-k^{-},0}", near start]\arrow[r, "\psi^{[1]}"] & \cdots\arrow[r, "\psi_{\diamond}^{[1]}"] & D(k^{+},k^{-})\arrow[ddlllll, "\varepsilon_{k^{+},k^{-}}", near start]\\
 & & & & & \\
\bb{Q}. & & & & &
\end{tikzcd}
\end{center}

The fact that this diagram is the image of a commuting diagram in $\Links_{1,\diamond}$ under the functor $\KhR_{2,\diamond}^{-}$ implies that the diagram commutes, but apriori only up to sign (see Remark \ref{rmk:signs}). However, a coherent choice of sign can be made in $\frak{gl}_{2}$-webs. Here $\psi_{T}^{+,\mbf{X}}$ denotes the disoriented annulus map from Proposition \ref{prop:disoriented_bundt_cake} (note that these only appear when $k^{+}\neq k^{-}$, i.e., when $\diamond=T$). The maps in the first column are all (tautologically) isomorphisms, and by inspection the same is true for all of the maps in the fourth row. Furthermore by tracing through the proof of (\cite{MN22}, Proposition 6.1) we can conclude that all the maps in the first and third rows are isomorphisms. By naturality of the Gluck twist isomorphism, the maps in the second rows are isomorphisms as well. By (1), all of the maps from the second to the first row are isomorphisms. Hence every arrow in the diagram is an isomorphism, from which (2) immediately follows.
\end{proof}

\begin{proof}[Proof of Proposition \ref{prop:tortellini_D(p)}]
Let $k^{\pm},\ell^{\pm}\in\bb{N}$ be such that
\begin{itemize}
    \item $k^{+}=k^{-}$ and $\ell^{+}=\ell^{-}$ if $\diamond=O$,
    \item $k^{+}\equiv k^{-}\pmod{2}$ and $\ell^{+}\equiv \ell^{-}\pmod{2}$ if $\diamond=T$.
\end{itemize}
Let $k:=k^{+}+k^{-}$, $\ell:=\ell^{+}+\ell^{-}$. Consider the following diagram:

\begin{center}
\begin{tikzcd}
\KhR_{2,\diamond}^{-,0,k+\ell}\big(\belt{k^{+},k^{-}}_{p}\cup\belt{\ell^{+},\ell^{-}}\big) \arrow[d, "\epsilon"] \arrow[r, "\KhR_{2,\diamond}^{-}(\psi_{k^{\pm},\ell^{\pm},1})", "\cong"'] & \KhR_{2,\diamond}^{-,0,k+\ell}\big(\belt{\ell^{+},\ell^{-}}_{p}\cup\belt{k^{+},k^{-}}\big) \arrow[d, "\epsilon"] \\
\KhR_{2}^{0,k}\big(T(k,pk)_{k^{+},k^{-}}\big)\otimes\KhR_{2}^{0,2s}\big(U(\ell^{+},\ell^{-})\big)\arrow[d, "\id\otimes\varepsilon_{s}"] & \KhR_{2}^{0,\ell}\big(T(\ell,p\ell)_{\ell^{+},\ell^{-}}\big)\otimes\KhR_{2}^{0,k}\big(U(k^{+},k^{-}))\big)\arrow[d, "\id\otimes\varepsilon_{r}"] \\
\KhR_{2}^{0,k}\big(T(k,pk)_{k^{+},k^{-}}\big)\arrow[d, "\wt{\varepsilon}_{O,k^{\pm},p}"] & \KhR_{2}^{0,\ell}\big(T(\ell,p\ell)_{\ell^{+},\ell^{-}}\big)\arrow[d, "\wt{\varepsilon}_{O,\ell^{\pm},p}"] \\
\bb{Q}\arrow[r, "="] & \bb{Q}
\end{tikzcd}
\end{center}

By similar argument as in the proof of Lemma \ref{lem:v_diamond,k_pm,p}, the above diagram commutes. If $w\in\KhR_{2,\diamond}^{-,0,k+\ell}(\belt{k^{+},k^{-}}_{p}\cup\belt{\ell^{+},\ell^{-}})$ is such that $\Phi_{\diamond,k^{\pm},\ell^{\pm},1}(w)=v_{\diamond,k^{\pm},p}$, then by Lemma \ref{lem:v_diamond,k_pm,p}, commutativity of the above diagram, and by the description of the maps $\Phi_{\diamond,k^{\pm},\ell^{\pm},1}$ and $\Psi_{\diamond,k^{\pm},\ell^{\pm},1}$ from Figure \ref{fig:D(p)_calc_1}, we must have that $\Psi_{\diamond,k^{\pm},\ell^{\pm},1}(w)=v_{\diamond,\ell^{\pm},p}$. This proves the desired result.
\end{proof}

\subsection{\texorpdfstring{$\mathbf{B^{4}}$}{B4}}
\label{subsec:B^{4}}

By \ref{eq:tortellini_1_h-body} we have an isomorphism 
\[\ol{\skein}^{2,\diamond}_{0}(B^{4})\cong\KhR_{2,\diamond}^{-}(S^{3},\emptyset)\cong\bb{Q}\]
concentrated in tri-degree $(0,0,0)$ for $\diamond\in\{O,T\}$. In this section we will (re-)calculate $\ol{\skein}^{2,\diamond}_{0}(B^{4};\emptyset)$ using the Kirby diagram of $B^{4}$ with a single 1-handle and a $p$-framed canceling 2-handle, as pictured in Figure \ref{fig:B4andSigma_gxD^2_kirby_diagrams}.

First consider the case where $\diamond=O$. Using the canceling 1- and 2-handle pair description of $B^{4}$, Corollary \ref{cor:tortellini_1_2_handlebody} furnishes an isomorphism
\[\ol{\skein}^{2,\diamond}_{0}(B^{4};\emptyset)\cong\Big(\bigoplus_{r=0}^{\infty}\KhR_{2,O}^{-}\big(\belt{r,r}_{p})\big)\{-2r\}\Big)/\sim,\]
where $\sim$ denotes the equivalence relation generated by the symmetric, bundt cake, and lasso relations. Note that we can identify $\Phi_{O,r,s,1}$ and $\Psi_{O,r,s,1}$ with the compositions of maps as pictured in Figure \ref{fig:B^4_calc_1}, and where $\psi_{r,s,1}$ is the diffeomorphism pictured in Figure \ref{fig:psi_i_B^4}.

\begin{center}
\begin{figure}
    \includegraphics[width=14cm]{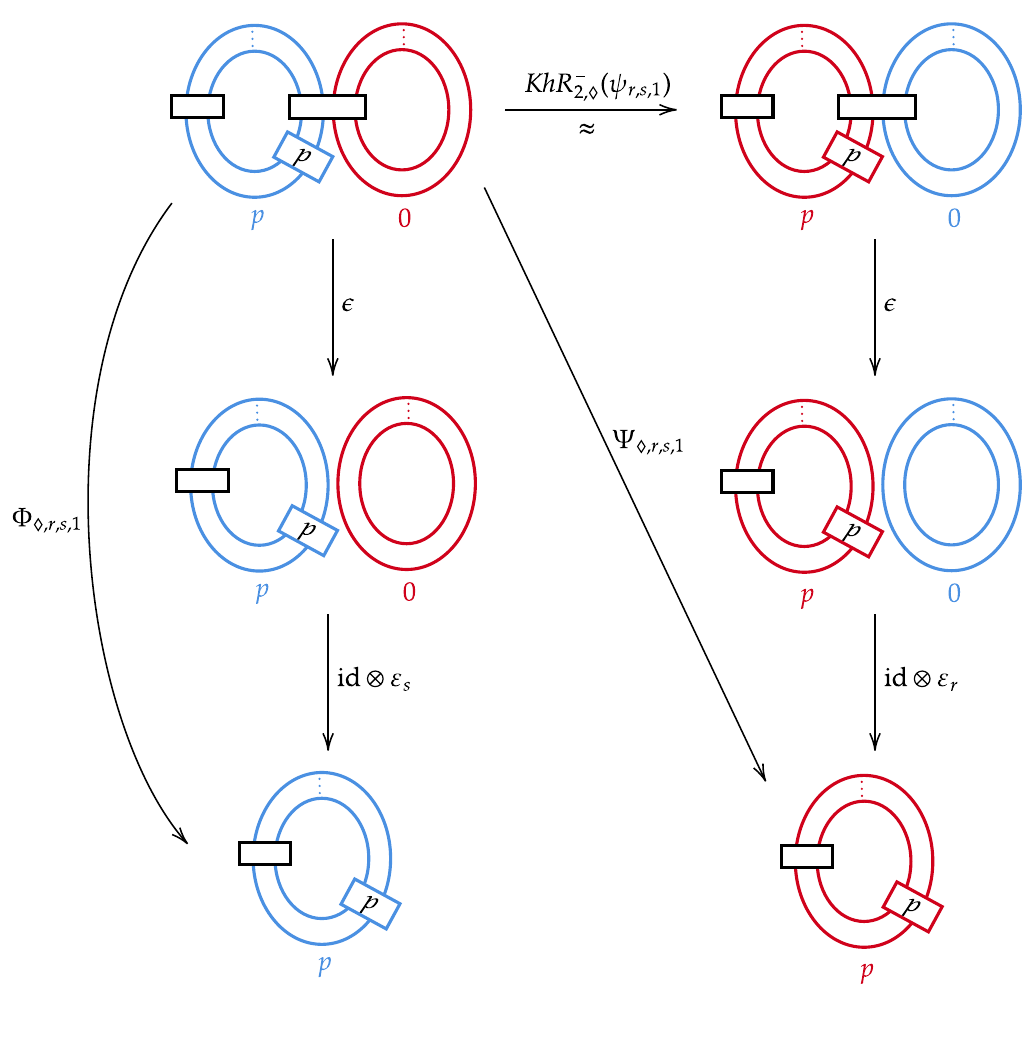}
    \caption{The maps $\Phi_{O,r,s,1}$ and $\Psi_{O,r,s,1}$ for the $p$-framed longitude of $S^{1}\times S^{2}$.}
    \label{fig:B^4_calc_1}
\end{figure}
\end{center}
\begin{center}
\begin{figure}
    \includegraphics[width=15cm]{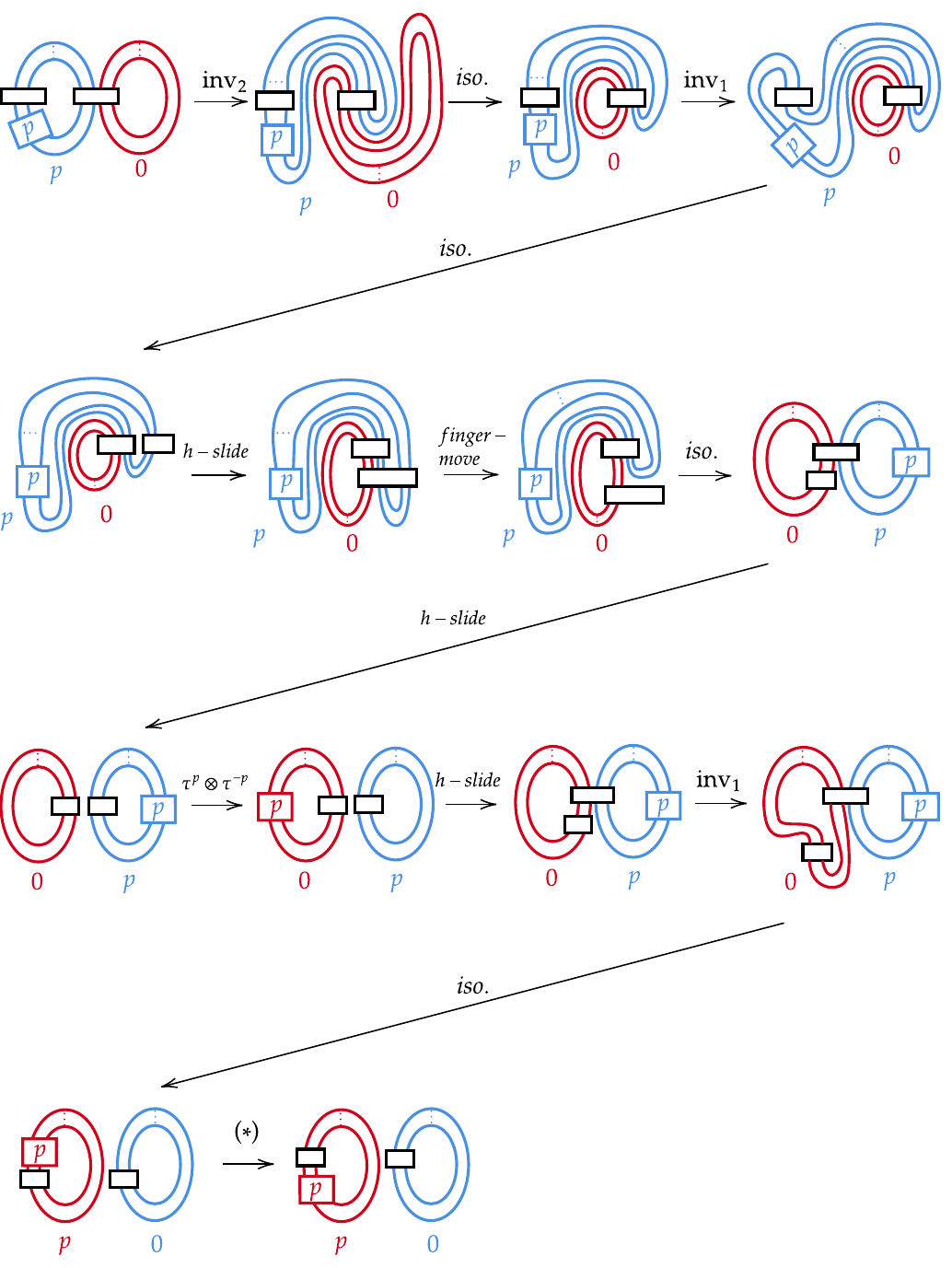}
    \caption{A diffeomorphism of pairs. The map $(*)$ is homotopy equivalence given by pulling the $p$-full twist $\mathrm{FT}^{\otimes{p}}_{2n}$ through a Rozansky projector.
    }
    \label{fig:psi_i_B^4}
\end{figure}
\end{center}

We can simplify these maps as follows. Let $L_{r,s,p}$ be the link featured in the upper left hand corner of Figure \ref{fig:B^4_calc_1}, and consider the isomorphism
\begin{align*}
    &\KhR_{2,O}^{-}\big(\#^{2}S^{1}\times S^{2},L_{r,s,p}\big)\{-2r-2s\} \\
    &\qquad\qquad\underset{\cong}{\xrightarrow{f}}\KhR_{2,O}^{-}\big(S^{1}\times S^{2},\belt{r,r}\big)\{-2r\}\otimes\KhR_{2,O}^{-}\big(S^{1}\times S^{2},\belt{s,s}\big)\{-2s\}
\end{align*}
obtained as the composition of isomorphisms as in Figure \ref{fig:B^4_calc_2}. We then have that
\begin{align*}
    &\Phi_{O,r,s,1}=\Phi_{O,r,s,1}'\circ f & &\Psi_{O,r,s,1}=\Psi_{O,r,s,1}'\circ f
\end{align*}
where $\Phi_{O,r,s,1}'$ and $\Psi_{O,r,s,1}'$ are the maps pictured in Figure \ref{fig:B^4_calc_3}.

\begin{center}
\begin{figure}
    \includegraphics[width=14cm]{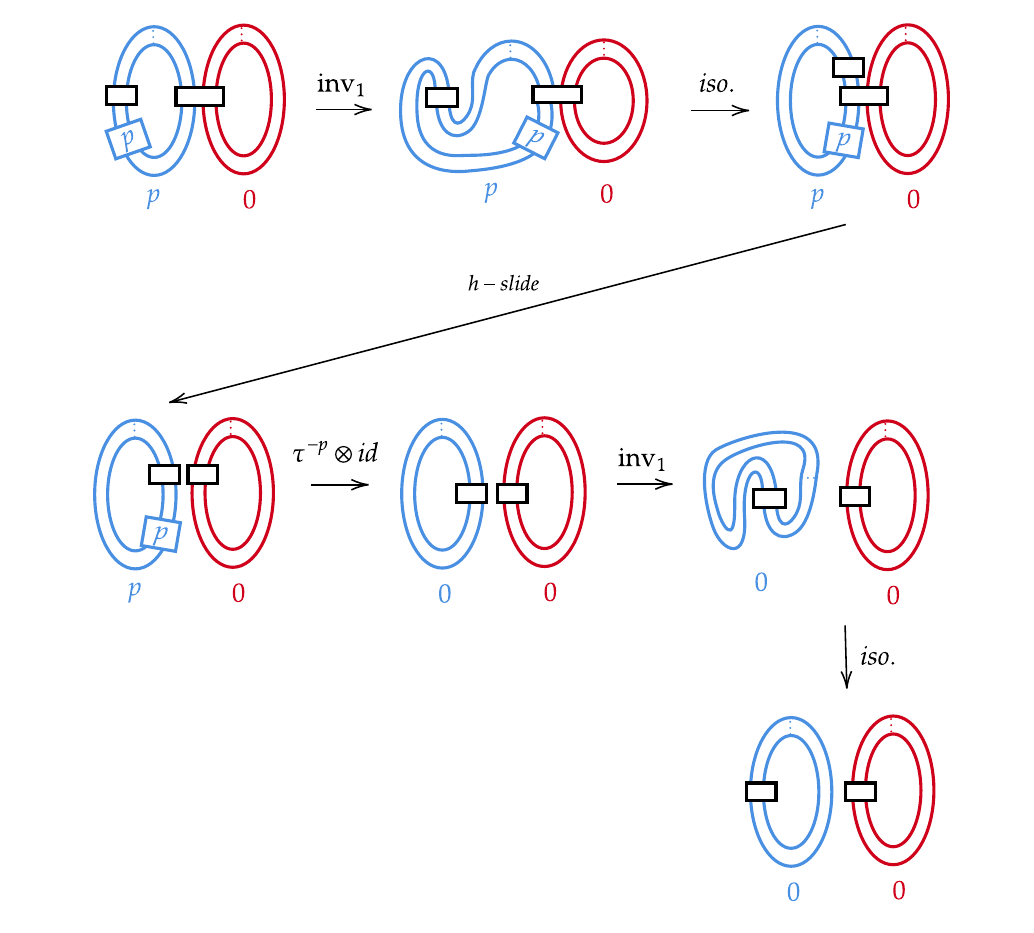}
    \caption{A sequence of isomorphisms interpolating between $L_{r,s,p}$ and $\belt{r,r}\sqcup\belt{s,s}$.}
    \label{fig:B^4_calc_2}
\end{figure}
\end{center}
\begin{center}
\begin{figure}
    \includegraphics[width=14cm]{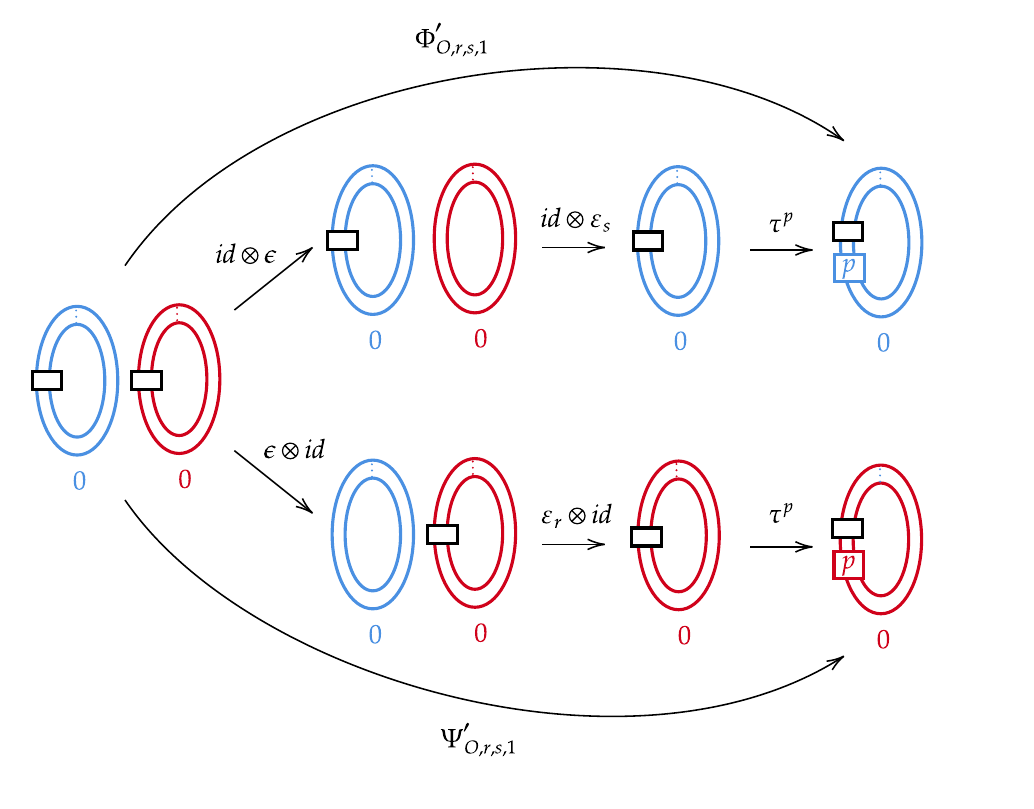}
    \caption{Alternate maps $\Phi'_{O,r,s,1}$ and $\Psi'_{O,r,s,1}$.}
    \label{fig:B^4_calc_3}
\end{figure}
\end{center}

We claim that under the above isomorphism, $\ol{\skein}^{2,O}_{0}(B^{4};\emptyset)$ can be identified with the one-dimensional $\bb{Q}$-vector space generated by the sequence of elements
\[\wt{v}_{O,r,p}:=(\epsilon)^{-1}(v_{O,r,p})\in\KhR_{2,O}^{-,0,2r}\big(S^{1}\times S^{2},\belt{r,r}_{p}\big),\qquad\qquad r\geq 0,\]
where $v_{O,r,p}\in\KhR_{2}^{0,2r}(T(2r,2rp)_{r,r})$ is defined as in (\ref{eq:v_diamond,k_pm_p}). Before we proceed, we shall consider the following lemma.

\begin{lemma}
\label{lem:B^4}
The following statements are true for each $r\in\bb{N}$ and $p\in\bb{Z}$:
\begin{enumerate}
    \item There exists a canonical isomorphism
    \[\KhR_{2,O}^{-,0,2r}\big(S^{1}\times S^{2},\belt{r,r}_{p}\big)\underset{\cong}{\xrightarrow{f_{O,r,p}}}\bb{Q}\]
    such that the following diagram commutes:
    \begin{equation}
    \label{eq:D(p)_lemma_cd}
    \begin{tikzcd}
    \KhR_{2,O}^{-,0,2r}\big(S^{1}\times S^{2},\belt{r,r}_{p}\big) \arrow[rd, "f_{O,r,p}"', "\cong"] \arrow[rr, "\psi^{[1]}_{O}", "\cong"'] & & \KhR_{2,O}^{-,0,2r+2}\big(S^{1}\times S^{2},\belt{r+1,r+1}_{p}\big) \arrow[ld, "f_{O,r+1,p}", "\cong"'] \\
    & \bb{Q}. &
    \end{tikzcd}
    \end{equation}
    \item We have a short exact sequence
    \begin{equation}
    \label{eq:D(p)_SES}
        \bigoplus_{(h,q)\neq(0,2r)}\KhR_{2,O}^{-,h,q}\big(S^{1}\times S^{2},\belt{r,r}_{p}\big)\xrightarrow{i}\KhR_{2,O}^{-}\big(S^{1}\times S^{2},\belt{r,r}_{p}\big)\xrightarrow{\varepsilon'_{O,r,p}}\bb{Q}
    \end{equation}
    where $i$ denotes the canonical inclusion map and $\varepsilon'_{O,r,p}:=\varepsilon_{r}\circ\epsilon\circ\tau_{O}^{-p}$.
\end{enumerate}
\end{lemma} 

\begin{proof}
By naturality of the Gluck twist isomorphism $\tau_{O}^{p}$, it suffices to prove the result for $p=0$. For (1), note that we have an isomorphism
\begin{equation}
\label{eq:D_p_lemma_isomorphism}
    \KhR_{2,O}^{-,0,2r}\big(S^{1}\times S^{2},\belt{r,r}\big)\underset{\cong}{\xrightarrow{h}}HH_{0}(H^{r}\otimes\bb{Q})_{2r},
\end{equation}
where $H^{r}$ denotes Khovanov's arc algebra as defined in \cite{Kho02}, and $HH_{0}(H^{r}\otimes\bb{Q})_{2r}$ denotes the piece of the zeroth Hochschild homology of the (rationalized) arc algebra in quantum grading $q=2r$. The isomorphism $h$ is furnished by (\ref{eq:conventions}) along with the discussion in (\cite{MN22}, Section 6). By tracing through the proof of (\cite{MN22}, Theorem 6.3) provided originally by (\cite{Kho04}, Theorems 3 and 4), we have that 
\[HH_{0}(H^{r})_{2r}\cong Z(H^{r})^{\vee}_{2r}\cong Z(H^{r})_{0}\cong\bb{Q}\]
is generated by a single element $x_{r}$, where:
\begin{itemize}
    \item If $r=0$, $x_{0}=1\in HH_{0}(H^{0}))_{0}\cong HH_{0}(\bb{Q})\cong\bb{Q}$.
    \item If $r\geq 1$, then $x_{r}\in HH_{0}(H^{r})_{2r}$ is the unique element which is represented by any all $\mbf{X}$-labeled generator of $H^{r}$.
\end{itemize}
We then set $f_{O,r,0}$ to be the composition of these isomorphisms. Furthermore, $h$ intertwines the dotted annulus homomorphism
\[\psi_{O}^{[1]}:\KhR_{2,O}^{-,0,2r}\big(S^{1}\times S^{2},\belt{r,r}\big)\to\KhR_{2,O}^{-,0,2r+2}\big(S^{1}\times S^{2},\belt{r+1,r+1}\big)\]
with the map
\begin{align*}
    \phi:HH_{0}(H^{r}\otimes\bb{Q})_{2r}&\xrightarrow{\cong} HH_{0}(H^{r+1}\otimes\bb{Q})_{2r+2} \\
    x_{r}&\mapsto x_{r+1},
\end{align*}
which implies commutativity of (\ref{eq:D(p)_lemma_cd}). For (2), it suffices to show that:
\begin{enumerate}[label=(\alph*)]
    \item The map
    \[\KhR_{2,O}^{-,0,2r}\big(S^{1}\times S^{2},\belt{r,r}_{p}\big)\xrightarrow{\varepsilon_{r}\circ\epsilon}\bb{Q}\]
    induces an isomorphism.
    \item The map $\varepsilon_{r}\circ\epsilon$ vanishes on $\KhR_{2,O}^{-,h,q}\big(S^{1}\times S^{2},\belt{r,r}\big)$ for all $(h,q)\neq (0,2r)$.
\end{enumerate}
For (a), let 
\begin{equation}
\label{eq:wt_v_diamond_r_0}
    \wt{v}_{O,r,0}:=f_{O,r,0}^{-1}(1)\in\KhR_{2,O}^{-,0,2r}\big(S^{1}\times S^{2},\belt{r,r}\big).
\end{equation}
By the description of $x_{r}$, we see that $\wt{v}_{O,r,0}$ is sent to the all $\mbf{X}$-labeled generator of $\KhR_{2}^{0,2r}(U(r,r))$ under the counit map $\epsilon$, hence
\[(\varepsilon_{r}\circ\epsilon)(\wt{v}_{O,r,0})=\varepsilon_{r}\big(\mbf{X}^{\otimes 2r}\big)=1.\]
For (b), note that since $\KhR_{2}(U(r,r))$ is supported in homological degree zero it follows that the counit map $\epsilon$ is the zero map for all homological degrees $\neq 0$. Hence, it suffices to consider the case $h=0$. We claim that any element
\[v\in\KhR_{2,O}^{-,0,q}\big(S^{1}\times S^{2},\belt{r,r}\big)\]
with $q<2r$ must be represented by a linear combination of generators, each of which contains at least one $\mbf{1}$ label. Indeed, by the above discussion, any element represented by an all $\mbf{X}$-labeled generator would be equivalent to the element $\wt{v}_{O,r,0}$ in $q$-degree $2r$, contradicting the assumption that $v$ is supported in $q$-degree $q<2r$. Hence 
\[\epsilon(v)\in\ker\big(\KhR_{2}^{0,2r}\big(U(r,r)\big)\xrightarrow{\varepsilon_{r}}\bb{Q}\big)\]
and so $(\varepsilon_{r}\circ\epsilon)(v)=0$.
\end{proof}

In view of Lemma \ref{lem:B^4}, it suffices to show the following:
\begin{enumerate}
    \item For every $r\in\bb{N}$ and $x\in\ker(\varepsilon'_{O,r,0})$, there exists some $s\in\bb{N}$ and some element 
    \[v\in\KhR_{2,O}^{-,0,2r}\big(S^{1}\times S^{2},\belt{r,r}\big)\otimes\KhR_{2,O}^{-,0,2s}\big(S^{1}\times S^{2},\belt{s,s}\big)\]
    such that $\Phi'_{O,r,s,1}(v)=x$ and $\Psi'_{O,r,s,1}(v)=0$.
    \item For any $r,s\in\bb{N}$ and any element
    \[w\in\KhR_{2,O}^{-}\big(S^{1}\times S^{2},\belt{r,r}\big)\{-2r\}\otimes\KhR_{2,O}^{-}\big(S^{1}\times S^{2},\belt{s,s}\big)\{-2s\}\]
    the following implication holds:
    \begin{equation}
    \label{eq:w_B^4}
        \Phi'_{O,r,s,1}(w)=\wt{v}_{O,r,p}\implies\Psi'_{O,r,s,1}(w)=\wt{v}_{O,s,p}+x
    \end{equation}
    for some $x\in\ker(\varepsilon'_{O,r,0})$.
\end{enumerate}
For (1), one can check that $v=x\otimes\wt{v}_{O,r,0}$ satisfies the desired property. For (2), note that if $w$ is such that $\Phi'_{O,r,s,1}(w)=\wt{v}_{O,r,p}$, then we must have $w=\wt{v}_{O,r,0}\otimes\wt{v}_{O,s,0}$, hence (\ref{eq:w_B^4}) holds with $x=0$.

\subsection{\texorpdfstring{$\mathbf{\Sigma_{g}\times D^{2}}$}{ΣgxD2}}
\label{subsec:Sigma_gxD^2}

By Theorem \ref{theorem:tortellini_2_handle_formula} and observing that the Rozansky projector is supported in non-positive homological degree, one may immediately obtain $\overline{\skein}_{0}^{2,\diamond}$ vanishing results for manifolds represented by Kirby diagrams satisfying restrictive properties.

\begin{proposition}\label{thrm:crossingless0framed}
    Let $X$ be a 4-manifold with a handle decomposition given by a Kirby diagram consisting of a link $L_{0}\cup L_{1}$, where $L_{0}$ is a crossingless diagram of the $0$-framed unlink, and $L_{1}$ is a crossingless unlink with dotted components representing the 1-handles of $X$. Then in homological degree $h\geq{1}$ we have 
    \[\overline{\skein}_{0,h\geq{1},*}^{2,\diamond}(X)=0.\]
\end{proposition}

\begin{proof}
    For such a Kirby diagram, each cable $L_{0}(r,r)$ admits an admissible diagram with no crossings and all 0 framed components. After replacing the components of $L_{1}$ with Rozansky projectors in the diagram, we note that the resulting Khovanov complex is supported entirely in non-positive homological degree. Thus, the cabling colimit of Corollary \ref{cor:tortellini_1_2_handlebody} vanishes for any $h\geq{1}$. 
\end{proof}

\begin{corollary}\label{cor:surfacebundle}
Let $\Sigma_{g}$ denote the orientable surface of genus $g\geq 0$. Then
\[\overline{\skein}_{0,h\geq{1},*}^{2,\diamond}(\Sigma_{g}\times{D^{2})}=0.\]
\end{corollary}

\begin{proof}
    There exists a Kirby diagram of $\Sigma_{g}\times{D^{2}}$ consisting of a single $0$-framed crossingless unknot linked with $2g$ dotted 1-handles given by a crossingless unlink such that the diagram of the corresponding knot $K\subset\#^{2g}S^{1}\times S^{2}$ obtained from performing surgeries on the dotted 1-handle components is admissible (see Figure \ref{fig:B4andSigma_gxD^2_kirby_diagrams}). The result then follows immediately from Proposition \ref{thrm:crossingless0framed}.
\end{proof}

\begin{center}
\begin{figure}
    \includegraphics[width=16cm]{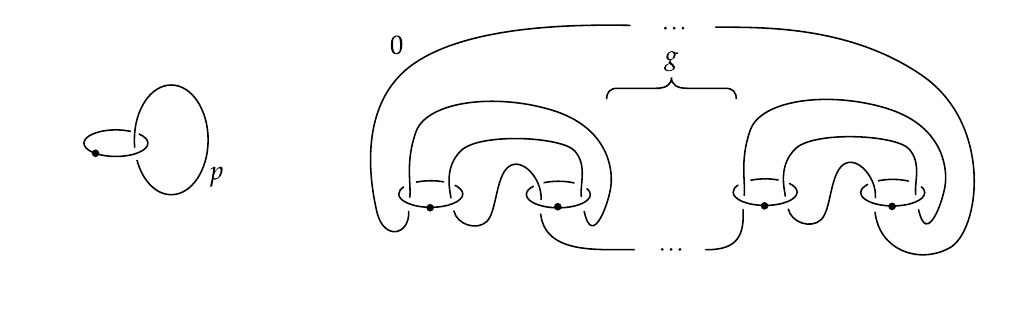}
    \caption{\textbf{Left:} A Kirby diagram of $B^{4}$. \textbf{Right:} A Kirby diagram for $\Sigma_{g}\times{D^{2}}$, the disk bundle over the genus $g$ surface $\Sigma_{g}$ with Euler number 0, built from $2g$ 1-handles and a single $0$-framed 2-handle.}
    \label{fig:B4andSigma_gxD^2_kirby_diagrams}
\end{figure}
\end{center}

\begin{remark}
\label{rmk:FTApprox}
In \cite{ROZ-Cat}, Rozansky shows that the projector $P_{n,0}$ is realized by the stable limit of tensor powers of a full-twist braid on $2n$ strands for particular choices of connecting maps. For a link $L\subset{S^{1}\times{S^{2}}}$, realized as the closure of a tangle $T$ passing through the surgery region of $S^{1}\times{S^{2}}$, let $L(p)$ be the link in $S^{3}$ obtained by replacing the $S^{1}\times{S^{2}}$ surgery region by a $p$-fold positive full-twist braid. 

\begin{center}
    \includegraphics[width=0.9\linewidth]{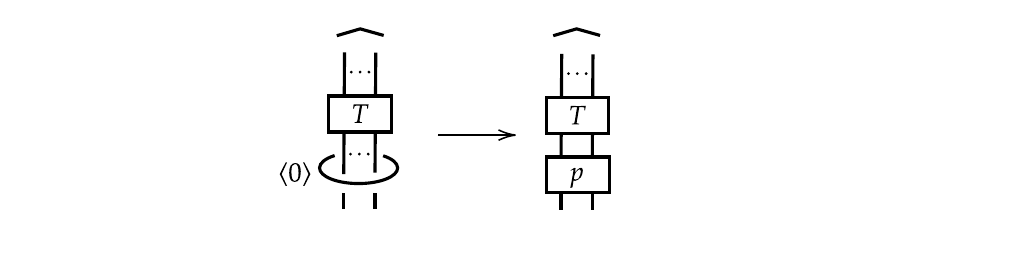}
\end{center}

In the conventions set in \cite[Theorem 6.8]{ROZ-Cat}, the invariant $\Kh^{*,*}(S^{1}\times{S^{2}};L)$ may be approximated in a given homological degree by computing the usual Khovanov homology of $L(p)$ up to a shift. In particular, there is a natural isomorphism
\[\Kh_{RW}^{h,*}(S^{1}\times{S^{2}};L)\cong{\Kh^{h,*}(S^{3};L(p))}
    \]\label{eq:KhFTApprox}
when $h\geq{n_{+}-2p+2}$. If we instead switch to the $\KhR_{2,\diamond}^{-}$ convention, we again have a natural isomorphism as in \ref{eq:KhFTApprox}, but the link $L(p)$ corresponding to $L\subset{S^{1}\times{S^{2}}}$ is obtained by negative full-twists. Under the assumption that $h\geq{n_{+}-2p+2}$, we denote the isomorphism in homological degree $h$ and for a fixed $p$ by:
\begin{equation}
\label{eq:ftpmap}
\KhR_{2,\diamond}^{-,h}(S^{1}\times{S^{2}};L)\xrightarrow{\cong}{\KhR_{2,\diamond}^{-,h}(S^{3};L(-p))}.
\end{equation}

Next, let $L$ be an admissible link in $\#^{n}(S^{1}\times{S^{2}})$, and for $\overrightarrow{p}=(p,...,p)\in{\mathbb{N}^{n}}$, let $C^{\#}(L(\overrightarrow{p}))$ denote the (shifted) complex in the Bar-Natan category obtained by erasing the $S^{1}\times{S^{2}}$ connect-summands and replacing them with a (shifted) Khovanov complexes corresponding to a $p$-fold positive full-twist braid. Willis \cite{WillisS1xS2} provides a full-twist approximation statement for the more general invariant $\Kh(\#^{n}S^{1}\times{S^{2}},L)$. For a fixed homological degree $h$ satisfying 
\begin{equation}\label{eq:twistbound}
    p\geq{\lceil\tfrac{1}{2}(n_{D_{L}}^{+}+1-h)}\rceil
    \end{equation}
where $n_{D_{L}}^{+}$ is the number of positive crossings in an admissible diagram $D_{L}$ of $L$, we have that
\begin{equation}\label{eq:approx}
\Kh^{h,*}(\#^{n}S^{1}\times{S^{2}};L)\cong H^{h}(C^{\#}(L(\overrightarrow{p})).
\end{equation}
\end{remark}

We remark that after fixing a cabling pattern of the link $L\subset{\#^{n}S^{1}\times{S^{2}}}$, one may use full-twist approximation isomorphisms to compute the Rozansky-Willis homology groups that appear in the cabling colimit. For example, if we consider again the case where the 2-handle attaching knot in the Kirby diagram is 0-framed and crossingless and fix a small $h$, then there is a corresponding fixed $p$ satisfying \ref{eq:twistbound} for all diagrams of cables $D_{L(r,r)}$. Let $D_{L(r,r)(\overrightarrow{p})}$ denote the diagram obtained from a diagram of $L(r,r)$ after erasing the $0$-framed unknots corresponding to $S^{1}\times{S^{2}}$ connect summands and replacing them with $p$-full twist braids. There are then isomorphisms of the form
\begin{equation}\label{eq:approxmap}
\phi_{(r,r)}:\KhR_{2}^{h,-}(\#^{n}S^{1}\times{S^{2}};L(r,r))\rightarrow{\KhR_{2}^{h,-}(S^{3};L^{k}(r,r)})
\end{equation}
for each $r$. Note that if either assumption on the attaching knot is dropped then $k$ cannot remain fixed as the number of positive crossings in each cable diagram increases as $r\rightarrow{\infty}$. For the maps $\phi_{(r,r)}$ to assemble to an isomorphism between 1-dimensional inputs skein lasagna modules, one would need establish an analogue of \ref{eq:approx}  for projectors in the $\mathfrak{gl}_{2}$ webs and foams formalism, as well as verify cobordism naturality for the maps \ref{eq:approxmap}. We leave this for future work (See Questions \ref{q:surfacexdh=0},\ref{Q:gl2roz}).

\bibliographystyle{alpha}
\bibliography{refs}

@article {MN22,
    AUTHOR = {Manolescu, Ciprian and Neithalath, Ikshu},
     TITLE = {Skein lasagna modules for 2-handlebodies},
   JOURNAL = {J. Reine Angew. Math.},
  FJOURNAL = {Journal f\"{u}r die Reine und Angewandte Mathematik. [Crelle's
              Journal]},
    VOLUME = {788},
      YEAR = {2022},
     PAGES = {37--76},
      ISSN = {0075-4102},
   MRCLASS = {57K18},
  MRNUMBER = {4445546},
MRREVIEWER = {Daniel V. Mathews},
       DOI = {10.1515/crelle-2022-0021},
       URL = {https://doi.org/10.1515/crelle-2022-0021},
}

@article{MSzT96,
  title={{A product formula for the Seiberg-Witten invariants and the generalized Thom conjecture}},
  author={Morgan, John W and Szab{\'o}, Zolt{\'a}n and Taubes, Clifford Henry},
  journal={Journal of Differential Geometry},
  volume={44},
  number={4},
  pages={706--788},
  year={1996},
  publisher={Lehigh University}
}

@article {MWW-lasagna,
    AUTHOR = {Morrison, Scott and Walker, Kevin and Wedrich, Paul},
     TITLE = {Invariants of 4-manifolds from {K}hovanov-{R}ozansky link
              homology},
   JOURNAL = {Geom. Topol.},
  FJOURNAL = {Geometry \& Topology},
    VOLUME = {26},
      YEAR = {2022},
    NUMBER = {8},
     PAGES = {3367--3420},
      ISSN = {1465-3060},
   MRCLASS = {57K18 (57R56)},
  MRNUMBER = {4562565},
       DOI = {10.2140/gt.2022.26.3367},
       URL = {https://doi.org/10.2140/gt.2022.26.3367},
}

@article {GLW-schur-weyl,
    AUTHOR = {Grigsby, J. Elisenda and Licata, Anthony M. and Wehrli,
              Stephan M.},
     TITLE = {Annular {K}hovanov homology and knotted {S}chur-{W}eyl
              representations},
   JOURNAL = {Compos. Math.},
  FJOURNAL = {Compositio Mathematica},
    VOLUME = {154},
      YEAR = {2018},
    NUMBER = {3},
     PAGES = {459--502},
      ISSN = {0010-437X},
   MRCLASS = {57M27 (17B37 57M25 81R50)},
  MRNUMBER = {3731256},
MRREVIEWER = {Matthew Stoffregen},
       DOI = {10.1112/S0010437X17007540},
       URL = {https://doi.org/10.1112/S0010437X17007540},
}

@article{HRW25,
  title={{A Kirby color for Khovanov homology}},
  author={Hogancamp, Matthew and Rose, David\_E V and Wedrich, Paul},
  journal={Journal of the European Mathematical Society},
  year={2025},
  publisher={EMS Press}
}

@article {WillisS1xS2,
    AUTHOR = {Willis, Michael},
     TITLE = {Khovanov homology for links in {$\#^r(S^2\times S^1)$}},
   JOURNAL = {Michigan Math. J.},
  FJOURNAL = {Michigan Mathematical Journal},
    VOLUME = {70},
      YEAR = {2021},
    NUMBER = {4},
     PAGES = {675--748},
      ISSN = {0026-2285,1945-2365},
   MRCLASS = {57K18},
  MRNUMBER = {4332675},
MRREVIEWER = {Daniel\ V.\ Mathews},
       DOI = {10.1307/mmj/1594281620},
       URL = {https://doi.org/10.1307/mmj/1594281620},
}

@article {MWWhandles,
    AUTHOR = {Manolescu, Ciprian and Walker, Kevin and Wedrich, Paul},
     TITLE = {Skein lasagna modules and handle decompositions},
   JOURNAL = {Adv. Math.},
  FJOURNAL = {Advances in Mathematics},
    VOLUME = {425},
      YEAR = {2023},
     PAGES = {Paper No. 109071, 40},
      ISSN = {0001-8708,1090-2082},
   MRCLASS = {57K41},
  MRNUMBER = {4589588},
       DOI = {10.1016/j.aim.2023.109071},
       URL = {https://doi.org/10.1016/j.aim.2023.109071},
}

@article{CHE-Floer,
  title={{F}loer lasagna modules from link {F}loer homology},
  author={Chen, Daren},
  journal={arXiv preprint arXiv:2203.07650},
  year={2022}
}

@article{ROZ-Cat,
  title={{A categorification of the stable $SU(2)$ Witten-Reshetikhin-Turaev invariant of links in $S^{2}\times S^{1}$}},
  author={Rozansky, Lev},
  journal={arXiv preprint arXiv:1011.1958},
  year={2010}
}

@article {BHPW23,
    AUTHOR = {Beliakova, Anna and Hogancamp, Matthew and Putyra, Krzysztof
              K. and Wehrli, Stephan M.},
     TITLE = {On the functoriality of {$\mathfrak{sl}_2$} tangle homology},
   JOURNAL = {Algebr. Geom. Topol.},
  FJOURNAL = {Algebraic \& Geometric Topology},
    VOLUME = {23},
      YEAR = {2023},
    NUMBER = {3},
     PAGES = {1303--1361},
      ISSN = {1472-2747,1472-2739},
   MRCLASS = {57K18 (18N25)},
  MRNUMBER = {4598808},
       DOI = {10.2140/agt.2023.23.1303},
       URL = {https://doi.org/10.2140/agt.2023.23.1303},
}

@article {lee-endo,
    AUTHOR = {Lee, Eun Soo},
     TITLE = {An endomorphism of the {K}hovanov invariant},
   JOURNAL = {Adv. Math.},
  FJOURNAL = {Advances in Mathematics},
    VOLUME = {197},
      YEAR = {2005},
    NUMBER = {2},
     PAGES = {554--586},
      ISSN = {0001-8708,1090-2082},
   MRCLASS = {57M27},
  MRNUMBER = {2173845},
MRREVIEWER = {Paola\ Cristofori},
       DOI = {10.1016/j.aim.2004.10.015},
       URL = {https://doi.org/10.1016/j.aim.2004.10.015},
}

@article {Kho02,
    AUTHOR = {Khovanov, Mikhail},
     TITLE = {A functor-valued invariant of tangles},
   JOURNAL = {Algebr. Geom. Topol.},
  FJOURNAL = {Algebraic \& Geometric Topology},
    VOLUME = {2},
      YEAR = {2002},
     PAGES = {665--741},
      ISSN = {1472-2747,1472-2739},
   MRCLASS = {57M27 (57R56)},
  MRNUMBER = {1928174},
MRREVIEWER = {Jacob\ Andrew\ Rasmussen},
       DOI = {10.2140/agt.2002.2.665},
       URL = {https://doi.org/10.2140/agt.2002.2.665},
}

@article{OSz00,
  title={{The symplectic Thom conjecture}},
  author={Ozsv{\'a}th, Peter and Szab{\'o}, Zolt{\'a}n},
  journal={Annals of Mathematics},
  pages={93--124},
  year={2000},
  publisher={JSTOR}
}

@article{RW24,
  title={Khovanov homology and exotic $4 $-manifolds},
  author={Ren, Qiuyu and Willis, Michael},
  journal={arXiv preprint arXiv:2402.10452},
  year={2024}
}

@article{SZ2024,
  title={{Kirby belts, categorified projectors, and the skein lasagna module of $S^{2}\times S^{2}$}},
  author={Sullivan, Ian A and Zhang, Melissa},
  journal={arXiv preprint arXiv:2402.01081},
  volume={1},
  year={2024}
}

@misc{rasinv-genus,
      title={Khovanov homology and the slice genus}, 
      author={Jacob A. Rasmussen},
      year={2004},
      eprint={math/0402131},
      archivePrefix={arXiv},
      primaryClass={math.GT},
      url={https://arxiv.org/abs/math/0402131}, 
}

@article{AGL25,
  title={{Colored knot Floer homology: structures and examples}},
  author={Alishahi, Akram and Gorsky, Eugene and Liu, Beibei},
  journal={arXiv preprint arXiv:2508.21776},
  year={2025}
}

@article{Kho04,
  title={{Crossingless matchings and the cohomology of $(n,n)$ Springer varieties}},
  author={Khovanov, Mikhail},
  journal={Communications in Contemporary Mathematics},
  volume={6},
  number={04},
  pages={561--577},
  year={2004},
  publisher={World Scientific}
}

@misc{MMSW22,
      title={{A generalization of Rasmussen's invariant, with applications to surfaces in some four-manifolds}}, 
      author={Ciprian Manolescu and Marco Marengon and Sucharit Sarkar and Michael Willis},
      year={2022},
      eprint={1910.08195},
      archivePrefix={arXiv},
      primaryClass={math.GT},
      url={https://arxiv.org/abs/1910.08195}, 
}

@article{RSWWZ25,
  title={Khovanov skein lasagna modules with $1 $-dimensional inputs},
  author={Ren, Qiuyu and Sullivan, Ian and Wedrich, Paul and Willis, Michael and Zhang, Melissa},
  journal={arXiv preprint arXiv:2510.05273},
  year={2025}
}

@article{hogancamp2020constructing,
  title={Constructing categorical idempotents},
  author={Hogancamp, Matthew},
  journal={arXiv preprint arXiv:2002.08905},
  year={2020}
}

@misc{EH17,
      title={Categorical diagonalization}, 
      author={Ben Elias and Matthew Hogancamp},
      year={2017},
      eprint={1707.04349},
      archivePrefix={arXiv},
      primaryClass={math.RT},
      url={https://arxiv.org/abs/1707.04349}, 
}

@article{BKL25,
  title={Cornered skein lasagna theory},
  author={Blackwell, Sarah and Krushkal, Vyacheslav and Luo, Yangxiao},
  journal={arXiv preprint arXiv:2512.05861},
  year={2025}
}

@article{MSW26,
  title={4-dimensional Skein modules, Handle attachments, and Tangles},
  author={Martin, Gage and Stelow, Mary and Wattal, Mira},
  journal={arXiv preprint arXiv:2602.17825},
  year={2026}
}

@book{Lau74,
  title={Topologie de la dimension trois: homotopie et isotopie},
  author={Laudenbach, Fran{\c{c}}ois},
  volume={12},
  year={1974},
  publisher={{Soci{\'e}t{\'e} math{\'e}matique de France}}
}

@article{Nahm25,
  title={{Khovanov homology can distinguish exotic Mazur manifolds}},
  author={Nahm, Gheehyun},
  journal={arXiv preprint arXiv:2510.10809},
  year={2025}
}

@article{Sul25,
  title={{Bar-Natan skein lasagna modules and exotic surfaces in 4-manifolds}},
  author={Sullivan, Ian A},
  journal={arXiv preprint arXiv:2504.03968},
  year={2025}
}

\end{document}